\documentclass[reqno,11pt]{amsart}

\usepackage{amscd}
\usepackage{amsmath}
\usepackage{amssymb}
\usepackage[american]{babel}
\usepackage{bbm}
\usepackage{bookmark}
\usepackage{cmap} 
\usepackage{dsfont}
\usepackage{enumerate}
\usepackage{epigraph}
\usepackage[mathscr]{euscript}
\usepackage[myheadings]{fullpage}
\usepackage{graphicx}
\usepackage[geometry]{ifsym}
\usepackage{mathabx}
\usepackage{mathrsfs}
\usepackage{mathtools}
\usepackage{refcount}
\usepackage{setspace}
\usepackage{stmaryrd}
\usepackage{thinsp}
\usepackage[safe]{tipa}
\usepackage{verbatim}
\usepackage[all,2cell]{xy} \UseTwocells
\usepackage{xr}
\usepackage[multiple]{footmisc}
\usepackage{hyperref}

\let\underbrace\LaTeXunderbrace

\newcommand{\e}{\mathds{1}}

\numberwithin{equation}{subsection}

\newtheorem{thm}{Theorem}[subsection]
\newtheorem*{thm*}{Theorem}
\newtheorem*{mainthm}{Main Theorem}
\newtheorem{cor}[thm]{Corollary}
\newtheorem*{cor*}{Corollary}
\newtheorem{lem}[thm]{Lemma}

\newtheorem{prop}[thm]{Proposition}
\newtheorem{prop-const}[thm]{Proposition-Construction}

\newtheorem{conjecture}{Conjecture}
\newtheorem*{conjecture*}{Conjecture}

\newtheorem*{princ*}{Principle}

\theoremstyle{remark}
\newtheorem{rem}[thm]{Remark}
\newtheorem{example}[thm]{Example}

\newtheorem{counterexample}[thm]{Counterexample}
\newtheorem{defin}[thm]{Definition}

\newtheorem{warning}[thm]{Warning}
\newtheorem{variant}[thm]{Variant}
\newtheorem{question}[thm]{Question}
\newtheorem{construction}[thm]{Construction}

\newcounter{steps}[thm]
\newtheorem{step}[steps]{Step}

\newcommand{\xar}[1]{\xrightarrow{#1}}
\newcommand{\rar}[1]{\xar{#1}}
\newcommand{\isom}{\rar{\simeq}}

\newcommand{\into}{\hookrightarrow}
\newcommand{\onto}{\twoheadrightarrow}

\newcommand{\La}{\Lambda}

\newcommand{\Loc}{\on{Loc}}

\newcommand{\bDelta}{\mathbf{\Delta}}

\newcommand{\bA}{{\mathbb A}}
\newcommand{\bB}{{\mathbb B}}

\newcommand{\bG}{{\mathbb G}}

\newcommand{\bQ}{{\mathbb Q}}

\newcommand{\bZ}{{\mathbb Z}}

\newcommand{\cD}{{\mathcal D}}

\newcommand{\cG}{{\mathcal G}}

\newcommand{\cK}{{\mathcal K}}

\newcommand{\cP}{{\mathcal P}}

\newcommand{\cS}{{\mathcal S}}

\newcommand{\cY}{{\mathcal Y}}
\newcommand{\cZ}{{\mathcal Z}}

\newcommand{\sA}{{\EuScript A}}

\newcommand{\sC}{{\EuScript C}}
\newcommand{\sD}{{\EuScript D}}
\newcommand{\sE}{{\EuScript E}}
\newcommand{\sF}{{\EuScript F}}
\newcommand{\sG}{{\EuScript G}}

\newcommand{\sK}{{\EuScript K}}

\newcommand{\sM}{{\EuScript M}}
\newcommand{\sN}{{\EuScript N}}
\newcommand{\sO}{{\EuScript O}}

\newcommand{\sX}{{\EuScript X}}
\newcommand{\sY}{{\EuScript Y}}
\newcommand{\sZ}{{\EuScript Z}}

\newcommand{\fL}{{\mathfrak L}}

\newcommand{\fg}{{\mathfrak g}}

\newcommand{\fl}{{\mathfrak l}}

\newcommand{\fp}{{\mathfrak p}}

\newcommand{\fs}{{\mathfrak s}}

\newcommand{\on}{\operatorname}
\newcommand{\ol}{\overline}
\newcommand{\ul}{\underline}

\newcommand{\mathendash}{\text{\textendash}}

\newcommand{\ldotsplus}{\mathinner{\ldotp\ldotp\ldotp\ldotp}}

\newcommand{\Ker}{\on{Ker}}
\newcommand{\Coker}{\on{Coker}}
\newcommand{\End}{\on{End}}
\newcommand{\Hom}{\on{Hom}}

\newcommand{\Aut}{\on{Aut}}
\newcommand{\Spec}{\on{Spec}}

\newcommand{\id}{\on{id}}
\newcommand{\Id}{\on{Id}}

\newcommand{\Ad}{\on{Ad}}

\newcommand{\ind}{\on{Ind}}

\newcommand{\Rep}{\mathsf{Rep}}

\newcommand{\act}{\on{act}}

\newcommand{\codim}{\on{codim}}

\renewcommand{\dot}{\bullet}

\renewcommand{\Im}{\on{Image}}

\newcommand{\Tor}{\on{Tor}}

\newcommand{\vph}{\varphi}
\newcommand{\vareps}{\varepsilon}
\newcommand{\varga}{
    \mathchoice
        {\hbox{\normalsize\textipa{7}}}
        {\hbox{\normalsize\textipa{7}}}
        {\hbox{\scriptsize\textipa{7}}}
        {\hbox{\tiny\textipa{7}}}
}
\newcommand{\vG}{\varGamma}

\newcommand{\Vect}{\mathsf{Vect}}

\newcommand{\Res}{\on{Res}}
\newcommand{\Gr}{\on{Gr}}

\newcommand{\Sym}{\on{Sym}}

\newcommand{\LocSys}{\on{LocSys}}

\newcommand{\Stab}{\on{Stab}}
\renewcommand{\mod}{\mathendash\mathsf{mod}}

\newcommand{\sinf}{\!\frac{\infty}{2}} 

\newcommand{\colim}{\on{colim}}

\newcommand{\DGCat}{\mathsf{DGCat}}
\newcommand{\ShvCat}{\mathsf{ShvCat}}

\newcommand{\Gpd}{\mathsf{Gpd}}

\renewcommand{\lim}{\on{lim}}
\newcommand{\Ind}{\mathsf{Ind}}
\newcommand{\Pro}{\mathsf{Pro}}

\newcommand{\TwoHom}{\mathsf{Hom}}

\newcommand{\Tot}{\on{Tot}}
\newcommand{\heart}{\heartsuit}

\newcommand{\Oblv}{\on{Oblv}}
\newcommand{\Av}{\on{Av}}
\newcommand{\Nm}{\on{Nm}}

\newcommand{\ld}{\check}

\newcommand{\Perf}{\mathsf{Perf}}

\newcommand{\IndCoh}{\mathsf{IndCoh}}
\newcommand{\QCoh}{\mathsf{QCoh}}

\newcommand{\AffSch}{\mathsf{AffSch}}
\newcommand{\PreStk}{\mathsf{PreStk}}

\newcommand{\Alg}{\mathsf{Alg}}

\newcommand{\Lie}{\on{Lie}}

\newcommand{\Tate}{\mathsf{Tate}}
\renewcommand{\Gauge}{\on{Gauge}}

\renewcommand{\ast}{\varoast}

\renewcommand{\o}[1]{\overset{\circ}{#1}}

\renewcommand{\subset}{\subseteq}

\newcommand{\gl}{\fg\fl}
\renewcommand{\sl}{\fs\fl}

\makeatletter
\newcommand{\biggg}{\bBigg@{4}}
\newcommand{\Biggg}{\bBigg@{5}}
\makeatother

\date{\today}

\begin{document}

\frenchspacing

\setlength{\epigraphwidth}{0.9\textwidth}
\renewcommand{\epigraphsize}{\footnotesize}

\title{On the notion of spectral decomposition in local geometric Langlands}

\author{Sam Raskin}

\dedicatory{For Sasha Beilinson}

\address{Massachusetts Institute of Technology, 
77 Massachusetts Avenue, Cambridge, MA 02139.} 
\email{sraskin@mit.edu}

\begin{abstract}

The geometric Langlands program is distinguished in assigning
spectral decompositions to all representations,
not only the irreducible ones.
However, it is not even clear what is meant by a spectral decomposition
when one works with non-abelian reductive groups and with 
ramification.
The present work is meant to clarify the situation, showing
that the two obvious candidates for such a notion 
coincide. 

More broadly, we study the moduli space of (possibly irregular) de Rham
local systems from the perspective of homological algebra.
We show that, in spite of its infinite-dimensional nature, this
moduli space shares many of the nice features of
an Artin stack. Less broadly, these results support the belief in
the existence of a version of geometric Langlands allowing
arbitrary ramification.

Along the way, we give some apparently new, if unsurprising,
results about the algebraic geometry of 
the moduli space of connections, using Babbitt-Varadarajan's
reduction theory for differential equations.

\end{abstract}

\maketitle

\setcounter{tocdepth}{1}
\tableofcontents


\section{Introduction}


\subsection{Local geometric Langlands}

Let $k$ be a ground field of characteristic zero, let $K \coloneqq k((t))$
for $t$ an indeterminate. Let $\o{\cD} \coloneqq \Spec(K)$ be the (``formal") punctured disc.
Let $G$ be a split reductive group over $k$, and let 
$\ld{G}$ be a split reductive group over $k$ with root datum dual 
to that of $G$.

Recall the format of local geometric Langlands from \cite{fg2}: there it is suggested
that, roughly speaking, DG categories\footnote{Throughout this introduction,
\emph{DG category} means \emph{presentable} (i.e., cocomplete plus a set theoretic condition) DG category.}
acted on (strongly) by the loop group
$G(K)$ should be equivalent to DG categories over the moduli space $\LocSys_{\ld{G}}(\o{\cD})$ 
of \emph{de Rham local Langlands parameters}, i.e., $\ld{G}$-bundles on $\o{\cD}$ with connection. 
Note that we work allow irregular singularities in our local systems, which by an old analogy is
parallel to wild ramification in the arithmetic theory. 

\subsection{}

Some remarks are in order.

\begin{itemize}

\item For the reader who is not completely comfortable with the translation
from the arithmetic theory: 

A DG category acted on by $G(K)$ should be thought
of as analogous to a smooth representation. The basic example is
the category of $D$-modules on a space acted on by $G(K)$ (e.g., $D$-modules
on the affine Grassmannian). 

One can then think of a category over\footnote{This paper
is an attempt to understand the phrase
``category over," so the reader should not think us too remiss in 
not yet explaining at a technical level what is meant here. 
However, we ask the reader to suspend disbelief for the moment, 
and to imagine something
like a DG category with a spectral decomposition over 
$\LocSys_{\ld{G}}(\o{\cD})$
(e.g., one should have the ability to take fibers at local systems).} 
$\LocSys_{\ld{G}}(\o{\cD})$ as something
richer than a measure on the space of Langlands parameters, and which is
precisely measuring the spectral decomposition of the corresponding 
smooth representation.

\vspace{10pt}

\item This idea is quite appealing. 
One of the distinguishing features of geometric Langlands, in contrast to the arithmetic theory,
is the existence of geometric structures on the sets of spectral parameters, and the suggestion 
(c.f. \cite{hitchin}, \cite{dennis-laumonconf})
that pointwise spectral descriptions should generalize over these moduli spaces. 
The use of spectral decompositions in families then allows to remove irreducibility hypotheses
from the analogous arithmetic questions.

\vspace{10pt}

\item Such an equivalence is not expected to hold literally as is, morally because of the existence
of Arthur parameters in the arithmetic spectral theory. 

However, some form of this conjecture may be true
as is if we take \emph{tempered} categories acted on by $G(K)$, c.f. \cite{arinkin-gaitsgory}.
This notion warrants further study: one can make a number of 
precise conjectures for which there are not obvious solutions.

\vspace{10pt}

\item There are two main pieces of evidence for believing in such theory.

First, Beilinson long ago observed\footnote{
See e.g. \cite{beilinson-heisenberg} 
Proposition 1.4. It is necessary here to seek out the published 
version of the article and not the preprint.} 
that Contou-Carr\`ere's construction of Cartier self-duality of 
$\bG_m(K)$ implies a precise form of the conjecture for a 
torus.\footnote{More precisely,
one should use Gaitsgory's notion of (and results on) 
1-affineness \cite{shvcat},
as will be discussed later: see especially \S \ref{ss:locsysgm-1aff}.} 

Second, (derived) geometric Satake 
(c.f. \cite{mirkovic-vilonen}, \cite{fgv}, \cite{bezrukavnikov-finkelberg}
and \cite{arinkin-gaitsgory}) and Bezrukavnikov's geometric affine Hecke
theory (c.f. \cite{arkhipov-bezrukavnikov}, \cite{bez-main}) fit
elegantly into this framework. They completely settle the questions 
over (the formal completion of) the locus of unramified connections 
and regular connections with unipotent monodromy respectively
(and are instructive about how to understand the temperedness
issues).

\vspace{10pt}

\item Local geometric Langlands is supposed to satisfy various
compatibilities as in the arithmetic theory, e.g. there should be 
a compatibilities with parabolic induction and with Whittaker models.

Much\footnote{This paragraph is included for the sake of academic
completeness, but we explicitly note that it plays no role in what follows.
The reader unfamiliar with this idea does not need to follow it,
and does not need to turn to \cite{hitchin} and \cite{fg2} to catch up.}
more intriguing is the compatibility with Kac-Moody representations
at the critical level: this idea is implicit in \cite{hitchin}, 
and is explicitly proposed in \cite{fg2}. There is nothing\footnote{Perhaps
this is too strong a claim: as is well-known among
experts, the appearance of opers can be understood
as a spectral analogue of the Whittaker model, which normally only
appears on the geometric side. Nevertheless, such a thing is impossible
arithmetically.} analogous in the arithmetic theory, so this marks
one of the major points of departure of geometric Langlands
from the usual theory of automorphic forms. 
Moreover, such a compatibility
has deep impliciations in representation theory, which are of independent
interest.

\end{itemize}

\subsection{Local geometric Langlands?}

Though there are exceptions in special cases, 
local geometric Langlands (for non-abelian $G$) 
essentially stalled out after Bezrukavnikov's theory.

There may be many reasons for this: Bezrukavnikov's theory
invokes many brilliant constructions, and perhaps no one has figured out
how to extend these. But more fundamentally, there are 
serious technical challenges when one reaches beyond Iwahori.

On the geometric side, for compact open subgroups 
$\cK \subset G(K)$ smaller than Iwahori, the $\cK$-orbits
on $G(K)/\cK$ are not ``combinatorial," i.e., the orbits
are not discretely parameterized. For starters,
this means one must handle non-holonomic $D$-modules.
More seriously, this transition abandons the comfort zone of 
classical geometric representation theory.

On the spectral side, recall that
the formal completion of the locus of regular singular (de Rham) local
systems with unipotent monodromy is isomorphic to the
formal completion of $\ld{\sN}/\ld{G} \subset \ld{\fg}/\ld{G}$ 
for $\ld{\sN} \subset \ld{\fg}$ the nilpotent cone. 
However, beyond this locus, $\LocSys_{\ld{G}}(\o{\cD})$ 
is no longer an Artin stack, and is not 
even finite type. So again, we find ourselves out our 
comfort zone.

This is to say the skeptic would not be rash in asking
if there is any true local geometric Langlands (which for starters
would incorporate irregular singularities);
nor to suggest that Iwahori ramification is the limit of
what is apparently just a 
geometric shadow of a much richer arithmetic theory. 
If our interlocutor is right, then there is not much more to explore in geometric Langlands,
and the subject is nearing exhaustion.

\subsection{}

But we remember Beilinson's class field theory for de Rham local 
systems and draw some resilience from it. 

\subsection{}

This paper is essentially an attempt to study how bad the geometry on the spectral side of
local geometric Langlands is.

The results are surprisingly positive: from the perspective of
homological algebra, $\LocSys_{\ld{G}}(\o{\cD})$
has all the favorable features of an Artin stack.

In short summary, our main result compares two (a priori quite different, c.f. below) notions of
category over $\LocSys_{\ld{G}}(\o{\cD})$, showing that
the two coincide. This result therefore answers the longstanding
question:
what is a category over $\LocSys_{\ld{G}}(\o{\cD})$? I.e., what
is the spectral side of (tempered) local geometric Langlands?

As a happy output of our methods, we also show 
that $\QCoh(\LocSys_{\ld{G}}(\o{\cD}))$ is compactly generated as a DG category,
which comes as a bit of a surprise.

So, perhaps the situation is not so bad.

\begin{rem}

As we will discuss below, these results can be understood
as a geometric strengthening of the simple observation that
tangent spaces at field-valued points of $\LocSys$ are finite-dimensional.

\end{rem}

\subsection{History}

We were first asked this question by Dennis Gaitsgory in 2010.
He informs us that the ambiguity in what was meant by a category
over $\LocSys_{\ld{G}}(\o{\cD})$ 
was the reason that no formulation of a local
geometric Langlands conjecture was given by him and Frenkel in \cite{fg2},
nor by Beilinson earlier.

\subsection{Changing notation}

At this point, we stop referring to any Langlands duality. For the remainder of the paper, 
let $G$ be an affine algebraic group over $k$, which will play the role that $\ld{G}$ played above. 
In particular, we do not assume that $G$ is reductive
in what follows.

\subsection{Main results}

We now proceed to give a detailed description of the results of this paper.

We proceed in increasing sophistication. 
We begin in \S \ref{ss:geom-start}-\ref{ss:geom-finish}
with the main novel geometric result
of this paper, which is less technical than the rest of the paper
in that it does not involve DG categories. 

In \S \ref{ss:intro-cptgen}, we will comment on the compact generation
of $\QCoh$, indicating why the claim is non-trivial
(i.e., why it encodes at a technical level the idea that
$\LocSys_G(\o{\cD})$ is like an Artin stack).

Finally, in \S \ref{ss:intro-1aff}, we will begin to discuss 1-affineness
and what could be meant by a category over $\LocSys_G(\o{\cD})$.
Recall that our main result here compares two different notions.

\subsection{Geometry of $\LocSys_G(\o{\cD})$}\label{ss:geom-start}

We now provide a more detailed description of the
results of this paper. 

Our treatment will be thorough, since there are not great references
treating the moduli of local systems on the punctured disc as an object
of algebraic geometry (though it is old folklore), and hoping that the
expert reader will forgive this. So we will give a precise definition
of $\LocSys_G(\o{\cD})$ followed by some examples, and then
state our main geometric result.

\subsection{}

First, we need to define $\LocSys_G(\o{\cD})$.

Let $\AffSch$ denote the category of affine schemes.\footnote{Here
these are \emph{classical} affine schemes in the language of
derived algebraic geometry. Fortunately, derived algebraic geometry
plays a minor role in the present work (though it always
lurks in the background). Moreover, manipulations
with connections (which is only possible in classical algebraic geometry)
is crucial in \S \ref{s:locsys}, so our default convention is
the language of classical algebraic geometry.} 
We will presently define a functor:

\[
\begin{gathered}
\AffSch \to 1\mathendash\Gpd \\
S \mapsto \LocSys_G(\o{\cD})(S)
\end{gathered}
\] 

\noindent where $1\mathendash\Gpd$ is the category of 1-groupoids (i.e., the usual notion of groupoid, not a higher groupoid).
This will be the geometric object we mean by $\LocSys_G(\o{\cD})$: in the language \cite{dennis-dag} 
of derived algebraic geometry, we say that $\LocSys_G(\o{\cD})$ is a 1-truncated classical prestack.

\begin{warning}

We emphasize from the onset that
we will \emph{not} sheafify this functor for any topology; though we are trained always to sheafify, not sheafifying 
actually avoids some useless and unnecessary confusion,\footnote{E.g., we never
need to worry about the subtle questions of projectivity of an $A((t))$-module
versus local freeness on $\Spec(A)$: this question is irrelevant for us.
So the cost is some awkwardness, and the benefit is that we never
engage with some subtleties.}
 and by descent does not change the categories we are interested
in (namely, quasi-coherent sheaves and sheaves of categories).

\end{warning}

\subsection{}

Let $\Omega_K^1$ denote the $K$-line of $k$-linear continuous differentials for $K$.
So we have a map $d:K \to \Omega_K^1$, and our coordinate $t$ defines the
basis element $dt \in \Omega_K^1$ so that $df = f' \cdot dt$. For economy, 
we usually use the notation $Kdt$ for $\Omega_K^1$, though we will not be using
our coordinate in any symmetry breaking way.

More generally, for $A$ a $k$-algebra and $V$ an $A((t))$-module,
we let $Vdt \coloneqq \Omega_K^1 \otimes_K V$.  

\subsection{Description of $\LocSys_G(\o{\cD})$ as a quotient}

We recall that there is a \emph{gauge} action of the indscheme $G(K)$ on the indscheme\footnote{We
begin a practice here where sometimes we view $\fg((t))dt$ (or $\fg[[t]]dt$) as a geometric object, i.e.,
an indscheme (or scheme), and sometimes as a linear algebra object, i.e., a $K$-vector space. We promise
always to be careful to distinguish which of the two perspectives we are using.}
$\fg((t))dt$. 

Morally, this is given by the formula:

\[
\begin{gathered}
G(K) \times \fg((t))dt \to \fg((t)) \\
(g,\Gamma dt)\mapsto \Gauge_g(\Gamma dt) \coloneqq \Ad_g(\Gamma) dt - 
(dg) g^{-1}
\end{gathered}
\]

\noindent though we will give a precise construction in what follows. Modulo the construction
of this action, we define $\LocSys_G(\o{\cD})$ as the prestack (i.e., non-sheafified) 
quotient $\fg((t))dt/G(K)$, where we are using the gauge action.

One approach to constructing this action is to notice that for $G = GL_n$, this formula makes sense
as is, and then to use the Tannakian formalism to reduce to this case. Indeed, the only difficulty in making
sense of this formula is the term $dg \cdot g^{-1}$, and this case a clear meaning for matrices.

Alternatively, recall that $G$ carries the canonical $\fg$-valued 1-form (the Cartan form). This
is a right invariant $\fg$-valued 1-form on $G$. Note that right 
invariant $\fg$-valued 1-forms
are the same as vectors in $\fg \otimes \fg^{\vee}$, and our 1-form then corresponds to the
identity matrix in $\fg \otimes \fg^{\vee} = \End(\fg)$. Note that for $G = GL_n$, the Cartan form is 
given by the formula $dg\cdot g^{-1}$. 

For $S = \Spec(A) \in \AffSch$, let $\o{\cD}_S \coloneqq \Spec(A((t)))$.
Given $g: \o{\cD}_S \to G$ (i.e., an $S$-point of $G(K)$), one can pullback\footnote{
For precision, since $A((t))dt$ is a bit outside the usual format of differential algebraic geometry:
the composition $\sO_G \xar{g^*} A((t)) \xar{d} A((t))dt$ is a derivation, so induces a map
$\Omega_G^1 \to A((t))dt$. Similarly, we obtain $\fg \otimes \Omega_G^1 \to (\fg\otimes A)((t))dt$,
and our form is the image of the Cartan form under this map.}
 the Cartan form to obtain an $\fg$-valued $S$-relative differential form on $\o{\cD}_S$, i.e.,
an element of $(\fg\otimes A)((t))dt = (\fg((t))dt) (S)$.

\subsection{Description of $\LocSys_G(\o{\cD})$ via local systems}

We now give a somewhat more concrete description of $\LocSys_G(\o{\cD})$,
especially for $G = GL_n$.

\begin{defin}

A \emph{differential module} on $\o{\cD}_S$ is a finite rank free $A((t))$-module
$V$ equipped with an $A$-linear map:

\[
\nabla: V \to Vdt 
\] 

\noindent satisfying the Leibniz rule:

\[
\nabla(fv) = f\nabla(v) + df \otimes v.
\]

\end{defin}

For $G = GL_n$, we define $\LocSys_{GL_n}(\o{\cD})(S)$ as the
groupoid of differential modules on $\o{\cD}_S$ of rank $n$ (as an $A((t))$-module). Indeed, for an $(n\times n)$-matrix 
$\Gamma dt \in \fg((t))$, $\nabla \coloneqq d + \Gamma dt$
defines a connection on $A((t))^{\oplus n}$, and it is standard
to see that this gives an equivalence of groupoids.

More generally, we obtain a Tannakian description of $\LocSys_G(\o{\cD})$. 
We should be careful not to sheafify in $S$ though! So we take the groupoid
of symmetric monoidal functors from finite-dimensional representations of $G$ 
to differential modules on $\o{\cD}_S$, such that the resulting
functor from finite-dimensional representations of $G$ to free 
$A((t))$-modules
admits an isomorphism with $\Oblv \otimes_k A((t))$, for $\Oblv$ the forgetful functor
for $G$-representations.

\subsection{Example: $G = \bG_m$}\label{ss:gm}

We now give some important explicit examples of $\LocSys_G(\o{\cD})$,
which the reader should always keep in mind.
We begin with the case $G = \bG_m$.

Since $G$ is commutative, the gauge action is
given by $(g,\omega) \in \bG_m(K)\times Kdt \mapsto \omega - d\log(g)$.

We will compute the quotient by the gauge action in stages,
ultimately showing that $\LocSys_{\bG_m}(\o{\cD})$ is 
a product of three terms: 
$\bG_a/\bZ$ where $\bZ$ acts by translations,
the quotient of an ind-infinite dimensional affine space
by its underlying formal group (i.e., $\colim_n \bA_{dR}^n$) ,
and $\bB \bG_m$ the classifying stack of $\bG_m$.

First, we quotient by the first congruence subgroup
$\cK_1 \coloneqq \Ker(\bG_m(O) \to \bG_m)$
of $\bG_m(K)$ (here and everywhere, $O \coloneqq k[[t]]$).

Since $\cK_1$ is prounipotent, $g \in \cK_1$ can 
be written canonically as $\exp(\xi)$ for $\xi \in t k[[t]]$.
Here we are abusing notation: by (e.g.) $g \in \cK_1$, we implicitly
are taking an $S$-point (for $S$ a test affine scheme), and
we are considering $t k[[t]]$ as a scheme in the obvious
way. That aside, the action of $\cK_1$ then sends 
$(g = \exp(\xi), \omega)$ to $\omega - d\xi$. Therefore,
we see that $Kdt/\cK_1 = K dt/Odt$, since
$d: t k[[t]]\isom Odt$. Note that $Kdt/Odt$ is an ind-affine space,
and is ind-finite type.

Quotienting by $\bG_m(O)$, 
we see that $\bG_m = \bG_m(O)/\cK_1$ acts on
$Kdt/Odt$ trivially, and therefore the quotient is
$Kdt/Odt \times \bB \bG_m$.

Next, we quotient by $\widehat{\bG_m(O)}$, 
the formal completion of $\bG_m(O)$ in $\bG_m(K)$,
i.e., the connected component of the
identity in $\bG_m(K)$. The exponential map defines
a canonical isomorphism $(K/O)_0^{\wedge} \isom \widehat{\bG_m(O)}/\bG_m(O)$,
where $(K/O)_0^{\wedge}$ is the formal completion at the identity of the 
group indscheme $K/O$ (considered as a group under addition).
Moreover, the resulting action on $Kdt/Odt \times \bB \bG_m$
is trivial on the second factor, and on the first factor
(up to sign) is induced from the translation action of $Kdt/Odt$ on itself
and the homomorphism $K/O \xar{d} Kdt/Odt$. Since the
latter homomorphism identifies the source with the kernel
of the residue map, the resulting quotient is:

\[
\bG_a\cdot \frac{dt}{t} \times 
\Ker(\Res:Kdt/Odt \to \bG_a)_{dR} \times \bB \bG_m.
\]

\noindent Here we are using the well-known 
de Rham prestack construction, and the fact that de Rham of a group
is the quotient of the group by its underlying formal group.

Finally, we quotient by $\bG_m(K)/\widehat{\bG_m(O)} = \bZ$. It is easy
to see that the generator acts as translation by $\frac{dt}{t}$ on the
factor $\bG_a \frac{dt}{t}$ and trivially on the other factors above,
so the resulting quotient is:

\[
(\bG_a\cdot \frac{dt}{t})/(\bZ\cdot \frac{dt}{t}) \times 
\Ker(\Res:Kdt/Odt \to \bG_a)_{dR} \times \bB \bG_m
\]

\noindent as originally claimed.

\subsection{Example: $G$ is unipotent}\label{ss:ex-unipotent}

Next, we treat the case where $G$ is unipotent.
We will show that $\LocSys_G(\o{\cD})$ is isomorphic 
to the stack $\fg/G$, where $G$ acts via the adjoint action.
More precisely, we claim that the map:

\[
\begin{gathered}
\fg/G \to \LocSys_G(\o{\cD}) \\
\xi \mapsto d + \xi \frac{dt}{t}
\end{gathered}
\]

\noindent gives an isomorphism.

First, suppose that $G$ is a commutative unipotent group.
We have an exponential isomorphism
$\exp: \fg((t)) \isom G(K)$, and by commutativity, the gauge action 
then becomes:

\[
\begin{gathered}
\fg((t)) \times \fg((t))dt \isom G(K) \times \fg((t))dt \to \fg((t))dt \\
(\xi, \Gamma dt) \mapsto \Gamma dt - d\xi.
\end{gathered}
\]

\noindent This clearly gives the claim in this case.

In the general case, let $Z$ be the center of $G$.
By commutativity, 
$\LocSys_Z(\o{\cD})$ has commutative group  structure as a prestack.
Moreover, by centrality it acts on 
$\LocSys_G(\o{\cD})$, and the prestack
quotient for this action is $\LocSys_{G/Z}(\o{\cD})$. 
Therefore, we can apply induction on the degree of nilpotence to
reduce to the commutative case.

\subsection{Example: $G = \bG_m \ltimes \bG_a$}\label{ss:gm-ga}

In this case, we will see that $\LocSys_G(\o{\cD})$ is
\emph{not} locally of finite type, unlike the previous examples.
Note that the nature of this example 
forces similar bad behavior for $G = GL_2$,
or any other nonabelian reductive group.

In what follows, we 
realize $G$ as a matrix group in $GL_2$ in the usual way.

Consider $O = \{f = \sum_{i=0}^{\infty} a_it^i\}$ as a scheme,
and let $\bA^1 \times O$ map to $\LocSys_G(\o{\cD})$ though
the map:

\[
(\lambda,f) \in \bA^1 \times O \mapsto 
d+\begin{pmatrix} \frac{\lambda}{t} & f \\ 0 & 0 \end{pmatrix} dt.
\] 

\noindent If $\LocSys_G(\o{\cD})$ were locally of finite type,
then this map would factor through $\bA^1 \times O/t^nO$ for some $n$,
i.e., it would be isomorphic to a map depending only on
the first $n$ of the $a_i$ in the above notation. We claim this is not the
case.

To this end, we claim that the $G$-connection:

\[
d+\begin{pmatrix} -\frac{n}{t} & t^{n-1} \\ 0 & 0 \end{pmatrix} dt 
\]

\noindent is not isomorphic to the trivial $G$-connection for 
$n \in \bZ^{\geq 0}$, while
the connection
$d+\begin{pmatrix} -\frac{n}{t} & 0 \\ 0 & 0 \end{pmatrix} dt $
is isomorphic to a trivial connection. Note that this immediately
gives a contradiction to the locally finite type claim.

To see these results, we apply
a gauge transformation by 
$\begin{pmatrix} t^{-n} & 0 \\ 0 & 1\end{pmatrix} \in G(K)$.
The former connection becomes:

\[
d+\begin{pmatrix} 0 & t^{-1} \\ 0 & 0 \end{pmatrix} dt 
\]

\noindent which is easily seen to be nontrivial (one solution
to the corresponding order 2 differential equation is the logarithm), 
while the the latter connection becomes trivial, as desired.

\begin{rem}

This example also shows that $t^{-1}\fg[[t]]dt / G(O)$ is not
locally of finite type.

\end{rem}

\begin{rem}

This example admits another description. We will be slightly
informal here, since we will not use this language later in the paper.

Recall that to a two step complex of vector spaces (or vector
bundles), we can associate a stack. 
Then
$\LocSys_{\bG_m \ltimes \bG_a}(\o{\cD}) \to \LocSys_{\bG_m}(\o{\cD})$
is obtained by the (Tate analogue of) this construction
for the two step complex on $\LocSys_{\bG_m}(\o{\cD})$ computing
de Rham cohomology of a rank 1 local system.
In these terms, the above corresponds to the fact that
(regular singular) $\bG_m$-connections can have de Rham cocycles with
an arbitrarily high order of zero.

\end{rem}

\subsection{Semi-infinite motivation}

The difficulty in working with $\LocSys_G(\o{\cD})$ as a geometric object
is that it is the quotient of something very infinite-dimensional by
something else very infinite-dimensional. Here ``very infinite-dimensional"
means ``an indscheme of ind-infinite type." So it is not at all an Artin stack
(except for $G$ unipotent).

But $\LocSys_G(\o{\cD})$ \emph{does} share some properties
of an Artin stack. E.g., the cotangent complex is computed
using de Rham cohomology, and therefore has finite-dimensional
cohomology groups at \emph{field-valued} points.

So as a first approximation, one should think 
that the infinite-dimensional forces are canceling
each other out into something almost finite-dimensional. 
We have
seen that this does happen in a fairly precise the special cases where $G$
is unipotent or commutative, though it is harder to say what we mean
by this if $G$ is merely, say, solvable. 

In fact, everything in this paper
marked as a ``theorem" may be regarded as an 
attempt to say in what sense $\LocSys_G(\o{\cD})$ behaves
like an Artin stack (i.e., that it has better geometric properties
than, say, $\fg((t))/G(K)$ under the adjoint action). Moreover,
already for $G = GL_2$, these results are the only such statements
of which I am aware.

\subsection{}\label{ss:geom-finish}

Motivated by the example of $\bG_m \ltimes \bG_a$,
we will show the following results 
in \S \ref{s:locsys}. Here $G$ can be any affine algebraic group.

We begin with the case of regular connections.

\begin{thm*}[Thm. \ref{t:artin} and Ex. \ref{e:rs-evalues}]

Let $\cK_1 \subset G(O)$ denote the first congruence subgroup
(i.e., the kernel of the evaluation map $G(O) \to G$).

For $\Gamma_{-1} \in \fg$ a $k$-point, 
consider the gauge action of $\cK_1$ on the subscheme:

\[
\Gamma_{-1}dt + \fg[[t]]dt \subset \fg((t))dt.
\]

\noindent Then the quotient of this scheme by $\cK_1$ is an
Artin stack smooth over $k$ (in particular, it is of finite type).

\end{thm*}

In other words, when we fix the polar part of the connection,
the quotient is an Artin stack. Note that the main difficulty in proving
this is the finiteness statement.

\begin{rem}

To emphasize, we have a very strange situation.
$\cK_1$ acts on $t^{-1}\fg[[t]]dt$ preserving the fibers
of the residue map $t^{-1}\fg[[t]]dt \to t^{-1}\fg[[t]]dt/\fg[[t]]dt = \fg$;
then the quotient by $\cK_1$ is an Artin stack on geometric fibers,
but is \emph{not} an Artin stack before we take fibers
(since it is not of finite type by \S \ref{ss:gm-ga}).

In more naive terms, for every fiber some congruence
subgroup acts freely, but the subgroup cannot be taken
independently of the fiber (again, by \S \ref{ss:gm-ga}).

This appears to be a new pathology in infinite type algebraic geometry.

\end{rem}

The result above is not quite true as is when we allow higher order poles,
but we will see that the following variant holds:

\begin{thm*}[Thms. \ref{t:artin} and \ref{t:jumps}]

For every $r>0$, there exists an integer $\rho$ (depending also on $G$)
such that for every $\Gamma_{-r},\ldots,\Gamma_{-r+\rho}$
$k$-points of $\fg$,
the quotient of the action of the {$(\rho+1)$-congruence} group 
$\cK_{\rho+1} \coloneqq \Ker(G(O) \to G(O/t^{\rho+1}O))$
on the subscheme of connections of the form:

\[
\Gamma_{-r} t^{-r}dt + \ldots + \Gamma_{-r+\rho} t^{-r+\rho}dt +
t^{-r+\rho+1}\fg[[t]]dt \subset \fg((t))dt
\]

\noindent is an Artin stack smooth over $k$.

\end{thm*}

\begin{rem}

This second result is a special case of the first result:
the only additional claim is that for $r = -1$, we can take $\rho = 0$.
Moreover, as indicated above, this will be clear
from the formulation of Theorem \ref{t:artin} and from
Example \ref{e:rs-evalues}. We emphasize that the only difference
from the previous theorem is that for $r>1$, we typically have
$\rho \geq r$, i.e., we have to fix more than just the polar part
of the connection.

\end{rem}

\begin{rem}

Note that the latter cited theorem, Theorem \ref{t:jumps}, 
is strongly influenced by Babbitt-Varadarajan \cite{babbitt-varadarajan}, 
and closely follows their
method for treating connections on the formal punctured disc.

\end{rem}

\begin{rem}

The proof of Theorem \ref{t:artin}, which says
that an infinitesimal finiteness hypothesis implies quotients
of the above kind are finite type, is surprisingly tricky, especially
considering how coarse the hypothesis and the conclusion 
of this result are. 
I would be very glad to hear a simpler proof.

\end{rem}

\subsection{Compact generation of $\QCoh$}\label{ss:intro-cptgen}

Next, let $\QCoh(\LocSys_G(\o{\cD}))$ denote the 
(cocomplete) DG category of quasi-coherent sheaves 
on $\LocSys_G(\o{\cD})$. Recall that $\QCoh$ is defined
as an appropriate homotopy limit for any prestack, 
and we are simply applying this construction in 
the case of $\LocSys_G(\o{\cD})$. 

\begin{thm*}[Thm. \ref{t:cpt-gen}]

If $G$ is reductive, then 
$\QCoh(\LocSys_G(\o{\cD}))$ is compactly generated.

\end{thm*}

Let us comment on why this result is nontrivial.

In forming $\LocSys_G(\o{\cD})$, we take
a certain quotient by $G(K)$. Since $G$ is reductive,
$G(K)/G(O) = \Gr_G$ is ind-proper, 
and the major difficulty arises in quotienting by $G(O)$.

Indeed, one can easily see that $\QCoh(\bB G(O))$ has
\emph{no} non-zero compact objects. Ultimately, this is
because the trivial representation has infinite cohomological
amplitude, because of the infinite dimensional pro-unipotent
tail of $G(O)$ (more precisely, one should
combine this observation with left completeness of the
canonical $t$-structure on this category). 
In other words, the global sections functor
on $\bB G(O)$ has infinite cohomological amplitude,
ruling out compactness.

We will prove the compact generation
for $\LocSys_G(\o{\cD})$ 
by showing that global sections for $t^{-r}\fg[[t]]dt/G(O)$ 
is cohomologically bounded (for the natural $t$-structure on this quotient).

Here's a sketch of the proof. This
result can be checked after replacing $G(O)$ by the
$\rho$th congruence subgroup $\cK_{\rho}$.
Then the claim follows from a Cousin spectral sequence argument
by noting that the geometric fibers of the map
$t^{-r}\fg[[t]]dt/\cK_{\rho} \to t^{-r}\fg[[t]]dt/t^{-r+\rho}\fg[[t]]dt$
are Artin stacks for $\rho$ large enough (by the
earlier geometric theorems), and the further (easy) observation
that the dimensions of these fibers are uniformly bounded
(in terms of $\rho$ and $G$).

\begin{rem}

This argument illustrates the main new idea of this work:
$\LocSys_G(\o{\cD})$ is nice from the
homological perspective because its
worst pathologies are rooted
in the poor behavior that occurs as we move between the fibers
of the map $t^{-r}\fg[[t]]dt/\cK_{\rho} \to t^{-r}\fg[[t]]dt/t^{-r+\rho}\fg[[t]]dt$
(as has long been known),
and the Cousin filtration means that these pathologies disappear
in the derived category. 

\end{rem}

\begin{question}

The above theorem relies on the properness of $\Gr_G$,
which is why I only know it for $G$ reductive. We know
it for $G$ unipotent by separate means. 
Is $\QCoh(\LocSys_G(\o{\cD}))$ compactly generated
for general $G$? 
Already for $G = \bG_m \ltimes \bG_a$, I do not know the answer.

\end{question}

\subsection{1-affineness}\label{ss:intro-1aff}

We now discuss the notion of \emph{1-affineness} from
\cite{shvcat}, which plays a major role in this text.

\begin{rem}

As some motivation for what follows:
1-affineness appears to play a key technical role in 
this flavor of geometric representation theory. Indeed, I think
I am not overstepping in asserting that every non-trivial 
formal manipulation in the subject is an application of 1-affineness,
or that the theorems on 1-affineness, all of which are contained in \cite{shvcat}, 
are what fundamentally undergirds the 
``functional analysis" of geometric representation theory. 

\end{rem}

\begin{rem}

This is the only review of 1-affineness given in the text,
and it may be slightly too detailed for an introduction.
We apologize to the reader if it seems to be so, and suggest
to skip anything that does not appear to be urgent.

\end{rem}

\subsection{}\label{ss:dgcat-linalg}

First, we briefly need to recall the linear algebra of DG categories.

We always work in the higher categorical framework, so our
default language is that a category is an $(\infty,1)$-category in the
sense of Lurie et al.

By a cocomplete DG category, we will always mean a presentable one,
i.e., a DG category admitting (small) colimits and 
satisfying a set-theoretic condition. The relevant set theory will lie
under the surface in our applications to e.g. the adjoint functor theorem,
and life is better for us all if we suppress it (but do not forget about it)
to the largest extent possible. 

Let $\DGCat_{cont}$ denote the category of cocomplete
(i.e., presentable) DG categories, with morphisms being continuous
(i.e., commuting with filtered colimits) DG functors. Note that
these functors actually commute with all colimits, since DG functors
tautologically commute with finite colimits. 

We let $\Vect$ denote the DG category of (complexes of) vector
spaces. For $A \in \Alg(\Vect)$, we let $A\mod$ denote the
DG category of left $A$-modules. For $A$ connective,
we let $A\mod^{\heart}$ denote the heart of the $t$-structure
on $A\mod$.

Recall that $\DGCat_{cont}$ is equipped with a standard tensor
product $\otimes$ with unit object $\Vect$.

For $\sA \in \Alg(\DGCat_{cont})$, we will often use $\sA\mod$ to
mean $\sA\mod(\DGCat_{cont})$. We sometimes
say that a functor $F:\sC \to \sD$ between objects of
$\sA\mod$ is $\sA$-linear if it is (equipped with a structure of)
morphism in $\sA\mod$: in particular, this means that the
functor $F$ commutes with colimits.

\subsection{}

Suppose that $\sY$ is a prestack in the sense of \cite{dennis-dag}:
note that this is inherently a notion of derived algebraic geometry. 
What should we mean by ``a (DG) category over $\sY$?"

First, if $\sY = \Spec(A)$ is an affine DG scheme, all roads lead
to Rome. I.e., the following structures on $\sC \in \DGCat_{cont}$ 
are equivalent:

\begin{itemize}

\item Functorially making $\Hom$s in $\sC$ into $A$-modules,
i.e., giving a morphism of $\sE_2$-algebras
$A \to Z(\sC)$, where $Z(\sC)$ is the (derived, of course) Bernstein center of $\sC$.

\item Giving $\sC$ the structure of $A\mod$-module in $\DGCat_{cont}$.

\end{itemize}

For a general prestack $\sY$, we have two options.

First, we could ask for a DG category \emph{tensored} over $\sY$,
i.e., an object of 
$\QCoh(\sY)\mod \coloneqq \QCoh(\sY)\mod(\DGCat_{cont})$
(where usual tensor products of quasi-coherent sheaves
makes $\QCoh(\sY)$ into a commutative algebra object
of $\DGCat_{cont}$). 

More abstractly, we could also ask for a (functorial) assignment
for every $f:\Spec(A) \to \sY$, of an assignment
of an $A$-linear category $f^*(\sC)$, with identifications:

\[
f^*(\sC) \underset{A\mod}{\otimes} B\mod = (f \circ g)^*(\sC)
\] 

\noindent for every $\Spec(B) \rar{g} \Spec(A) \rar{f} \sY$,
and satisfying higher (homotopical) compatibilities.
I.e., we ask for an object of the homotopy limit of
the diagram indexed by $\{\Spec(A) \to \sY\}$ and with
value the category of $A$-linear categories,
with induction as the structure functors.
We denote the resulting category by $\ShvCat_{/\sY}$.

Roughly speaking, we should think that the former notion
is more concrete, and that the latter notion has better functoriality
properties.

\begin{rem}

A toy model: a categorical level down, for $\sY$ a prestack,
these two ideas give two notions of ``vector space over $\sY$,"
namely, a $\Gamma(\sY,\sO_{\sY})$-module, or a
quasi-coherent sheaf on $\sY$.

\end{rem} 

As in this analogy, we have adjoint functors:

\[
\xymatrix{
\QCoh(\sY)\mod \ar@<.4ex>[rr]^(.55){\Loc = \Loc_{\sY}} &&
\ShvCat_{/\sY}
\ar@<.4ex>[ll]^(.45){\Gamma = \Gamma(\sY,-)}
}
\]

\begin{defin}[Gaitsgory, \cite{shvcat}]

$\sY$ is \emph{1-affine} if these functors are mutually
inverse equivalences.

\end{defin}

\begin{rem}

1-affineness is a much more flexible notion than usual affineness,
as we will see in \S \ref{ss:1aff-exs} below.

\end{rem}

\begin{rem}

We let $\QCoh_{/\sY}$ denote the sheaf of categories
$\Loc(\QCoh(\sY))$, i.e., this is the sheaf of categories
that assigns $A\mod$ to every $\Spec(A) \to \sY$.
Note that $\Gamma(\QCoh_{/\sY})$ is always equal to $\QCoh(\sY)$.

\end{rem}

\subsection{Regarding sheafification}

A quick aside: by 
descent for sheaves of categories (\cite{shvcat} Appendix A),
$\ShvCat_{/-}$ is immune to fppf sheafification. Therefore,
we will often \emph{not} sheafify, since this is more simpler
and more convenient in many circumstances.

We have used this convention once already 
in defining $\LocSys_G(\o{\cD})$.
We will further use it in forming quotients by group schemes $G$:
for $G$ acting on $S$, $S/G$ will denote the \emph{prestack} quotient,
and $\bB G$ will denote $\Spec(k)/G$.

\subsection{}

Some remarks  (which may be skipped at first pass) 
on functoriality, following \cite{shvcat} \S 3:

For any morphism $f: \sY \to \sZ$ of prestacks, 
we have adjoint functors:

\[
f^{*,\ShvCat}: \ShvCat_{/\sZ} \rightleftarrows \ShvCat_{/\sY} : f_*^{\ShvCat}.
\]

\noindent Wonderfully, these functors satisfy base-change
without any restrictions on the nature of the morphism $f$. 
However, $f_*^{\ShvCat}$ may not commute with colimits,
or satisfy any projection formula, or
commute with tensor products by objects of $\DGCat_{cont}$, etc.

The functor $f^{*,\ShvCat}$ is tautologically compatible with
$\Loc$, while $f_*^{\ShvCat}$ is tautologically compatible with
$\Gamma$.

E.g., one readily deduces the following result:

\begin{prop}

Given a Cartesian diagram:

\[
\xymatrix{
\sY_1 \underset{\sY_3}{\times} \sY_2\ar[r] \ar[d] &
\sY_2 \ar[d] \\
\sY_1 \ar[r] & \sY_3 
}
\]

\noindent with $\sY_i$ a 1-affine prestack for $i = 1,2,3$, 
the canonical functor:

\[
\QCoh(\sY_1) \underset{\QCoh(\sY_3)}{\otimes} \QCoh(\sY_2) \to
\QCoh(\sY_1 \underset{\sY_3}{\times} \sY_2)
\]

\noindent is an equivalence.

\end{prop}

\begin{rem}

Conversely, establishing 1-affineness in practice typically amounts
to proving many such identities.

\end{rem}

\subsection{}

We say that a morphism $f: \sY \to \sZ \in \PreStk$ 
is \emph{1-affine} if for every affine DG scheme $S$ with a map
$S \to \sZ$, the fiber product $\sY \times_{\sZ} S$ is 1-affine.
As in \cite{shvcat}, a prestack $\sY$ is 1-affine if and only
if the structure map $\sY \to \Spec(k)$ is 1-affine, and
1-affine morphisms are closed under compositions
(see also \cite{chiralcats} Appendix A).

If $f$ is 1-affine, then $f_*^{\ShvCat}$ satisfies all desirable properties:
it commutes with colimits and satisfies the projection formula.

\subsection{Examples}\label{ss:1aff-exs}

We now give the basic examples and counterexamples 
of 1-affineness. All of these results are proved in \cite{shvcat}:
see \S 2 of \emph{loc. cit}., where all of these results are stated.

\begin{thm}[Gaitsgory]

The following prestacks are 1-affine:

\begin{itemize}

\item Any quasi-compact quasi-separated DG scheme.

\item Any (classically) finite type algebraic stack, or more generally,
any eventually coconnective almost finite type DG Artin stack.
In particular, the classifying stack of an algebraic group is 1-affine.

\item For any ind-finite type indscheme $S$, $S_{dR}$ is 1-affine.

\item The formal completion $T_S^{\wedge}$ of any
quasi-compact quasi-separated DG scheme $T$ along
a closed subscheme $S \into T$ with $S^{cl} \subset T^{cl}$
defined by a locally finitely generated sheaf of ideals.

\item For $G$ an algebraic group, the classifying
(pre)stack $\bB G_e^{\wedge}$ of its formal group is
1-affine.

\end{itemize}

The following prestacks are \emph{not} 1-affine:

\begin{itemize}

\item

The indscheme\footnote{In this paper, the notation $\bA^{\infty}$
should be regarded as ``locally defined": it may refer to either
a pro-infinite dimensional affine space or to an ind-infinite dimensional
affine space, and we will specify locally which we mean.}
 $\bA^{\infty} \coloneqq \colim_n \bA^n$. Same for its formal completion
 at the origin.
 
\item The classifying prestack $\bB (\prod_{i=1}^{\infty}\bG_a)$.

\item The classifying prestack $\bB \bA^{\infty}$, with 
$\bA^{\infty}$ being the ind-infinite dimensional affine space.
The same holds for its formal group.

\end{itemize}

\end{thm}

\begin{rem}

In remembering some of these examples, a helpful\footnote{But
imperfect: many examples in \cite{shvcat} contradict this principle.
Still, it is helpful for the purposes of the present paper.}
mnemonic is
that infinite-dimensional tangent spaces are the primary
obstruction to 1-affineness. 
E.g., $\bA^{\infty} \coloneqq \colim_n \bA^n$ 
is not 1-affine ``since" it has infinite dimensional tangent spaces,
but $\bA_{dR}^{\infty}$ \emph{is} 1-affine, and its tangent spaces
vanish.

\end{rem}

\subsection{}\label{ss:mainthm-statement}

We can now formulate the main theorem of this paper.

\begin{mainthm}

For $G$ a reductive group, $\LocSys_G(\o{\cD})$
is 1-affine.

\end{mainthm}

\begin{rem}

We remind that this theorem answers a question of Gaitsgory.

\end{rem}

\begin{rem}

Note that the space of gauge forms $\fg((t))dt$ is
an indscheme that is isomorphic to a product of ind-infinite dimensional
affine space and pro-infinite dimensional affine space, and therefore
it is not 1-affine. 

Moreover, we are quotienting not just by $G(O)$, which itself
tends to create prestacks that are not 1-affine (like $\bB G(O)$),
but by $G(K)$.

The good news is that, at least for reductive $G$, these forces
are opposing and produce a 1-affine prestack in the quotient.

\end{rem}

\begin{conjecture}

$\LocSys_G(\o{\cD})$ is 1-affine for any affine algebraic group.

\end{conjecture}  

The above treats the case of reductive $G$ for this
conjecture; and the unipotent case 
follows from \S \ref{ss:ex-unipotent}, since $\LocSys$ is an Artin
stack then. The first group for which I do not know how to prove
this conjecture is $\bG_m \ltimes \bG_a$.

\subsection{Example: $G = \bG_m$}\label{ss:locsysgm-1aff}

It is instructive to analyze $G = \bG_m$, where this theorem
quickly reduces to Gaitsgory's results. We assume that
the reader has retained the material of \S \ref{ss:gm}.

Note that the intermediate $Kdt / \bG_m(O)$ is \emph{not} 1-affine,
since there is an ind-infinite dimensional affine space
as a factor.

Instead, we need to quotient gauge forms $K dt$ by
$\widehat{\bG_m(O)} \coloneqq $ the formal completion
of $\bG_m(O)$ in $\bG_m(K)$. In this case, 
we obtain:

\[
\bG_a\cdot \frac{dt}{t} \times 
\Ker(\Res:Kdt/Odt \to \bG_a)_{dR} \times \bB \bG_m
\]

\noindent as in \S \ref{ss:gm}. Since the infinite-dimensional
affine space is replaced by its de Rham version, this quotient 
\emph{is} 1-affine.

It remains to quotient by $\bG_m(K)/\widehat{\bG_m(O)} = \bZ$.
One can readily show that $\bB \bZ$ is 1-affine.\footnote{A proof,
for the interested: tautologically, categories over $\bB \bZ$
are equivalent to categories with an automorphism
(the fiber functor corresponds to pullback $\Spec(k) \to \bB \bZ$).
Moreover, $\QCoh(\bB \bZ)$ is equivalent to $\QCoh(\bG_m)$
with its convolution structure, so module categories
for $\QCoh(\bB \bZ)$ are equivalent to categories over
$\bB \bG_m$. By 1-affineness of $\bB \bG_m$, the 
latter are equivalent to $\Rep(\bG_m)$-module categories,
i.e., to categories with an automorphism. Then
it is a simple matter of chasing the constructions to
see that $\Gamma: \ShvCat_{/\bB \bZ} \to \QCoh(\bB \bZ)\mod$
corresponds to the identity functor for categories with
an automorphism, and therefore is an equivalence.

Or, if one likes, this follows more conceptually 
from the method of \cite{shvcat} \S 11.}

Then the morphism:

\[
\LocSys_{\bG_m}(\o{\cD}) = (Kdt/\widehat{\bG_m(O)})/\bZ \to \bB \bZ
\]

\noindent is 1-affine (since the fibers are $Kdt/\widehat{\bG_m(O)}$).
Since $\bB \bZ$ is 1-affine, this implies that $\LocSys_{\bG_m}(\o{\cD})$
is also 1-affine.

\subsection{}

For general reductive $G$, such an explicit analysis does not work.
However, we will show the following results, inspired by the above.

The following is the main result in \S \ref{s:tame}.

\begin{thm*}[Thm. \ref{t:tame}]

For any affine algebraic group $G$, $t^{-r}\fg[[t]]dt/G(O)$ is
1-affine.

\end{thm*}

Unfortunately, this result is not so convenient for passing to
the limit in $r$. Therefore, we show the following in \S \ref{s:z-red}-\ref{s:z-geom}.

\begin{thm*}[Thm. \ref{t:z}]

For any affine algebraic group $G$, the formal completion 
of $t^{-r}\fg[[t]]dt$ in $\fg((t))dt$ modulo the gauge
action of $\widehat{G(O)}$ (the formal completion of
$G(O)$ in $G(K)$) is 1-affine.

\end{thm*}

We will obtain the following result as a corollary.

\begin{thm*}[Thm. \ref{t:mod-g(o)-hat}]

$\fg((t))dt/\widehat{G(O)}$ is 1-affine (for any affine algebraic group $G$).

\end{thm*}

We will finally use the ind-properness of $\Gr_G$ 
(the sheafification
of $G(K)/G(O)$) for \emph{reductive} $G$ 
to deduce the main theorem in \S \ref{s:finale}.

\subsection{A heuristic}

Let us give a heuristic explanation for why our geometric results imply 
Theorem \ref{t:tame}, i.e., the 1-affineness of $t^{-r}\fg[[t]]dt/G(O)$
and the first major step towards the main theorem.
The reader may safely skip this section.

First, why is $\bB G(O)$ \emph{not} 1-affine? 
Here is a heuristic, which is less scientific than the proof given
in \cite{shvcat}. It relies on some general notions from the
theory of group actions on categories that are reviewed in 
\S \ref{s:tame}.

It is easy to see that if $\bB G(O)$ were 1-affine,
then invariants and coinvariants for $\QCoh(G(O))$-module categories would coincide.\footnote{Proof: $\QCoh(G(O))$-module categories
are easily seen to be the same as sheaves of categories on
$\bB G(O)$, and global sections matches up with invariants.
So if $\bB G(O)$ were 1-affine, global sections would commute with
colimits and be $\DGCat_{cont}$-linear. 
This would allow us to reduce to checking that the
norm functor $\sC_{G(O),w} \to \sC^{G(O),w}$ is an equivalence
for $\sC = \QCoh(G(O))$, where it is clear.}

However, the identity functor for $\Vect$ 
induces a functor $\Vect_{G(O),w} \to \Vect$. If
if $\Vect_{G(O),w} \isom \Vect^{G(O),w}$, we obtain an 
induced functor $\QCoh(\bB G(O)) = \Vect^{G(O),w} \to \Vect$, 
and morally this
functor computes $G(O)$-invariants of representations. Moreover,
this functor tautologically is continuous, since that is built into our
framework.

However, as discussed above, the trivial representation
in $\QCoh(\bB G(O))$ is \emph{not} compact, so 
$G(O)$-invariants does not commute with colimits.
What would the functor above be? 
And indeed, Gaitsgory's result that $\bB G(O)$ is not 1-affine
rules out the existence of this functor.

Then the heuristic explanation for the difference
between $t^{-r}\fg[[t]]dt/G(O)$ and $\Spec(k)/G(O)$ is
that the former has a continuous global sections functor,
as was explained in \S \ref{ss:intro-cptgen}.

\subsection{Structure of this paper}

We have basically given it already above.

In \S \ref{s:locsys}, we give our geometric results on 
$\LocSys_G(\o{\cD})$, as described above. 
In \S \ref{s:tame}, we show that $t^{-r}\fg[[t]]dt/G(O)$ is 1-affine;
the main ideas were summarized in \S \ref{ss:intro-cptgen}.
In \S \ref{s:cpt-gen}, we prove the compact generation
for $\QCoh(\LocSys_G(\o{\cD}))$ (for $G$ reductive).
In \S \ref{s:z-red}-\ref{s:z-geom}, we extend Theorem \ref{t:tame} in infinitesimal
directions to obtain Theorem \ref{t:z}. This proof of this
result is somewhat involved, which is why we have spread it across
three sections.
Finally, in \S \ref{s:finale}, we complete the proof of the main theorem,
which is straightforward after Theorem \ref{t:z}.

The paper is intentionally structured in increasing order of
complexity: \S \ref{s:locsys} does not have any category
theory, and then the level of categorical sophistication
increases over \S \ref{s:tame}-\ref{s:z-geom} (e.g., starting
in \S \ref{s:cpt-gen}, it is very helpful, if not strictly
necessary, to be familiar with
$\IndCoh$, c.f. \cite{indcoh}).

In particular, ideas I learned from
Akhil Mathew's senior thesis \cite{akhil-thesis} regarding
the Barr-Beck formalism were key in proving Theorem
\ref{t:z}; these ideas are summarized in \S \ref{s:z-tens}.

\subsection{Some conventions}

We use higher categorical language throughout, letting \emph{category}
mean \emph{$(\infty,1)$-category}, (co)limit means \emph{homotopy
(co)limit}, etc. 

Most of our conventions about DG categories were recalled in \S \ref{ss:dgcat-linalg}.
One warning: following the above conventions, we use the notation 
$\Coker(\sF \to \sG)$ where others would use $\on{Cone}$, and we
use $\Ker(\sF \to \sG)$ where others would use $\on{Cone}[-1]$.
If we mean to take a co/kernel in an abelian category, not in the
corresponding derived category, we will be cautious to indicate this desire to the reader.

For $\sC$ a DG category with $t$-structure, we let
$\sC^{\leq 0}, \sC^{\geq 0} \subset \sC$ denote the corresponding
subcategories, where we use \emph{cohomological} grading
throughout. We let $\sC^{\heart} \coloneqq \sC^{\leq 0} \cap \sC^{\geq 0}$
denote the heart of the $t$-structure. 

Finally, we assume that the reader is quite comfortable with the linear
algebra of DG categories. We refer to \cite{dgcat}, \cite{shvcat}, \cite{indcoh}, \cite{indschemes} (esp. \S 7)
and other foundational papers on Gaitsgory's website for an introduction to the
subject. 

\subsection{Acknowledgements}

I'm grateful to Dima Arinkin, Vladimir Drinfeld, 
Akhil Mathew, and Ivan Mirkovi\'c for their influence on this work.

Thanks to Dennis Gaitsgory for highlighting this question in the first
place, and for his careful reading of a draft of this paper and his
many suggestions for its improvement. 

Thanks also to Dario Beraldo, whose ready conversation
in the early stages of this project diverted many wrong turns
and were key to its development.

Finally, thanks to Sasha Beilinson, to whom this paper is dedicated,
for so much; not least of all, thanks are due for guiding me to
local class field theory and differential equations in the same breath.

This material is based upon work supported by the 
National Science Foundation under Award No. 1402003.


\section{Semi-infinite geometry of de Rham local systems}\label{s:locsys}


\subsection{}

The motto of this section: in spite of all the evil in the world
(e.g., $\LocSys(\o{\cD})$ is far from an Artin stack; not all $\nabla$ are Fredholm), 
there is some current of good (c.f. Theorems \ref{t:artin} and \ref{t:jumps}). 

\begin{rem}

We remind the reader to visit \S \ref{ss:geom-start}-\ref{ss:geom-finish}
for a proper introduction to this material. 
In particular, \S \ref{ss:gm-ga} is essential for understanding
why we need to work with geometric fibers over the leading terms
space.

\end{rem}

\begin{rem}

To the reader overly steeped in derived algebraic geometry,
we emphasize that the manipulations in this section are really
about classical algebraic geometry, and we allow ourselves
the full toolkit of classical commutative algebra throughout.

\end{rem}

\subsection{Tate's linear algebra} 

We give a quick introduction to the language of Tate objects in the derived setting. This language plays a fairly supporting role in what follows,
and we include it only for convenience.

The definitions and basic properties given here were independently 
found by Hennion, see \cite{hennion}.

\begin{rem}

We remark that this setup works just as well for
stable $\infty$-categories as for DG categories: it is only a question 
of language.

\end{rem}

\subsection{}\label{ss:setup} Let $\sC\in\DGCat$ be a fixed compactly generated DG category and let $\sC^0\subset \sC$ be the full subcategory
of compact objects.

\begin{defin}

The Tate category $\Tate(\sC)$ is the full subcategory of $\Pro(\sC)$ 
Karoubi-generated (i.e., generated under
finite colimits and retracts) by $\sC \subset \Pro(\sC)$ and $\Pro(\sC^0)$.

Objects of $\Pro(\sC^0) \subset \Tate(\sC)$ are sometimes called 
\emph{lattices}, and objects of $\sC \subset \Tate(\sC)$ are sometimes
called \emph{colattices}.

\end{defin}

\subsection{}

The following gives a more symmetrical perspective on the Tate construction.

\begin{prop}\label{p:tate-pushout}

The pushout of the diagram:

\[
\xymatrix{
\sC^0 \ar[r] \ar[d] & \Ind(\sC^0) = \sC \\
\Pro(\sC^0)
}
\]

\noindent exists in the (very large) category $\DGCat_{Kar}$, the category
of Karoubian (alias: idempotent complete) DG categories.
Moreover, the obvious functor from this pushout to $\Tate(\sC)$ is an equivalence.

\end{prop}

\begin{cor}

The Tate construction commutes with duality: $\Tate(\sC)^{op}=\Tate(\sC^{\vee})$.

\end{cor}

Proposition \ref{p:tate-pushout} plays only a supporting role
in what follows, and is proved in \cite{hennion} (c.f. \emph{loc. cit}. 
Theorem 2), so we do not give the proof here.

\subsection{Fredholm operators}

Following \cite{bbe} \S 2, we make the following definition.

\begin{defin}

A morphism $T:\sF \to \sG \in \Tate(\sC)$ is \emph{Fredholm}
if $\Coker(T) \in \sC^0 \subset \Tate(\sC)$.

\end{defin}

\subsection{Example: Laurent series}

Suppose $A$ is a (classical) commutative ring and 
$V$ is a rank $n$ free $A((t))$-module. Then $V$ inherits
an obvious structure of object of $\Tate(A\mod)$.

\begin{rem}

One advantage of using the Tate formalism (or at least pro-objects)
is that we can use formulae like $A((t)) \otimes_A B = B((t))$,
as long as we are understanding (as we always will) $A((t))$ as an 
object of $\Pro(A\mod)$ and $B((t))$ as an object
of $\Pro(B\mod)$. 

\end{rem}

In this setting, we will use the following terminology.

\begin{defin}

A \emph{lattice} in $V$ is an object\footnote{We are using the
standard $t$-structure on $\Pro(A\mod)^{\heart}$, characterized
by the fact that it is compatible with filtered limits and restricts to the usual
$t$-structure on $A\mod$.}
$\Lambda \in \Pro(A\mod)^{\heart}$
that can be written as a limit in\footnote{I.e., we are asking $R^i\lim$ to vanish for $i>0$.}
$\Pro(A\mod)$ of finite rank projective 
$A$-modules,and which has been equipped with an admissible monomorphism
$\Lambda \into V$ with $\Coker(\Lambda \to V)$ lying in\footnote{Of course,
it lies in $A\mod^{\heart}$ then.} 
$A\mod \subset \Pro(A\mod)$ and flat.

\end{defin}

\begin{lem}\label{l:lattice-quot}

For any $\Lambda_1 \into \Lambda_2 \into V$ a pair of lattices,
$\Lambda_2/\Lambda_1$ is a finite rank projective $A$-module.

\end{lem}

We will deduce this using the following general result.

\begin{lem}\label{l:tate-cpts}

For $\sC$ a compactly generated DG category, the intersection
$\Pro(\sC^0) \cap \sC$ in $\Tate(\sC)$ equals $\sC^0$.

\end{lem}

\begin{proof}

Tautologically, this intersection is formed in $\Pro(\sC)$. 
If $\sF \in \Pro(\sC^0) \cap \sC$, write $\sF = \lim_i \sF_i$ with
$\sF_i \in \sC^0$.
Since $\sF \in \sC$, the identity map for $\sF$ 
must factor through some $\sF_i$, meaning that $\sF$ is a retract
of $\sF_i$. But $\sC^0$ is closed under retracts, so we obtain the claim.

\end{proof}

\begin{proof}[Proof of Lemma \ref{l:lattice-quot}]

First, observe that 
$\Lambda_2/\Lambda_1 \in \Pro(\Perf(A\mod))\cap A\mod$.
Indeed, it obviously lies in $\Pro(\Perf(A\mod))$, and it lies
in $A\mod$ since it sits in an exact triangle with
$V/\Lambda_1$ and $V/\Lambda_2$.
Therefore, Lemma \ref{l:tate-cpts} implies that it lies in $\Perf(A\mod)$.

Since this quotient lies in cohomological degree $0$, we
see that it is finitely presented.
Therefore, it suffices to show that this quotient is flat.
This follows again from the resolution:

\[
\Lambda_2/\Lambda_1 = \Ker(V/\Lambda_1 \to V/\Lambda_2).
\]

\end{proof}

\begin{defin}

An \emph{$A[[t]]$-lattice} in $V$ is a lattice $\Lambda \subset V$ which
is an $A[[t]]$-submodule, (equivalently: for which $t\Lambda \subset \Lambda$).

\end{defin}

We have the following structural result.

\begin{lem}

Any $A[[t]]$-lattice $\Lambda \subset V$ is a projective
$A[[t]]$-module with
$\Lambda \otimes_{A[[t]]} A((t)) \isom V$.
(In particular, $\Lambda$ has rank $n$ over $A[[t]]$.)

\end{lem}

\begin{proof}

Note that
no $\Tor$s are formed when we form 
$\Lambda/t^r\Lambda = \Lambda \otimes_{A[[t]]} A[[t]]/t^r$,
and by Lemma \ref{l:lattice-quot}, $\Lambda/t\Lambda$ is a finite rank projective
$A$-module. It follows that each $\Lambda/t^r\Lambda$ is projective
$A[t]/t^r$-module with rank independent of $r$.
By a well-known argument, this implies that $\Lambda$ is projective over $A[[t]]$.

For a choice of isomorphism $V \isom A((t))^{\oplus n}$, 
$\Lambda$ is wedged between two $A[[t]]$-lattices of the 
form $t^s A[[t]]^{\oplus n}$ for some choices of $s\in \bZ$.
It then immediately follows that $\Lambda \otimes_{A[[t]]} A((t)) \isom V$.

\end{proof}

\subsection{Example: local de Rham cohomology}

Let $S \coloneqq \Spec(A)$, and suppose that we are 
given a differential module $\chi = (V,\nabla)$ over $\o{\cD}_S$, i.e., 
$V$ is an $A((t))$-module free of some rank $n$ and
$\nabla: V \to V dt$ is $A$-linear and satisfies the Leibniz rule.

Below, we will construct the local de Rham cohomology\footnote{The notation
is misleading: the reader should think de Rham cochains, not merely de Rham cohomology.
We will be careful to use an $i$ instead of $*$ when we mean to refer to a specific
de Rham cohomology group, and promise the reader never to refer to the graded
vector space usually denoted in this way.} $H_{dR}^*(\o{\cD}_S,\chi)$
as a Tate $A$-module, i.e., as an object of $\Tate(A\mod)$.

Note that $V \simeq A((t))^{\oplus n}$ obviously defines an object of
$\Tate(A\mod)$. 

Let $\Lambda \subset V$ be an $A[[t]]$-lattice, i.e., a finite rank free $A[[t]]$-submodule
spanning under the action of $A((t))$.  
By definition of differential module, there exists an integer $r$ such that
$\nabla(\Lambda) \subset t^{-r} \Lambda dt$ for every choice of $A[[t]]$-lattice 
$\Lambda$.\footnote{To see this: choose a basis, so $V\isom A((t))^{\oplus n}$
and $\nabla = d + \Gamma dt$ for some matrix $\Gamma$. 
Then combine the fact that $\Gamma$ has a pole of bounded order
with the fact that every lattice is wedged between two lattices
of the form $t^s A[[t]]$, $s\in \bZ$.}
 
We see that for each integer $s$, the map:

\[
\Lambda \to t^{-r}\Lambda dt \to t^{-r}\Lambda dt / t^s\Lambda dt
\]

\noindent factors through $\Lambda/t^{r+s} \Lambda$.

Therefore, we obtain a morphism $\Lambda \to t^{-r}\Lambda dt$ in $\Pro(\Perf(A\mod))$, where we consider
$\Lambda$ and $t^{-r}\Lambda dt$ as objects of this category in the obvious way. Taking the kernel
of this morphism, we obtain an object of
$\Pro(\Perf(A\mod))$ encoding the complex $\Lambda \to t^{-r} \Lambda$ (considered as a complex in degrees $0$ and $1$).

Passing to the colimit in $\Pro(A\mod)$ under all such choices of $A[[t]]$-lattice, 
we obviously  obtain an object of $\Tate(A\mod) \subset \Pro(A\mod)$, since 
$V/\Lambda$ and $Vdt/t^{-r}\Lambda dt$ are both objects of $A\mod \subset \Pro(A\mod)$.

Note that $H_{dR}^*(\o{\cD}_S,\chi)$ is the kernel of the morphism:

\[
\nabla:V \to Vdt \in \Tate(A\mod).
\]

\begin{example}

In this formalism, formation of the de Rham complex commutes with base-change
in the $A$-variable as is.

\end{example}


\subsection{}

The following notion will play a key role in what follows.

\begin{defin}

A differential module $\chi = (V,\nabla)$ on $\o{\cD}_S$ is \emph{Fredholm}\footnote{A
closely related notion was introduced in \cite{bbe}, where it was called
\emph{$\vareps$-nice} due to the nice behavior of $\vareps$-factors under this hypothesis.} 
if $H_{dR}^*(\o{\cD}_S,\chi) \in \Perf(A\mod) \subset \Tate(A\mod)$, i.e., if
$\nabla: V \to Vdt$ is Fredholm.

\end{defin}

\begin{example}

If $A = F$ is a field, then it is well-known that every differential
module $(V,\nabla)$ is Fredholm: indeed, this follows at once
from the finite-dimensionality of the de Rham cohomology in this setting
(see \cite{bbe} \S 5.9 for related discussion).

\end{example}

\begin{example}\label{e:invertible}

Suppose that we are given a connection on $A((t))^{\oplus n}$ written as:

\[
d + \Gamma_{-r}t^{-r}dt + \text{lower order terms} 
\]

\noindent where $\Gamma_{-r}$ is an $n \times n$-matrix with
entries in $A$. Suppose that $r>1$ and $\Gamma_{-r}$ is invertible. Then the corresponding
differential module is Fredholm.

Indeed, for  
$\Lambda = A[[t]]^{\oplus n}$ and any integer $s > 0$,  
the map:

\[
t^{-s} \Lambda \xar{\nabla} t^{-r-s} \Lambda dt
\]

\noindent is a quasi-isomorphism, as is readily seen using the $t$-adic filtrations on both sides.
Passing to the limit, we see that $H_{dR}^*(\o{\cD}_S,\chi)$ vanishes in this case.

In particular, in the rank $1$ case, if the pole order is at least $2$ and does not jump (i.e., if the leading term is invertible),
then the corresponding connection is Fredholm. 

\end{example}

\begin{example}\label{e:rs-evalues}

Suppose that we are given a \emph{regular singular} connection on 
$A((t))^{\oplus n}$, so:

\[
\nabla = d + \Gamma_{-1}t^{-1}dt + \text{lower order terms} 
\]

\noindent where $\Gamma_{-1}$ is an $n \times n$-matrix with
entries in $A$. Suppose that $N\Id+\Gamma_{-1}$ is invertible for almost
every integer $N$.

Then using $t$-adic filtrations as above, we find that $\chi$ is Fredholm. 
Indeed, let $\Lambda = A[[t]]^{\oplus n}$ (for $A$ our ring of coefficients,
as usual), and note that $\nabla$ maps $t^N \Lambda$ to
$t^{N-1}\Lambda dt$ by assumption. Moreover, at the
associated graded level, $\nabla$ induces the map:

\[
t^N\Lambda/
t^{N+1} \Lambda = A^{\oplus n} 
\xar{\Gamma_{-1} + N\cdot \id}
A^{\oplus n} =
t^{N-1} \Lambda dt/ t^N \Lambda dt.
\]

Therefore, for $s\gg 0$, the map:

\[
t^s A[[t]]^{\oplus n} \xar{\nabla} t^{s-1} A[[t]]^{\oplus n} dt
\]

\noindent is an isomorphism, while for $r\gg 0$ the map:

\[
A((t))^{\oplus n}/t^{-r}A[[t]]^{\oplus n} \xar{\nabla} A((t))^{\oplus n} dt/t^{-r-1}A[[t]]^{\oplus n} dt
\]

\noindent is an isomorphism.

\end{example}

\begin{counterexample}\label{ce:rs-rk1-notfredholm}

For $A = k[\lambda]$ with $\lambda$ an indeterminate, the connection:

\[
V = A((t)), \nabla = d+\lambda \frac{dt}{t}
\]

\noindent is \emph{not} Fredholm. Indeed, for this connection, $H_{dR}^1$ 
is the sum of skyscraper sheaves supported on $\bZ \subset \bA^1 = \Spec(k[\lambda])$.

\end{counterexample}

\subsection{}

The following basic results will be of use to us.

\begin{lem}\label{l:fred-lattices}

Let $\chi = (V,\nabla)$ be a Fredholm differential module over $S = \Spec(A)$, and let 
$\Lambda \subset V$ and $\Lambda^{\prime} \subset V^{\prime}$ be
$A[[t]]$-lattices with the property that $\nabla(\Lambda) \subset \Lambda^{\prime}$. Then:

\begin{enumerate}

\item\label{i:fred-lattices-1}

The map:

\[
\nabla: \Lambda \to \Lambda^{\prime} \in \Pro(\Perf(A\mod)) \subset \Tate(A\mod)
\]

\noindent is Fredholm.

\item\label{i:fred-lattices-2}

For $N\gg 0$, the complex:

\[
\Coker(t^N \Lambda \to \Lambda^{\prime}) \in \Perf(A\mod)
\]

\noindent is of the form (i.e., quasi-isomorphic to a complex) $P[-1]$ 
for $P$ a finite rank projective $A$-module.

\end{enumerate}

\end{lem}

\begin{proof}

For \eqref{i:fred-lattices-1}:

The 2-step complex $\Lambda \xar{\nabla} \Lambda^{\prime}$
lies in $\Pro(A\mod)$ (by construction). Moreover, it sits in an exact
triangle with $H_{dR}^*(\o{\cD}_S,\chi)$, which by assumption lies
in $\Perf(A\mod)$, and the complex:

\[
\Coker(V/\Lambda \xar{\nabla} Vdt/\Lambda^{\prime})
\] 

\noindent which obviously lies in $A\mod \subset \Tate(A\mod)$. Therefore,
we obtain the claim from Lemma \ref{l:tate-cpts}.

We now deduce \eqref{i:fred-lattices-2}. 

First, note 
that $\nabla$ maps $t^N\Lambda$ to $t^{N-1}\Lambda^{\prime}$ for
every $N\geq 0$. Indeed, for $s\in \Lambda$, $\nabla(ts) = sdt + t\nabla(s)$,
so we see that $sdt \in \Lambda^{\prime}$ (since $\nabla(ts)$ and $t\nabla(s)$ are).
Therefore, $\nabla(t^Ns) = Nt^{N-1}sdt + t^N\nabla(s)$ lies in $t^{N-1} \Lambda^{\prime}$
as desired.

We now claim that for $N\gg 0$, the map of 2-step complexes:

\[
\xymatrix{
t^N\Lambda \ar[r]^{\nabla} \ar[d] & t^{N-1} \Lambda^{\prime} \ar[d] \\
V \ar[r]^{\nabla} & Vdt 
}
\]

\noindent is (isomorphic to) the zero map in $\Tate(\sC)$. 

Indeed, since $\Lambda \to \Lambda^{\prime} = 
\lim_N \Lambda/t^N\Lambda \to \Lambda^{\prime}/t^{N-1}\Lambda^{\prime}$
as a pro-object, and since $H_{dR}^*(\o{\cD}_S,\chi)$ lies in $\Perf(A\mod)$,
the map $(\Lambda \to \Lambda^{\prime}) \to H_{dR}^*(\o{\cD}_S,\chi)$
must factor through 
$(\Lambda/t^N\Lambda \to \Lambda^{\prime}/t^{N-1}\Lambda^{\prime})$
for some $N$, giving the claim.

We now claim that taking $N$ with this property suffices for the conclusion.
Indeed, we have seen in \eqref{i:fred-lattices-1} that 
$\Coker(\nabla: t^N\Lambda \to \Lambda^{\prime})$ is perfect as an $A$-module,
so it suffices to see that it has $\Tor$-amplitude $0$. 
It obviously suffices to show this for 
$\Coker(\nabla:t^N\Lambda \to t^{N-1}\Lambda^{\prime})$ instead.
By construction, this object has $\Tor$-amplitude in $[-1,0]$, so it suffices
to show that it has $\Tor$-amplitude $\geq 0$.

To this end, note that since the above map is zero, we obtain an isomorphism:

\[
\Ker(V/t^N\Lambda \xar{\nabla} Vdt/t^{N-1}\Lambda^{\prime}) \simeq 
H_{dR}^*(\o{\cD}_S,\chi) \oplus \Coker(t^N\Lambda \xar{\nabla} t^{N-1}\Lambda^{\prime})
\]

\noindent upon taking its cone. Therefore, our cokernel is a direct summand
of a complex visibly of $\Tor$-amplitude $[0,1]$, and therefore itself
has $\Tor$-amplitude $[0,1]$ as desired.

\end{proof}

\subsection{}

With a bit more work, we have the following more precise version of 
Lemma \ref{l:fred-lattices} \eqref{i:fred-lattices-2}.

\begin{lem}\label{l:fred-lattices-quant}

In the notation of Lemma \ref{l:fred-lattices}:

Suppose that $A$ has finitely many minimal prime ideals and its nilradical
is nilpotent.\footnote{E.g., $A$ is Noetherian, or an integral domain, or
a (possibly infinite) polynomial algebra over a Noetherian ring.}

Then there exist integers $\ell$ and $r_0$ such that for all $r\geq r_0$,
$t^{r+\ell}\Lambda^{\prime} \subset \nabla(t^r\Lambda)$ with finite rank projective quotient.

\end{lem}

\begin{proof}

We proceed by steps.

\step\label{st:fred-2-start}

First, observe that projectivity of the quotient follows at once if we know the inclusion:

\begin{equation}\label{eq:nabla-incl}
t^{r+\ell}\Lambda^{\prime} \subset \nabla(t^r\Lambda).
\end{equation}

Indeed, first note that we may safely assume $r_0$ is large enough such that
the conclusion of Lemma \ref{l:fred-lattices} holds (i.e., so that
$\nabla(t^r\Lambda)$ is a lattice). 
Then we are taking the quotient of one lattice by another,
so the claim follows from Lemma \ref{l:lattice-quot}. 
\step 

Next, we reduce to the case where $A$ is reduced. More precisely, suppose $I \subset A$
is a nilpotent ideal. We claim that if the lemma holds for our
lattices modulo $I$, then it holds for our lattices.
We obviously can reduce to the case where $I^2 = 0$ (just to make 
the numerics simpler).

Suppose $(r_0,\ell)$ satisfy the conclusion of the lemma for our
lattices modulo $I$. We will show that $(r_0,2\ell+r_0)$ satisfies the conclusion of the lemma for $A$.

Using Lemma \ref{l:fred-lattices}, we can make sure to choose $r_0$
so that $t^{r_0-1}\Lambda dt \subset \Lambda^{\prime}$
and $\nabla(t^{r_0}\Lambda) \subset \Lambda^{\prime}$,
in which case:

\begin{equation}\label{eq:simple-incl}
\nabla(t^r\Lambda) \subset t^{r-r_0}\Lambda^{\prime}
\end{equation}

\noindent for all $r \geq r_0$. Indeed, for $s \in \Lambda$,
we then have:

\[
\nabla(t^r s) = \nabla(t^{r-r_0} \cdot ( t^{r_0} s)) =
(r-r_0) t^{r-1} s dt + t^{r-r_0} \nabla(t^{r_0}s) 
\]

We now claim that:

\begin{equation}\label{eq:alsosimple-incl}
t^{r+\ell}\Lambda^{\prime} \subset \nabla(t^r\Lambda) + I t^{r-r_0}\Lambda^{\prime}.
\end{equation}

\noindent Indeed, for $\omega \in t^{r+\ell}\Lambda^{\prime}$,
note that $\omega = \nabla(s) + \vareps$ where
$s\in t^r\Lambda$ and $\vareps \in I\cdot V$, since
we are assuming $(r_0,\ell)$ satisfies our hypotheses
modulo $I$. Moreover, by \eqref{eq:simple-incl},
$\vareps \in I V \cap t^{r-r_0}\Lambda^{\prime}$,
so, it suffices to note that
$I V \cap t^{r-r_0}\Lambda^{\prime} = 
I \cdot t^{r-r_0}\Lambda^{\prime}$. In turn, this equality holds
because $I V = I \otimes_A V$ and 
$I \Lambda^{\prime} = I \otimes_A \Lambda^{\prime}$ by
pro-projectivity, and then we see that (with e.g. $\Ker$ denoting
homotopy kernels everywhere):

\[
\begin{gathered}
\big(I V \cap t^{r-r_0}\Lambda^{\prime} \big) /
I \cdot t^{r-r_0}\Lambda^{\prime} = 
H^0\big(\Ker(t^{r-r_0}\Lambda^{\prime}/It^{r-r_0}\Lambda^{\prime} \to
V/IV)\big) = \\
H^0\big(A/I \underset{A}{\otimes} \Ker(t^{r-r_0}\Lambda^{\prime} \to V)\big)
= H^0(A/I \underset{A}{\otimes} V/t^{r-r_0}\Lambda^{\prime}[-1]) = 0
\in \Tate(A\mod)
\end{gathered}
\]

\noindent as desired.

Finally, we show that $(r_0,2\ell+r_0)$ satisfies for the
lemma. Indeed, iterating \eqref{eq:alsosimple-incl},
we have:

\[
\begin{gathered}
t^{r+r_0+2\ell} \Lambda^{\prime} \subset
\nabla(t^{r+r_0+\ell}\Lambda) + I t^{r+\ell}\Lambda^{\prime} \subset
\nabla(t^{r+r_0+\ell}\Lambda) + I (\nabla(t^r\Lambda)+I t^{r-r_0}\Lambda^{\prime}) = \\
\nabla(t^{r+r_0+\ell}\Lambda) + I \nabla(t^r\Lambda) \subset
\nabla(t^r\Lambda)
\end{gathered}
\]

\noindent giving the claim.

\step

We now reduce to the case where $A = F$ is an algebraically closed field.

First, suppose $A \into B$ is any embedding of commutative rings.
We claim that if $(r_0,\ell)$ suffice for our lattices tensored with $B$,
then they suffice before tensoring with $B$ as well.

By Step \ref{st:fred-2-start},
it suffices to show that the inclusion \eqref{eq:nabla-incl} 
holds if and only if it holds after tensoring with $B$.
Indeed, such an inclusion is equivalent to the fact that
the map $t^{r+\ell}\Lambda^{\prime} \to \Lambda^{\prime}/\nabla(t^r\Lambda)$
is zero, and since the quotient on the right is projective, it embeds
into its tensor product with $B$. 

Now the fact that $A$ is reduced with finitely many minimal primes means\footnote{Recall
that the intersection of minimal prime ideals in any commutative ring is
the nilradical.} that we have:

\[
A \into \prod_{\fp \subset A \text{ minimal}} A/\fp \into 
\prod_{\fp \subset A \text{ minimal}} F_{\fp}
\]

\noindent where $F_{\fp}$ is a choice of algebraic closure of the fraction field of $A/\fp$.
Since this product is finite, we can reduce to the case where $A$ coincides
with one of the factors, as desired.

\step\label{st:assume-levelt-lattice}

In the next step, we will show that for $A = F$ algebraically closed, there is an $F[[t]]$-lattice 
$\Lambda_0 \subset V$ with the property that:

\[
t^{r-1} \Lambda_0 dt \subset \nabla(t^{r} \Lambda_0) 
\] 

\noindent for all $r \geq 0$. Assume this for now, and we will deduce the lemma.

We may then find $r_0 \geq 0$ such that $t^{r_0+1}\Lambda^{\prime} \subset \Lambda_0 dt$
and $\ell \geq 0$ such that $t^{\ell}\Lambda_0 \subset t^{r_0}\Lambda$. 
We claim that this choice of integers suffices.

Indeed, for any $r \geq r_0$, we have
$t^{\ell+r-r_0}\Lambda_0 \subset t^r\Lambda$,
and applying our hypothesis on $\Lambda_0$, we obtain:

\[
t^{\ell+r-r_0-1}\Lambda_0 dt \subset
\nabla(t^{\ell+r-r_0}\Lambda_0) \subset
\nabla(t^{\ell+r}\Lambda).
\]

\noindent We then have:

\[
t^{\ell+r}\Lambda^{\prime} \subset
t^{\ell+r-r_0-1}t^{r_0+1}\Lambda^{\prime} \subset t^{\ell+r-r_0-1}\Lambda_0 dt
\]

\noindent as desired.

\step

We now construct $\Lambda_0$ as above using the Levelt-Turrittin 
decomposition\footnote{It would be great to have a more direct argument here.}
c.f. \cite{levelt}.

First, suppose that for some $e \in \bZ^{>0}$ we have constructed:

\[
\Lambda_1 \subset V\underset{F((t))}{\otimes} F((t^{\frac{1}{e}}))
\]

\noindent satisfying the corresponding property, i.e., such that:

\[
(t^{\frac{1}{e}})^{r-1} \Lambda_1 d(t^{\frac{1}{e}}) = 
t^{\frac{r}{e}-1} \Lambda_1 dt
\subset \nabla(t^{\frac{r}{e}} \Lambda _1)
\]

\noindent for every $r \geq 0$.
Define $\Lambda_0$ as $\Lambda_1 \cap V \subset V \otimes_{F((t))} F((t^{\frac{1}{e}}))$.
Clearly $\Lambda_0$ is an $A[[t]]$-lattice, and we claim it satisfies the desired property:

Indeed, if $r \in \bZ^{\geq 0}$ and $s \in t^{r-1} \Lambda_0 dt $, then
we can find $\sigma \in t^r \Lambda_1$ with $\nabla(\sigma) = s$.
Note that $V\otimes_{F((t))} F((t^{\frac{1}{e}}))$ decomposes as an $F((t))$-module
as $\oplus_{i=0}^{e-1} V t^{\frac{i}{e}}$, and similarly for 
$V \otimes_{F((t))}F((t^{\frac{1}{e}})) dt$, and these decompositions are compatible with
$\nabla$. Since $\nabla(\sigma) = s \in Vdt$, the
component $\sigma_0 \in V \subset \oplus_{i=0}^{e-1} V t^{\frac{i}{e}}$ of $\sigma$
also maps to $s$ under $\nabla$. Since $\sigma_0 \in t^r\Lambda_0$, this gives the
desired claim.

Therefore, we may replace $F((t))$ by $F((t^{\frac{1}{e}}))$ for any $e>0$. 
Then, because $F$ is algebraically closed, 
the Levelt-Turrittin theorem says\footnote{More specially, it says that
after such an extension a differential module decomposes as a direct
sum of a regular module and modules that are tensor products of
a rank 1 irregular connection with a regular connection. The latter kind
of connection obviously has an invertible (diagonal) leading term.}
that (after such an extension) $V$ is a direct sum of 
differential modules each of which is either regular
or else (up to gauge transformation) has an invertible leading term,
reducing us to treating each of these cases.

But these cases were treated already in Examples \ref{e:invertible} and \ref{e:rs-evalues},
completing the proof of the existence of the lattice
claimed in Step \ref{st:assume-levelt-lattice}.

\end{proof}

\subsection{}

We also record the following simple result for later use.

\begin{lem}\label{l:extn-scalars}

Let $(V,\nabla)$ be a differential module over $S = \Spec(A)$. Suppose 
$d \in \bZ^{>0}$ is given. 

Then if $V[t^{\frac{1}{d}}] \coloneqq V \otimes_{A((t))} A((t^{\frac{1}{d}}))$ equipped
with its natural connection is Fredholm, then $(V,\nabla)$ is Fredholm.

\end{lem}

\begin{proof}

As a differential module over $A((t))$, $V[t^{\frac{1}{d}}]$ is isomorphic to:

\[
\oplus_{i=0}^{d-1} \, V \otimes \chi_{\frac{i}{d}}
\]

\noindent where $\chi_{\frac{i}{d}}$ is the rank 1 connection:

\[
d + \frac{i}{d} \frac{dt}{t}.
\]

\noindent This obviously gives the claim, since $V$ is a direct summand of this differential
module.

\end{proof}

\subsection{Application to algebraicity of some stacks}

Our principal use of the above notions is the following technical theorem, which
should be understood as dreaming that
every connection is Fredholm, and deducing that 
$t^{-r}\fg[[t]]dt/G(O)$ is an algebraic stack of finite type.
 
The reader who 
is interested to first see the construction of a large supply
of Fredholm local systems may safely skip ahead to \S \ref{ss:post-artin}.

\begin{thm}\label{t:artin}

Let $G$ be an affine algebraic group with $\fg \coloneqq \Lie(G)$,
and integers $r>0$ and $s\geq 0$.

Let $T$ be a Noetherian affine scheme with a morphism 
$T \to t^{-r} \fg[[t]]dt /t^s \fg[[t]] dt$, and let
$S$ denote the fiber product:

\[
S \coloneqq T \underset{t^{-r} \fg[[t]]dt /t^s \fg[[t]] dt} {\times} t^{-r} \fg[[t]]dt.
\]

Note that $S$ is equipped with an action of the congruence subgroup
$\cK_{r+s} \coloneqq \Ker(G(O) \to G(O/t^{r+s}O)$: indeed, the gauge action
of this group on forms with poles of order $\leq r$ leaves the first $r+s$ coefficients 
fixed.\footnote{The numerics here is the reason we assume $r>0$;
for $r = 0$, we would need to take $\cK_{r+s+1}$ instead of $\cK_{r+s}$ everywhere.}

Note that the structure morphism $S \to t^{-r}\fg[[t]]$ 
defines a $G$-local system on $\o{\cD}_S$.
Suppose that the local system associated to this $G$-local system by the adjoint representation is Fredholm.

Then the prestack quotient $S/\cK_{r+s}$ is an Artin stack, 
and the morphism $S/\cK_{r+s} \to T$ is smooth (in particular, finitely presented).

\end{thm}

\begin{example}

Suppose $G$ is commutative. 
Then the adjoint representation is trivial, and therefore the above connection
on $S$ is trivial, and in particular Fredholm. It follows that:

\[
t^{-r}\fg[[t]]dt/G(O)
\]

\noindent is an Artin stack in this case. 
More generally, if $\Lie(G)$ is a successive extension of trivial representations
(e.g., if $G$ is unipotent),
then the above connection is a successive extension of trivial local systems
and therefore Fredholm, so the same conclusion holds. 

\end{example}

We will prove Theorem \ref{t:artin} in \S \ref{ss:artin-pf-start}-\ref{ss:artin-pf-finish}
below, after some preliminary remarks in 
\S \ref{ss:affine-iso}-\ref{ss:affine-iso-2}.

\subsection{}\label{ss:affine-iso}

The most serious difficulty in proving Theorem \ref{t:artin} is that we need to ``integrate"
from infinitesimal information about the gauge action
(through the Fredholm condition on the adjoint representation) to
global information. We will use two results, Lemmas \ref{l:affine-iso}
and \ref{l:affine-fp}, to do this. Together, they give a way to check that
a morphism of affine spaces is finitely presented: the former
gives a reasonably soft\footnote{Note that there is no hope of
obtaining a purely infinitesimal criterion, because the Jacobian
conjecture is unknown. The criterion below is instead modeled on
the simplest construction of non-linear automorphisms of affine space,
c.f. Example \ref{e:jacobianconj}.}
criterion for such a map to be a closed embedding,
while the latter gives an infinitesimal criterion for a closed embedding
to be finitely presented.

Let $S$ be a base scheme, say affine with $S = \Spec(A)$ for expositional ease only.

We define an \emph{($\aleph_0$-)affine space} $V$ over $S$ to be an $S$-scheme
arising by the following construction:

Suppose $P = P_V \in A\mod^{\heart}$ is a colimit $P = \colim_{i \geq 0} P_i$
with $P_i$ a finite rank projective $A$-module and each structure
map $P_i \to P_j$ being injective with projective quotient.
Then we set $V \coloneqq \Spec(\Sym_A(P))$.
Note that $V(S) = \lim\, P_i^{\vee}$, so informally we should think of
$V$ as the pro-projective $A$-module
$\lim_i P_i^{\vee} \in \Pro(A\mod^{\heart})$.

\begin{construction}

Let $f: V \to W$ be a morphism of affine spaces preserving zero sections
(i.e., commuting with the canonical sections from $S$).
We define the \emph{linearization} $\fL(f): V \to W$ as the
induced morphism by considering $V$ and $W$ as the tangent spaces
of $V$ and $W$ at $0$. 

More precisely, if $P$ (resp. $Q$) is the 
ind-projective module defining $V$ (resp. $W$) and $I_V \subset \Sym_A(P)$
is the augmentation ideal, then $I_V/I_V^2 = P$ (resp. $I_W/I_W^2 = Q$).
Moreover, the map $f^*:\Sym_A(Q) \to \Sym_A(P)$ preserves augmentation
ideal by assumption, so we obtain a morphism:

\[
Q = I_W/I_W^2 \to I_V/I_V^2 = P
\]

\noindent and by covariant functoriality of the assignment $P \mapsto V$,
we obtain the desired map $\fL(f):V \to W$. 

\end{construction}

Suppose $P$ and $Q$ are as above, and suppose
$P = \colim_{i\geq 0} P_i$ and $Q = \colim_{i\geq 0} Q_i$
as in the definition of affine space. 
Let $V_i$ (resp. $W_i$) denote $\Spec(\Sym_A(P_i))$
(resp. $\Spec(\Sym_A(Q_i))$),
so that e.g. $V = \lim V_i$ with each morphism $V \to V_i$ being dominant
(and a fibration with affine space fibers).

\begin{lem}\label{l:affine-iso}

Suppose $f: V \to W$ is a morphism of $S$-schemes preserving $0$.
Suppose that: 

\begin{itemize}

\item For every $i$, we have a (necessarily unique, by dominance) map:

\[
\xymatrix{
V \ar[r]^f \ar[d] & W \ar[d] \\
V_i \ar@{..>}[r]^{f_i} & W_i.
}
\]

\item For every $i$, there is a (necessarily unique) factorization:

\[
\xymatrix{
V_{i} \ar[rr]^{f_{i}-\fL(f_{i})} \ar[d] && W_{i} \\
V_{i-1} \ar@{..>}[urr]_{\alpha_i}.
}
\]

\noindent where for $i = 0$ we use the normalization $V_{i-1} = \Spec(A)$.

\item  Each morphism $\fL(f_i):V_i \to W_i$ is a closed embedding.

\end{itemize}

Then $f$ is a closed embedding. 

\end{lem}

\begin{example}\label{e:jacobianconj}

In the above, the lemma remains true if ``closed embedding" is replaced by 
``isomorphism" everywhere.
Then a toy model for the above lemma is the fact that
any morphism of the form:

\[
\bA^2 \xar{(x,y) \mapsto (x,y+p(x))} \bA^2
\]

\noindent (with $p$ any polynomial) is an isomorphism.

\end{example}

\begin{proof}[Proof of Lemma \ref{l:affine-iso}]

Passing to the limit, it suffices to show that each $f_i$ is a closed embedding.
We will prove this by induction on $i$.

Note that $\alpha_0: \Spec(A) \to W_i$ must be the zero
section, since $f_0 - \fL(f_0)$ preserves zero sections. 
Therefore, $f_0 = \fL(f_0)$, i.e., $f_0$ is linear. 
Since $\fL(f_0)$ is assumed to be a closed embedding, 
this of course implies that $f_0$ is as well.

We now show $f_i$ is a closed embedding, assuming that $f_{i-1}$ is.
Since $V_i$ and $W_i$ are affine, it suffices to show that
pullback of functions along $f_i$ is surjective.

Suppose $\vph$ is a function on 
$V_i$. Since $\fL(f_i)$ is an isomorphism, there is a
function $\psi$ on $W_i$ with $\fL(f_i)^*(\psi) = \vph$.
We obtain:

\[
\vph = f_i^*(\psi) - (f_i - \fL(f_i))^*(\psi).
\]

\noindent Using $\alpha_i$, we see that 
$(f_i - \fL(f_i))^*(\psi)$ descends to a function on $V_{i-1}$,
and therefore is obtained by restriction (along $f_{i-1}$) from a function on
 $W_{i-1}$ by induction.
In particular, $(f_i - \fL(f_i))^*(\psi)$ is obtained by restriction (along $f_i$) from
a function on $W_i$, as desired.

\end{proof}

\subsection{}\label{ss:affine-iso-2}

Above, we gave a way to check that a morphism of
affine spaces is a closed embedding.
We also have the following criterion for checking when a closed embedding
as above is actually finitely presented.

\begin{lem}\label{l:affine-fp}

Suppose $I_1$ and $I_2$ are sets and 
we are given a closed embedding:

\[
i:\bA_T^{I_1} \into \bA_T^{I_2}
\]

\noindent of affine spaces over a Noetherian base $T$,
and suppose that $i$ is compatible with zero sections, i.e.,
the diagram:

\[
\xymatrix{
\bA_T^{I_1} \ar[rr]^i & & \bA_T^{I_2} \\
& T \ar[ul]_0 \ar[ur]^0
}
\]

\noindent commutes.

Then $i$ is finitely presented if and only if its
conormal sheaf $N_{\bA^{I_1}/\bA^{I_2}}^* \in \QCoh(\bA^{I_1})^{\heart}$ is coherent.

\end{lem}

\begin{proof}

We can assume $T$ affine, so $T = \Spec(B)$. Moreover, straightforward 
reductions allow us to assume $T$ is integral (and in all our
applications, $B$ will actually be a field).

The key fact we will need from commutative algebra is that a prime ideal in
any (possibly infinitely generated) polynomial algebra over a Noetherian ring (e.g., $k$)
is finitely generated if and only if it has finite height (in the usual sense of commutative algebra).
Indeed, this result is essentially given by \cite{polynomials} Theorem 4 
(see also \cite{opers} Proposition 4.3, which completes the argument in 
some simple respects).\footnote{
For the reader's
convenience, we quickly sketch the proof down
here in the footnotes.

The main point is the following simple observation:
if $Z_{n+1} \subset \bA^{n+1}$ is a closed and integral
subscheme, then the scheme-theoretic image
$Z_n \coloneqq \ol{p(Z_{n+1})} \subset \bA^n$ of 
$Z_{n+1}$ along the projection
$p:\bA^{n+1} \to \bA^n$ either has smaller codimension,
or else $Z_{n+1} = p^{-1}(Z_n)$. Indeed, 
$Z_{n+1} \subset p^{-1}(Z_n)$, and either the latter
has codimension one less than $Z_{n+1}$, or else
they have equal codimension, and so are equal
by irreducibility.

Now for $Z \subset \bA^{\infty} \coloneqq
\Spec(k[x_1,x_2,\ldots])$ (note that $k$ can be any Noetherian
ring, and we could be less lazy and allow uncountable
indexing sets for our indeterminates)
a closed and integral subscheme corresponding to
a finite height prime ideal, the scheme-theoretic image 
$Z_n \subset \bA^n$ of $Z$ along the projection
$p_n: \bA^{\infty} \to \bA^n$ corresponds to a prime ideal
of the same height for $n$ large enough, and then
the above analysis shows that $Z = p_n^{-1}(Z_n)$. In
particular, $Z$ is defined by a finitely presented ideal.

Similarly, if $Z$ is defined by a finitely presented
ideal, then $Z = p_n^{-1}(Z_n)$ for $n$ large enough,
and our earlier analysis shows that the height
does not increase as we pullback to finite affine spaces;
since finite height prime ideals in $k[x_1,x_2,\ldots]$ 
are finitely presented, one readily deduces that
the height of a prime ideal in this ring is bounded
by the number of generators (since this holds in
each of the Noetherian rings
$k[x_1,\ldots,x_m]$ by Krull's theorem).}

Let $J \subset B[\{x_i\}_{i\in I_2}]$ be the ideal of the closed embedding.
Note that $\bA_T^{I_1}$ is integral (since $T$ is), and therefore $J$ is a prime
ideal.

By assumption, $J/J^2$ is finitely generated. Choose $f_1,\ldots,f_n \in J$
generating modulo $J^2$.

Suppose that $f_1,\ldots,f_n$ lie in $B[\{x_i\}_{i\in I_2^{\prime}}]$
for $I_2^{\prime} \subset I_2$ a finite subset (as we may safely do).
We claim that $J$ is contained in the ideal generated by $\{x_i\}_{i \in I_2^{\prime}}$.
Note that this implies the claim by commutative
algebra, since $J^+$ is a finitely generated prime ideal,
and as noted above, prime subideals of finitely generated prime ideals are finitely generated
themselves.

Let $\ol{J}$ denote the reduction of $J$ modulo $(\{x_i\}_{i \in I_2^{\prime}})$ (the ideal
generated by our finite subset of $x_i$'s).
It suffices to see that $\ol{J} = 0$. 

First, note that $\ol{J} = \ol{J}^2$ by construction of our generators.
Moreover:

\[
\ol{J} \subset 
B[\{x_i\}_{i \in I_2}]/(\{x_i\}_{i \in I_2^{\prime}}) = 
B[\{x_i\}_{i \in I_2 \setminus I_2^{\prime}}]
\]

\noindent is contained in the ideal generated by all the $x_i$: indeed, this
is true for $J$ in the original ring $B[\{x_i\}_{i\in I_2}]$ by the compatibility with zero sections,
implying the corresponding statement for $\ol{J}$.

But now we have:

\[
\ol{J} = \bigcap_{r=1}^{\infty} \ol{J}^r \subset 
\bigcap_{r=1}^{\infty} (\{x_i\}_{i \in I_2\setminus I_2^{\prime}})^r = 0
\]

\noindent giving the claim.

\end{proof}

\subsection{Proof of Theorem \ref{t:artin}}\label{ss:artin-pf-start}

We now prove Theorem \ref{t:artin}. The proof will occupy 
\S \ref{ss:artin-pf-start}-\ref{ss:artin-pf-finish}.

\subsection{}

Note that, since $r>0$, $\cK_{r+s}$ is pro-unipotent.
Because torsors
for unipotent groups are trivial on affine schemes, $\cK_{r+s}$-torsors on affine schemes 
are as well. 
Therefore, the prestack quotient $S/\cK_{r+s}$ is already a sheaf.

\subsection{Formulation of the key lemma}

The following result will be the key input to the proof of Theorem \ref{t:artin}.

\begin{lem}\label{l:artin-fp}

For all sufficiently large integers $N\geq r+s$, the map:

\begin{equation}\label{eq:artin-gauge}
\begin{gathered}
\cK_N \times S \to t^{-r}\fg[[t]]dt \times S \\
(\varga,s) \mapsto (\Gauge_{\varga}(\Gamma(s) dt),s)
\end{gathered}
\end{equation}

\noindent is a finitely presented closed embedding,
where here $s\mapsto \Gamma(s) dt$ is the structure map $S \to t^{-r}\fg[[t]]dt$.

\end{lem}

\begin{proof}

We proceed by steps.
Let $\nabla$ be the connection
on the $\fg_A[[t]]$ defined by 
$\nabla \coloneqq d - [\Gamma,-] dt$ throughout.

\step 

First, we construct $N$:

Note that because $T$ is Noetherian and $S \to T$ is a fibration
with affine space fibers, $S$ satisfies the hypotheses of Lemma \ref{l:fred-lattices-quant}. 
Applying the lemma (with $\Lambda = \fg_A[[t]]$ and 
$\Lambda^{\prime} = t^{-r}\fg_A[[t]]dt$), 
we find $\ell \geq 0$ and an integer $N \geq r+s$ such that
$\nabla|_{t^N\fg_A[[t]]}$ is injective, and for all
$i \geq 0$ we have:

\[
t^{N+i+\ell-r}\fg_A[[t]]dt \subset \nabla(t^{N+i}\fg_A[[t]]) \subset t^{N+i-r}\fg_A[[t]]dt
\] 

\noindent with finite rank projective quotients.
Note that $\nabla(t^{N+i}\fg_A[[t]])$ is a lattice in this case.

Finally, note that we can safely replace $N$ by $N+i$ for any $i$ 
as above and
the conclusion remains, so we may assume $N\geq \ell+1$.

\step \label{st:affineiso-start}

In Steps \ref{st:affineiso-start}-\ref{st:affineiso-finish},
we will apply Lemma \ref{l:affine-iso} to see that \eqref{eq:artin-gauge}
is a closed embedding. 
In the present step, we just set up notation for this.

We abuse notation in letting e.g. $\fg_A[[t]]$ denote the affine space over $S = \Spec(A)$ 
associated with $\fg_A[[t]]$.

To be in the setting of Lemma \ref{l:affine-iso}, we need to have a map 
between affine spaces over a Noetherian base preserving zero. First, we use the 
exponential isomorphism $t^N\fg_A[[t]]dt \isom K_N \times S$
so that our map goes between affine spaces. Moreover, 
since $S \to T$ is an affine space over $T$, our map is between
affine spaces over a Noetherian base (namely, $T$). Our map does not
quite preserve zero: $0 \mapsto \Gamma dt \in t^{-r}\fg_A[[t]]dt$ the
given gauge form. But of course, we can just correct this by subtracting off $\Gamma dt$.

Define:

\[
\Lambda_i \coloneqq \nabla^{-1}(t^{N+i+\ell-r}\fg_A[[t]]dt) \cap t^{N+i}\fg_A[[t]].
\]

\noindent Note that $\Lambda_i$ is a lattice, 
since it maps isomorphically onto the lattice
$t^{N+i+\ell-r}\fg_A[[t]]dt$. 
In particular, it has an associated affine space over $S$,
which we again denote by $\Lambda_i$.

\step\label{st:fp-factor-start}

In the next step, we will show that the composite map:

\begin{equation}\label{eq:fp-comp}
t^N\fg_A[[t]] \xar{\xi \mapsto \Gauge_{\exp(\xi)}(\Gamma dt) - \Gamma dt} 
t^{-r}\fg_A[[t]]dt \xar{\text{proj}.}
t^{-r}\fg_A[[t]]dt/t^{N+i+\ell-r}\fg_A[[t]]dt
\end{equation}

\noindent factors through $t^N\fg_A[[t]]/\Lambda_i$
(therefore satisfying the first condition from Lemma \ref{l:affine-iso}).

In other words, for\footnote{Here $\xi$ and $\eta$ are points of these
schemes with values in some test affine scheme. We suppress the affine
scheme from the notation to keep the notation simple, and we maintain such
abuses throughout.}
$\xi \in t^N\fg_A[[t]]$ and $\eta \in \Lambda_i$, we need
to see that:

\[
\begin{gathered}
\Gauge_{\exp(\xi)}(\Gamma dt) - \Gamma dt)  \text{ and } \\
\Gauge_{\exp(\xi+\eta)}(\Gamma dt) - \Gamma dt)
\end{gathered}
\]

\noindent differ by an element of $t^{N+i+\ell-r}\fg_A[[t]]dt$.
We will do this by explicitly computing 
$ \Gauge_{\exp(\xi+\eta)}(\Gamma dt) - \Gamma dt) $.

In a first motion towards this, we analyze 
the relationship between the ``matrix" exponential
appearing above and the term ``$dg\cdot g^{-1}$"
appearing in the definition of the gauge action.
More precisely, this step will show the following identity
of points of $\fg_A[[t]]dt$:

\begin{equation}\label{eq:dlog}
\begin{gathered}
\big(d \exp(\xi+\eta)\big) \cdot \exp(-\xi-\eta) \in  
\big(d \exp(\xi)\big) \cdot \exp(-\xi) + d \eta + t^{2N+i-1} \fg_A[[t]]dt. \\
\end{gathered}
\end{equation}

\noindent
Note that e.g. $\exp(\xi)^{-1} = \exp(-\xi)$,
hence the appearance of these terms.

By the Tannakian formalism, it suffices to prove
\eqref{eq:dlog} for the general linear group.
Then note that:

\[
\exp(\xi+\eta) \in \exp(\xi)\exp(\eta) + t^{2N+i}\fg_A[[t]].
\]

\noindent Indeed, the Campbell-Baker-Hausdorff formula says:

\[
\exp(\xi)\exp(\eta) = \exp(\xi+\eta+
\underset{\in \, t^{2N+i}\fg_A[[t]]}{\underbrace{
\frac{1}{2}[\xi,\eta] + \text{etc.}
}}
).
\]

\noindent (Here we recall that $\xi \in t^N\fg_A[[t]]$ and 
$\eta \in t^{N+i}\fg_A[[t]]dt$.) Then for 
$a \coloneqq  \xi+\eta$ and 
$b \coloneqq \frac{1}{2}[\xi,\eta] + \text{etc.} \in t^{2N+i}\fg_A[[t]]$,
we have:

\[
\begin{gathered}
\exp(a+b) = \id + (a+b) + \frac{1}{2}(a+b)^2 + \ldots = \\
\id + a + \frac{1}{2} a^2 + \ldots +
b + \frac{1}{2}(ab+ba+b^2) + \ldots \in 
\exp(a) + t^{2N+i}\fg_A[[t]]
\end{gathered}
\]

\noindent as desired.

We now show \eqref{eq:dlog}. Let $g = \exp(\xi)$ and
$h = \exp(\eta)$.
By the above, we have:

\[
\begin{gathered}
\big(d \exp(\xi+\eta)\big) \cdot \exp(-\xi-\eta) \in \\
(d (gh) + t^{2N+i-1} \fg_A[[t]]dt) \cdot 
(h^{-1} g^{-1} + t^{2N+i}\fg_A[[t]]) = \\
d(gh) h^{-1} g^{-1} + d(gh)\cdot t^{2N+i}\fg_A[[t]]dt +
t^{2N+i-1}\fg_A[[t]]dt \cdot h^{-1}g^{-1} + t^{4N+2i-1}\fg_A[[t]]dt \subset \\
d(gh) h^{-1} g^{-1} + t^{2N+i-1}\fg_A[[t]]dt
\end{gathered} 
\]

\noindent where the last line follows from the observations
that $d(gh) \in \fg_A[[t]]dt$ and $h^{-1}g^{-1} \in G(O)$.

Note that $d(gh) h^{-1} g^{-1} = dg\cdot g^{-1} + g (dh \cdot h^{-1}) g^{-1}$.
We compute:

\begin{equation}\label{eq:dexp-approx}
dh \cdot h^{-1} = (d \exp(\eta)) \cdot \exp(-\eta) =
(d\eta + \frac{1}{2}d(\eta^2) + \ldots ) (\id - \eta + \ldots) \in 
d\eta + t^{2(N+i)-1}\fg_A[[t]]dt.
\end{equation}

\noindent Since $g \in \cK_N$ and since 
$d\eta \in t^{N+i-1}\fg_A[[t]]dt$, we have
$g (d\eta) g^{-1} \in d\eta + t^{2N+i-1}\fg_A[[t]]dt$.
Since $\Ad_g$ clearly preserves $t^{2(N+i)-1}\fg_A[[t]]$,
this combines with the above to give:

\[
dg \cdot g^{-1} + g (dh \cdot h^{-1}) g^{-1} \in 
dg \cdot g^{-1} + d\eta + t^{2N+i-1}\fg_A[[t]]
\]

\noindent as was claimed in \eqref{eq:dlog}.

\step

We now complete the factorization claim from the beginning
of Step \ref{st:fp-factor-start}. 

Note that the map \eqref{eq:fp-comp} 
sends $\xi+\eta$ to:\footnote{
For clarity, the notation $\Ad_{\xi+\eta}^n(\Gamma dt)$ means
$[\xi+\eta,[\xi+\eta,[\ldots,[\xi+\eta,\Gamma dt]\ldots]]]$.}

\begin{equation}\label{eq:exp-expansion}
\begin{gathered}
\Gauge_{\exp(\xi+\eta)}(\Gamma dt) - \Gamma dt = 
\Ad_{\exp(\xi+\eta)}(\Gamma dt) - 
d\big(\exp(\xi+\eta)\big) \cdot \exp(-\xi-\eta) - \Gamma dt = \\
[\xi+\eta,\Gamma dt] + \sum_{n=2}^{\infty} \frac{1}{n!}\Ad_{\xi+\eta}^n(\Gamma dt) -
d\big(\exp(\xi+\eta)\big) \cdot \exp(-\xi-\eta) \in \\
[\xi+\eta,\Gamma dt] + \sum_{n=2}^{\infty} \frac{1}{n!}\Ad_{\xi+\eta}^n(\Gamma dt) -
(d\exp(\xi)) \cdot \exp(-\xi) - d\eta + t^{2N+i-1}\fg_A[[t]]dt = \\
-\nabla(\eta) + 
[\xi,\Gamma dt] + 
\sum_{n=2}^{\infty} \frac{1}{n!}\Ad_{\xi+\eta}^n(\Gamma dt) -
(d\exp(\xi)) \cdot \exp(-\xi) + t^{2N+i-1}\fg_A[[t]]dt.
\end{gathered}
\end{equation}

\noindent Here we have applied the calculation \eqref{eq:dlog}.
Note that, since $N \geq \ell+1$, 
we have $2N+i-1 \geq N+i+\ell \geq N+i+\ell-r$, and therefore
in the last equation we can replace $t^{2N+i-1}\fg_A[[t]]dt$
by the larger lattice $t^{N+i+\ell-r}\fg_A[[t]]dt$.

We will show the following points all lie in 
$t^{N+i+\ell-r}\fg_A[[t]]dt$:

Next, observe that $\nabla(\eta) \in t^{N+i+\ell-r}\fg_A[[t]]dt$ by definition of
$\Lambda_i$. 
Moreover, below we will show that that for any $n \geq 2$:

\begin{equation}\label{eq:ad-higherterms}
\Ad_{\xi+\eta}^n(\Gamma dt) \in  
\Ad_{\xi}^n(\Gamma dt) + t^{N+i+1+\ell-r}\fg_A[[t]]dt.
\end{equation}

\noindent Assume this for the moment, and we will conclude the argument.
We then see that
the last term in \eqref{eq:exp-expansion} lies in:

\[
\begin{gathered}
[\xi,\Gamma dt] + 
\sum_{n=2}^{\infty} \frac{1}{n!}\Ad_{\xi}^n(\Gamma dt) -
(d\exp(\xi)) \cdot \exp(-\xi) + t^{N+\ell+i-r}\fg_A[[t]]dt = \\
\Ad_{\exp(\xi)}(\Gamma dt) -
(d\exp(\xi)) \cdot \exp(-\xi) + t^{N+\ell+i-r}\fg_A[[t]]dt = \\
\Gauge_{\exp(\xi)}(\Gamma dt) + t^{N+\ell+i-r}\fg_A[[t]]dt
\end{gathered}
\]

\noindent as desired.

It remains to show \eqref{eq:ad-higherterms}.
More precisely, we claim that for any $n \geq 0$,
we have:

\begin{equation}\label{eq:ad-higherterms-ind}
\Ad_{\xi+\eta}^n(\Gamma dt) \in  
 \Ad_{\xi}^n(\Gamma dt) + t^{Nn+i-r}\fg_A[[t]]dt
\end{equation}

\noindent which would imply the claim of \eqref{eq:ad-higherterms}, 
since if $n \geq 2$, we have $Nn+i-r \geq 2N+i-r \geq  N + i + 1 + \ell - r$ 
(since $N \geq \ell+1$ by definition).

We show \eqref{eq:ad-higherterms-ind} by induction, 
the base case $n = 0$ being obvious.
Assume \eqref{eq:ad-higherterms-ind} holds for $n$,
and we will show it holds for $n+1$.
We apply $\Ad_{\xi+\eta} = \Ad_{\xi} + \Ad_{\eta}$ to both
sides. Obviously the left hand side of \eqref{eq:ad-higherterms-ind}
transforms as desired. For the right hand side, 
note that $\Ad_{\xi+\eta}$ maps $t^{Nn+i-r}\fg_A[[t]]dt$
into $t^{N(n+1)+i-r}\fg_A[[t]dt$, since $\xi+\eta \in t^N\fg_A[[t]]dt$; 
moreover, since $\Ad_{\xi}^n(\Gamma dt) \in t^{Nn - r}\fg_A[[t]]dt$,
$\Ad_{\eta}$ maps it into 
$t^{Nn-r+N+i} \fg_A[[t]]dt = t^{N(n+1)+i-r}\fg_A[[t]]dt$.

\step

We now verify the remaining conditions of Lemma \ref{l:affine-iso}.

The linearization of \eqref{eq:fp-comp} is:

\[
-\nabla: 
t^N\fg_A[[t]]/\Lambda_i \to t^{-r}\fg_A[[t]]dt/t^{N+i+\ell-r}\fg_A[[t]]dt
\]

\noindent which is an embedding: indeed, if $\xi \in t^N\fg_A[[t]]$
with $\nabla(\xi) \in t^{N+i+\ell-r}\fg_A[[t]]dt$, then
we have $\xi \in t^{N+i}\fg_A[[t]]$ by construction, and
$\Lambda_i$ is exactly $ \nabla^{-1}(t^{N+i+\ell-r}\fg_A[[t]]dt) \cap 
t^{N+i} \fg_A[[t]]$.

\step\label{st:affineiso-finish}

Finally, we show that the maps $\alpha_i$ from Lemma \ref{l:affine-iso}
exist.
I.e., we should show that if $\eta \in \Lambda_{i-1}$, then:

\[
\Gauge_{\exp(\xi+\eta)}(\Gamma dt) - \Gamma dt + \nabla(\eta) 
\text{ mod } t^{N+i+\ell-r}\fg_A[[t]]dt
\]

\noindent is independent of $\eta$.

This follows readily from our earlier work.
We compute the left hand side using \eqref{eq:exp-expansion},
and we find that it equals:

\[ 
[\xi,\Gamma dt] +
\sum_{n=2}^{\infty} \frac{1}{n!}\Ad_{\xi+\eta}^n(\Gamma dt) -
(d\exp(\xi)) \cdot \exp(-\xi) + t^{2N+i-2}\fg_A[[t]]dt.
\]

\noindent Now the only terms involving $\eta$ occur inside
that infinite sum, and by \eqref{eq:ad-higherterms}
(substituting $i-1$ for $i$ in \emph{loc. cit}.), we have (for $n \geq 2$):

\[
\Ad_{\xi+\eta}^n(\Gamma dt) \in 
\Ad_{\xi}^n(\Gamma dt) + t^{N+i+\ell-r}\fg_A[[t]]dt
\]

\noindent It remains to observe that $2N+i-2 \geq N+i+\ell-r$,
since $2N+i-2 \geq N+\ell+i-1$ and since we have $r\geq 1$ by
a running assumption.

\step 

Finally, it remains to see that our closed embedding is 
finitely presented. Of course, we do so via Lemma \ref{l:affine-fp};
\emph{loc. cit}. means that we should show that the conormal sheaf is finitely generated.

This is quite easy, in fact.
Our map:

\[
K_N \times S \to t^{-r}\fg_A[[t]]dt
\]

\noindent is given by acting on an $S$-point of the right hand side, so
is obviously $K_N$-equivariant. Therefore, its cornormal sheaf
is also $K_N$-equivariant, so is determined by its restriction to $S$.
Now the cotangent complex of the above map restricted to $S$ is 
dual (is the sense of pro-vector bundles) to the complex:

\[
-\nabla: t^N\fg_A[[t]] \to t^{-r}\fg_A[[t]]dt.
\]

\noindent Therefore, the restriction of the conormal bundle
is dual to the cokernel of this map and therefore a finite rank vector bundle.

\end{proof}

\subsection{}\label{ss:artin-pf-finish}

We now deduce the theorem.

\begin{proof}[Proof of Theorem \ref{t:artin}] \setcounter{steps}{0}

Let $N$ be as in Lemma \ref{l:artin-fp}.
It obviously suffices to show that $S/\cK_N$ is an algebraic space
smooth over $T$, as we will show below.

\step\label{st:fp-claim}

For each integer $m\geq s$, let $S_m$ denote the fiber product:

\[
T \underset{t^{-r} \fg[[t]]dt /t^s \fg[[t]] dt} {\times} t^{-r} \fg[[t]]dt/t^m\fg[[t]]dt.
\]

\noindent Note that as we vary $m$, the structure maps are smooth,
and $\lim_m S_m = S \in \AffSch$.

We claim that the structure map:

\[
S \to S/\cK_N
\]

\noindent factors through $S_m$ for $m\gg 0$ (depending on $N$).

We will show this in Steps \ref{st:fp-start}-\ref{st:fp-finish}.

\step\label{st:fp-start}

In this step, we say in explicit terms what it takes to give a factorization
$S \to S_m \to S/\cK_N$, and in the next step we provide such a construction.

We claim that the data of a section $\sigma_m:S_m \to S$ to
the structure map $\pi_m:S \to S_m$ plus a map
$\varga: S \to \cK_N$ such that the composite:

\[
S \rar{\pi_m \times \varga} 
S_m \times \cK_N \rar{\sigma_m \times \id} 
S \times \cK_m \rar{\act}
S
\]

\noindent is the identity gives rise to a section as above.\footnote{Since every $\cK_N$-torsor
on $S_m$ is trivial, it is easy to see that such a datum is essentially equivalent to giving such a 
factorization.}\footnote{In words: we want to conjugate $s \in S$ to $\sigma_m\pi_m(s)$
to $s$ by means of $\varga(s)$. Our Fredholm assumption and the construction
of $\cK_N$ gives an infinitesimal version of this, and our problem is to integrate to get
a global version.}

Indeed, the map $S_m \times \cK_N \to S$ is equivariant for the
is $\cK_N$-equivariant for the usual action of $\cK_N$ on $S$
and for the action of $\cK_N$ on $S_m \times \cK_N$ defined by the
trivial action on $S_m$ and the left action on $\cK_N$. Passing to the
quotient, we obtain a map $S_m \to S/\cK_N$. 

The composite
map $S \to S_m \to S/\cK_N$ is the composite of $\sigma_m\pi_m$ with
the tautological projection of $S$ to $S/\cK_N$. But the map $\varga$ provides an
obvious way to identify these two maps.

Now observe that any choice of uniformizer $t$ gives an obvious
section $\sigma_m$. Below, we will construct $\varga$ with the desired properties
(for $m$ sufficiently large).

\step\label{st:fp-fred}

Next, we claim that the morphism: 

\[
\cK_N \times S \xar{\act \times p_2} S \underset{T}{\times} S
\]

\noindent is a finitely presented closed embedding.

Observe that: 

\[
S \underset{T}{\times} S =
t^{-r}\fg[[t]]dt \underset{t^{-r}\fg[[t]]dt/t^s\fg[[t]]dt}{\times} S
\]

\noindent and that when we compose the above map with the
structure map to $t^{-r}\fg[[t]]dt \times S$, it sends
$(\varga, s) \in \cK_N \times S$ to $(\Gauge_{\varga}(\omega_s),s)$, where
$s\mapsto \omega_s$ is the structure map $S \to t^{-r}\fg[[t]]dt$.
It suffices to show that this composite map is a finitely presented closed embedding,
but this is the content of Lemma \ref{l:artin-fp}.

\step\label{st:fp-finish}

We now complete the proof of the claim of Step \ref{st:fp-claim} 
by constructing a map $\varga:S \to \cK_N$ with the desired properties.

Observe that $S = \lim_m S \times_{S_m} S$ as an $S \times_T S$-scheme,
where everything in sight is affine. Since we have a morphism:

\[
\underset{m}{\lim} \, S \underset{S_m}{\times} S = 
S \to S \times \cK_N \to S \underset{T}{\times} S
\]

\noindent of $S \times_T S$-schemes (sending $s \in S$ to $(s,1) \in S \times \cK_N$)
and since the latter map is finitely presented, there must exist an integer $m$
and a map:

\begin{equation}\label{eq:noeth-desc}
S \underset{S_m}{\times} S \to S \times \cK_N \to S \underset{T}{\times} S
\end{equation}

\noindent of $S \times_T S$-schemes, and which
induces the diagonal map on the diagonally embedded copy of 
$S \subset S\times_{S_m} S$.

Composing the first map of \eqref{eq:noeth-desc} with
the map $S \xar{\id \times \sigma_m\pi_m} S \times_{S_m} S$
(for any choice of splitting $\sigma_m$, say the one defined by a coordinate $t$),
we obtain a map $S \to S \times \cK_N$ of $S$-schemes, i.e., a map
$\varga:S \to \cK_N$.

It is tautological from the construction that $\varga$ has the desired property.

\step

It follows immediately from the claim in Step \ref{st:fp-claim} 
that $S/\cK_N$ is a stack locally of finite presentation over $T$.
Therefore, it suffices to show that $S_m \to S/\cK_N$ is a smooth covering
and provides an atlas.

Indeed, $S/\cK_N$ obviously has an affine diagonal, and therefore
$S_m \to S/\cK_N$ is affine. Moreover, since $S_m$ and $S/\cK_N$
are locally finitely presented over $T$, this implies that the morphism
is finitely presented (since it is affine and therefore quasi-compact). 
We therefore deduce smoothness by looking at cotangent
complexes. Finally, we easily see that this map is a covering
by base-change along $S \to S/\cK_N$.

\end{proof}

\subsection{Turning points}\label{ss:post-artin}

We have seen in Counterexample \ref{ce:rs-rk1-notfredholm} that the a 
connection can fail to be Fredholm quite severely: for the
$\bA^1$-family of connections there,
there is no stratification of $\bA^1$ such that the restrictions of the connection to strata 
are Fredholm.

However, our next result says that this behavior cannot occur locally.

\begin{thm}\label{t:jumps}

For every pair of positive integers $n,r>0$, 
there exists a constant\footnote{The proof of Theorem \ref{t:jumps} below
gives a very explicit recursive procedure for computing $N_{n,r}$.}
$N_{n,r} \in \bZ^{\geq 0}$ with the following property:

Let $T$ be the spectrum of a field and be equipped with a map:

\[
T \to t^{-r}\gl_n[[t]]dt/t^s\gl_n[[t]]dt
\]

\noindent with $s \geq N_{n,r}$.
Define:

\[
S \coloneqq T \underset{t^{-r} \gl_n[[t]]dt /t^s \gl_n[[t]] dt} {\times} t^{-r} \gl_n[[t]]dt
\]

\noindent and note that there is a canonical rank $n$ family of local systems on $\o{\cD}$
parametrized by $S$.

Then the corresponding family of connections on $S$ is Fredholm.

\end{thm}

\begin{rem}

This result says that any family of connections
with all leading terms the same (to some large enough order) is Fredholm. 

\end{rem}

Combining Theorem \ref{t:artin} with Theorem \ref{t:jumps}, we obtain:

\begin{cor}\label{c:locsys-pol-strat}

For $G$ an affine algebraic group, $s\gg r$ and $\eta$ any
(possibly non-closed) point in $t^{-r} \fg[[t]]dt /t^s \fg[[t]] dt$,
the quotient of the fiber
of $t^{-r}\fg[[t]]dt$ at $\eta$
by $\cK_{r+s}$ is a Noetherian and regular Artin stack 
that is smooth over $\eta$.

\end{cor}

\begin{rem}

The proof is essentially an application of the Babbitt-Varadarajan algorithm 
\cite{babbitt-varadarajan}, which they introduced for finding canonical forms for linear
systems of Laurent series differential equations (which is a somewhat different
concern from ours here). We make no claims to originality in our methods here,
and indeed, the reader who is familiar with \cite{babbitt-varadarajan} will
find the argument redundant; we rather include the argument for the
reader who is not so familiar with their reduction theory for linear differential equations.
  
The proof of the theorem will be given in 
\S \ref{ss:jumps-start}-\ref{ss:jumps-finish} below.

\end{rem}

\subsection{Some counterexamples}

We begin by clarifying why Theorem \ref{t:jumps} is formulated in quite the way it is.
This material may safely be skipped by the reader, though we do not particularly advise this.

First, why can we not always take $N_{n,r} = 0$?

\begin{counterexample}

Suppose that we could always take $N_{2,2} = 0$. It would follow that
any connection of the form:

\[
d+
\begin{pmatrix}
0 & 1 \\
0 & 0 
\end{pmatrix} \frac{dt}{t^2} +
0\cdot \frac{dt}{t} + 
\text{lower order terms}
\]

\noindent is Fredholm (this follows from the case $T = \Spec(k)$ in the theorem). 
However, we claim that this is not the case.

Indeed, let $S = \Spec(k[\lambda])$, and consider the connection:

\[
d + \begin{pmatrix}
0 & 1 \\
0 & 0 
\end{pmatrix} \frac{dt}{t^2} +
\begin{pmatrix}
0 & 0 \\
\lambda(\lambda+1) & 0 
\end{pmatrix} dt.
\]

\noindent We claim that this connection is not Fredholm,
in contradiction with the above. 

When we specialize to $\lambda = \ell \in \bZ \subset k$, we find that
the de Rham complex for this connection has a 0-cocycle
$\begin{pmatrix} t^{\ell} \\ -\ell \cdot t^{\ell+1} \end{pmatrix} \in k((t))^{\oplus 2}$.

Now observe that over the generic point of $S$, the de Rham
complex is acyclic. Indeed, by (the proof of) Lemma \ref{l:extn-scalars},
it suffices to see this after adjoining a square root $t^{\frac{1}{2}}$ of $t$.
Then we can apply a gauge transformation by:

\[
\begin{pmatrix}
t^{\frac{1}{2}} & 0 \\
0 & t^{-\frac{1}{2}} 
\end{pmatrix} \in GL_2(K)
\]

\noindent so that our connection becomes:

\[
d + \begin{pmatrix}
-\frac{1}{2} & 1 \\
\lambda(\lambda+1) & \frac{1}{2}
\end{pmatrix} \frac{dt}{t}.
\]

\noindent The leading term of this regular singular connection
has determinant $-\frac{1}{4} - \lambda(\lambda+1)$ and trace zero, 
and therefore is diagonalizable with eigenvalues 
$\pm \sqrt{\frac{1}{4} + \lambda(\lambda+1)} = \pm (\lambda + \frac{1}{2})$.
But the connection $d+\frac{\eta}{t}dt$ is acyclic for $\eta \not\in\bZ$,
giving the acyclicity here.\footnote{Note that 
$\frac{dt}{t} = 2 \frac{d\sqrt{t}}{\sqrt{t}}$, so one should ``really" view the eigenvalues
as being $2\lambda + 1$, which explains why the
complex is not acyclic for $\lambda$ an integer.}

As in Counterexample \ref{ce:rs-rk1-notfredholm}, this implies that
the connection is not Fredholm: indeed, the de Rham cohomology vanishes
at the generic point, but is not supported on any Zariski closed subvariety
of $\Spec(k[\lambda])$, and therefore cannot be coherent.
 
\end{counterexample}

Why do we need perfect complexes? Can we not arrange
that our de Rham cohomology be the sum of two (cohomologically
shifted) vector bundles?

\begin{counterexample}

For rank 2 connections, the example of \S \ref{ss:gm-ga} readily
shows that the answer to the latter question is ``no."

\end{counterexample}

We include the next counterexample just to indicate how jumps
can occur as we vary our point in $t^{-r}\fg[[t]]dt/t^s\fg[[t]]dt$.

\begin{counterexample}

Let $T = \Spec(F)$ be the spectrum of the localization
of $k[\lambda]$ at $0$. Then we claim that the rank one connection:

\[
d+ \frac{\lambda}{t^2}dt
\]

\noindent is not Fredholm. 

Indeed, if it were, then Lemma \ref{l:fred-lattices} would imply that the complex:

\[
B[[t]] \xar{\nabla} t^{-2} B[[t]]dt
\]

\noindent is perfect. However, after extending scalars
to the fields $B[\lambda^{-1}]$ and $B/\lambda = k$, we obtain different
Euler characteristics, a contradiction. Indeed, for $\lambda$ invertible,
we have seen in Example \ref{e:invertible} that the complex
is acyclic, while for $\lambda = 0$, the complex
is obviously quasi-isomorphic\footnote{The kernel is generated by
$1 \in k[[t]]$, while the cokernel is generated
by the classes of $t^{-1}dt$ and $t^{-2}dt$} to
$k \oplus k[-1]\oplus k[-1]$. 

\end{counterexample}

\subsection{Proof of Theorem \ref{t:jumps}}\label{ss:jumps-start}

We now give the proof of Theorem \ref{t:jumps}.

\subsection{}

We let $T = \Spec(F)$ and let $S = \Spec(A)$ in what follows.


\subsection{Regular singular case}\label{ss:rs-pf}

Suppose $r = 1$. Then we claim that $N_{n,1} = 0$ suffices for any rank $n$
for the local system.

Indeed, suppose that the map $T = \Spec(F) \to t^{-1}\gl_n[[t]]dt/\gl_n[[t]]dt = \gl_n$
is defined by the matrix $\Gamma_{-1} \in \gl_n(F)$.

Since the characteristic polynomial of $\Gamma_{-1}$ has finitely many roots in $F$, there are only finitely many integers $\ell \in \bZ$ such that
$\Gamma_{-1} + \ell \cdot \id_n$ is not invertible. Therefore, by 
Example \ref{e:invertible}, the corresponding connection is Fredholm.

\subsection{The rank 1 case}

At this point, we proceed by induction on $n$, the rank of the local system.

We begin with the case $n = 1$. We claim that we can take
$N_{1,r} = 0$ for all $r$.

Indeed, suppose our map:

\[
T = \Spec(F) \to t^{-r}k[[t]]dt/k[[t]]dt
\]

\noindent is defined by $f_{-r}t^{-r}dt + \ldots + f_{-1} t^{-1}dt + k[[t]]dt$ with $f_i \in F$.

If $f_i = 0$ for $i<-1$, then we are in the paradigm of \S \ref{ss:rs-pf}.
Otherwise, $f_i \neq 0$ for some $i < -1$, and
by Example \ref{e:invertible}, the corresponding connection is Fredholm.

Therefore, we can assume the result true for all $n'<n$ in what follows.

\subsection{}

We will use the following notation in what follows. Given a map:

\[
T \to t^{-r}\gl_n[[t]]dt/t^s\gl_n[[t]]dt
\]

\noindent we let $\Gamma_i \in \gl_n(F), -r\leq i <s$ denote the coefficient
of $t^i$ above.

For all $-r\leq i$, we let $\Gamma_i \in \gl_n(A)$ denote the coefficient
of $t^i$ in the tautological connection on $S$. Note that the abuse
of notation is justified by the fact that for $-r \leq i<s$, the two matrices
we have called $\Gamma_i$ are the same under the embedding
$\gl_n(F) \into \gl_n(A)$.

Finally, we let $\Gamma \in t^{-r}(\gl_n\otimes A)[[t]]$ denote
the connection matrix $\sum \Gamma_i t^i$.

(Note that $\Gamma_{-r}$ is independent of choice
of coordinate up to scaling, justifying the prominent role that it plays below.)

\subsection{Reduction to the case of nilpotent leading term}\label{ss:nilp-redn}

Suppose $N_{n,r} \geq N_{n^{\prime},r^{\prime}}$ for all $n^{\prime}<n$
and $r^{\prime}\leq r$ (where we can suppose the latter
numbers are all known by induction). 
We then claim that the conclusion of the theorem is true for 
a given map $T \to t^{-r}\gl_n[[t]]dt/t^s\gl_n[[t]]dt$
as long as $\Gamma_{-r}$ is not nilpotent.\footnote{Since $A$ is
an integral domain, there is no ambiguity in the meaning
of \emph{nilpotent} here: it just means that $\Gamma_{-r}$ is nilpotent as a matrix.}

Note that we can safely assume $r > 1$ below.

\step\label{st:noeth-desc}

Let $\Gamma_{-r} = f + s$ with $f$ nilpotent and $s$ semisimple commuting
elements of $\gl_n(F)$ be the Jordan decomposition of $\Gamma_{-r}$.

Note that $F^{\oplus n} \simeq \Ker(s) \oplus \Im(s)$. 
Similarly, $\gl_n(F) = \Ker(\Ad_s) \oplus \Im(\Ad_s)$, and every
matrix in $\Ker(\Ad_s)$ preserves this decomposition of 
$F^{\oplus}$ as a direct sum. 

Next, note that $\Im(\Ad_{s}) \subset \Im(\Ad_{\Gamma_{-r}})$.
Indeed, since $\Ad_s: \Im(\Ad_s) \isom \Im(\Ad_s)$
and since $\Ad_f$ is a nilpotent endomorphism of $\Im(\Ad_s)$
commuting with $\Ad_s$, 
$\Ad_{\Gamma_{-r}} = \Ad_s + \Ad_f: \Im(\Ad_s) \to \Im(\Ad_s)$
is an isomorphism.

Note that we obtain a similar decomposition:

\[
A^{\oplus n} = \Ker(\Ad_s) \oplus \Im(\Ad_s)
\]

\noindent where we consider $s$ acting on $A^{\oplus n}$ by extension
of scalars.

\step\label{st:brackets}

Next, we claim that we can apply a gauge transformation by an element of $\cK_1(S)$
so that each matrix $\Gamma_j \in \gl_n(A)$ ($j \geq -r$) preserves each the
$A$-submodules $\Ker(s)$ and $\Im(s)$ of $A^{\oplus n}$.
(This method is very standard, and goes back at least to \cite{sibuya}.)

More precisely, we will construct by induction $g_i \in \cK_i(S)$ with the property
that the first $i+1$ matrices in $\Gauge_{g_i g_{i-1} \ldots g_1}(\Gamma dt)$
have the property above; note that the infinite product of the 
$g_i$ makes sense in $\cK_1$, and therefore
gives a gauge transformation with the desired property.

These elements $g_i$ will have the property, which will 
be important later, that $g_i$ depends only on $\Gamma_{-r},\ldots,\Gamma_{-r+i}$.
In particular, for $1 \leq i <  r+s$, we will have $g_i \in \cK_1(T) \subset \cK_1(S)$.

To construct the $g_i$: applying the gauge transformation by $g_{i-1}\ldots g_1$, we
can assume that $\Gamma_{-r},\ldots,\Gamma_{-r+i-1}$ preserve our 
submodules. 

Then, we can (uniquely) write 
$\Gamma_{-r+i} = M_1 + M_2$ with $M_1 \in \Ker(\Ad_s)$ and
$M_2 \in \Im(\Ad_s)$. Note that $M_2 = [\Gamma_{-r}, C]$ for
some $C \in \gl_n(A)$ (or $C \in \gl_n(F)$ if $i<r+s$), 
since $\Im(\Ad_{s}) \subset \Im(\Ad_{\Gamma_{-r}})$.

We then take $g_i = \exp(-t^i C) \in \cK_i(A)$. Then:

\[
\begin{gathered}
\Gauge_{g_i}(\Gamma dt) = 
\Ad_{\exp(-t^iC)}(\Gamma dt) - 
(d \exp(-t^i C))\cdot  \exp(t^i C) = \\
\Big(\Gamma dt - [t^iC,\Gamma dt] + {\text{terms divisible by }t^{2i-r}}\Big) 
+ 
\Big(t^iC + {\text{terms divisible by }t^{2i-1}}\Big).
\end{gathered}
\]

\noindent Note that $r \geq 1$ by running assumption,
so divisibility by $2i-1$ implies divisibility by $2i-r$, and since
$i \geq 1$, this implies divisibility by $i-r+1$.
Then reducing modulo $t^{i-r+1}\gl_n[[t]]dt$, we obtain:

\[
\Gamma dt - [t^iC,\Gamma_{-r} dt] + t^{i-r+1}\gl_n[[t]]dt.
\]

\noindent Then we have not changed any $\Gamma_j$ for $j\leq i-r$ except
$\Gamma_{i-r}$, and we have changed it to $M_1$ in the above notation.
Since $M_1$ commutes with the semisimple matrix $\Ad_s$, it 
preserves the decomposition of $F^{\oplus n}$ as $\Im(\Ad_s) \oplus \Ker(\Ad_s)$.

\step 

Finally, suppose that $\Gamma_{-r}$ is not a nilpotent matrix.
Then $s \neq 0$, i.e., $\Ker(\Ad_s) \neq A^{\oplus n}$. 

Therefore, we have shown above that we can gauge transform 
$\Gamma$ to be a direct sum of connection matrices $\Gamma^1$ and
$\Gamma^2$ of smaller rank (so $\Gamma^i \in \gl_{n_i}((t))$
for $n_1+n_2 = n$, $n_i \neq 0$).

Since the $g_i$ constructed above lie in $\cK_1$,
$\Gamma^1$ and $\Gamma^2$ have order of pole at most $r$.

Moreover, since the construction of $g_i$ above depended only 
on the leading $i+1$ terms of $\Gamma$, the coefficients
of $t^j$ for $-r \leq j <s $ depends only on the leading terms of $\Gamma$.
That is, the connection $d + \Gamma^1dt + \Gamma^2 dt$ on $S$
has leading $r+s$-terms defined by some map:

\[
T \to t^{-r}\gl_n[[t]]dt/t^s\gl_n[[t]]dt.
\]

Therefore, in this case we obtain the claim
as long as $N_{n,r}$ satisfies the inequalities we said.

\subsection{}

Below, we will show how to handle the case the $\Gamma_{-r}$ 
is a nilpotent matrix. We use the key idea of \cite{babbitt-varadarajan} here:
use the geometry of Slodowy slices to proceed by
induction on $\dim(\Ad_{\Gamma_{-r}})$.

\subsection{Slodowy review}

We briefly review some facts about nilpotent elements
in reductive Lie algebra. 
Let $G$ be a split reductive group ($G = GL_n$ for us) with Lie algebra
$\fg$ over
a ground field of characteristic zero, which we suppress from the notation
(it will be $F$ for us later).

Let $f$ be a nilpotent element in $\fg$, and extend $f$ to an
$\sl_2$-triple $\{e,f,h\} \subset [\fg,\fg]$ via Jacobson-Morozov.
Let $H:\bG_m \to G$ be the cocharacter with derivative $h$. 

Recall that the \emph{Slodowy slice} $\cS_f$ is the
scheme $f+\Ker(\Ad_e)$ 
(considered as a scheme by thinking of $\Ker(\Ad_e)$ as an affine space).

Equip $\fg$ with the $\bG_m$-action:

\[
\lambda \star \xi \coloneqq \lambda^{2}\Ad_{H(\lambda)}(\xi), \hspace{.5cm} 
\lambda \in \bG_m, \xi \in \fg.
\]

\noindent Here the $\lambda^{-2}$ in front is the normal action 
by homotheties of $\bG_m$ on $\fg$. 

Note that this $\star$-action preserves $\cS_f$ because
$\Ad_{H(\lambda)}(f) = \lambda^{-2}f$. 
Moreover, this action contracts $\cS_f$ onto the point $f$, since
$\Ad_h$ has non-negative eigenvalues on $\Ker(\Ad_e)$
(by the representation theory of $\sl_2$).

Next, observe that the $\star$-action preserves the $G$-orbit
through $f$: indeed, for $g\in G$, we have:

\[
\begin{gathered}
\lambda \star \Ad_g(f) = 
\lambda^2 \Ad_{H(\lambda)} \Ad_g(f) =
\Ad_{H(\lambda)} \Ad_g(\lambda^2 f) = \\
\Ad_{H(\lambda)} \Ad_g \Ad_{H(\lambda)^{-1}} (f) =
\Ad_{H(\lambda) g H(\lambda)^{-1}}(f).
\end{gathered}
\]

More generally, the $\star$-action preserves the $G$-orbit through
\emph{any} nilpotent element $f^{\prime}$. Indeed,
first notice that $\lambda f \in G \cdot f$ for any $\lambda$, since
$\Ad_{H(\sqrt{\lambda})}(f) = \lambda f$. Here nothing about $f$ is
special, so $\lambda f^{\prime} \in G \cdot f^{\prime}$ for any $\lambda$.
Then: 

\[
\lambda \star f^{\prime} = \lambda^{2}\Ad_{H(\lambda)}(f^{\prime}) \in
G \cdot \Ad_{H(\lambda)}(f^{\prime}) = G \cdot f^{\prime}.
\]

With these preliminary geometric observations, we now deduce:

\begin{lem}[c.f. \cite{babbitt-varadarajan} \S 2]\label{l:nilp-orbits}

For any field-valued point $f^{\prime} \in \cS_f$ nilpotent, 
$\dim(G \cdot f^{\prime}) (= \dim \Im(\Ad_{f^{\prime}}))$ is
greater than or equal to $\dim(G \cdot f) (=\dim \Im(\Ad_f))$,
with equality if and only if $f^{\prime} = f$. 

\end{lem}

\begin{proof}

If $f^{\prime}$ is as above, we must have
$G\cdot f \subset \overline{G \cdot f^{\prime}}$ (the orbit closure),
since the $\star$-action preserves $G \cdot f^{\prime}$ and
contracts $\cS_f$ onto $f$. This implies the claim
on dimensions.

Since $G\cdot f^{\prime}$ is open in its closure, 
if $\dim(G\cdot f) = \dim(G\cdot f^{\prime})$ then we must have
$G \cdot f = G \cdot f^{\prime}$. Therefore, we should see that
$G \cdot f \cap \cS_f = f$.

Observe that the $G$-orbit $G\cdot f$ through $f$ and
$\cS_f$ meet transversally at $f$ \textemdash{} this follows from the identity:

\[
\fg = \Ker(\Ad_e) \oplus \Im(\Ad_f)
\] 

\noindent (which is again a consequence of the representation theory
of $\sl_2$). Because the $\star$-action on 
$\fg$ preserves $G\cdot f$ and $\cS_f$ and contracts onto $f$,
it follows that $G \cdot f$ and $\cS_f$ meet only at the point $f$.

\end{proof}

\subsection{Nilpotent leading term} 

We now treat the case of nilpotent leading term.
At this point, the reader may wish to skip ahead to
\S \ref{ss:bv-example}, where we indicate how the reduction
theory works in the simplest non-trivial case.

\subsection{}

Let $\Gamma_{-r} = f \in \gl_n(A)$ be nilpotent,
and let $\delta \coloneqq \dim G\cdot f$, 
where $G\cdot f \subset \fg \times_{\Spec(k)} \Spec(F)$ as a scheme.
We will proceed by descending induction on
$\delta$.

More precisely, below we will construct 
$N_{n,r,\delta}$ with the property that for 
$s \geq N_{n,r,\delta}$ and $\Gamma_{-r}$ of the above type,
the conclusion of the theorem holds. By induction,
we can assume that knowledge of $N_{n^{\prime},r^{\prime}}$
for all $n^{\prime}<n$, and can assume the knowledge of
$N_{n,r^{\prime},\delta^{\prime}}$ for all $\delta^{\prime}>\delta$.

Note that $\delta>\dim(G)$, the hypotheses are vacuous, 
giving the base case in the induction. Moreover,
this makes clear that if we accomplish the above construction 
of $N_{n,r,\delta}$, we have completed the proof of the theorem: 
combining this with \S \ref{ss:nilp-redn}, we see that if we take:

\[
N_{n,r} \coloneqq \max \big\{
\{N_{n^{\prime},r}\}_{n^{\prime}<n},
\{N_{n,r,\delta}\}_{\delta\leq \dim(G)}
\big\}
\]

\noindent we have obtained a constant satisfying the conclusion of the theorem.

\subsection{}

We take an $\sl_2$-triple $\{e,f = \Gamma_{-r},h\} \in \sl_n(F)$ as
before.
We obtain an identification:

\[
\gl_n(F) = \Ker(\Ad_e) \oplus \Im(\Ad_f)
\]

\noindent and similarly for $\gl_n(A)$. Our $\sl_2$-triple 
integrates to a map at the level of group schemes over $F$, 
and in particular we let $H:\bG_{m,F} \to SL_{m,F}$ integrate $h$. 

By the same method as in Step \ref{st:brackets} of \S \ref{ss:nilp-redn},
we may perform a gauge transformation by an element of $\cK_1(F)$
to obtain a new connection matrix 
$\Gamma^{\prime} = \sum_{i \geq -r} \Gamma_i^{\prime} t^i 
\in t^{-r}\gl_n[[t]]$ with $\Gamma_{-r}^{\prime} = \Gamma_{-r}$ and
$\Gamma_i^{\prime} \in \Ker(\Ad_e)$ for all $i>-r$. 
Moreover, by the construction of \emph{loc. cit}., 
$\Gamma_i^{\prime} \in \gl_n(F) \subset \gl_n(A)$ (i.e., it ``has constant coefficients")
for $-r\leq i <s$.

Since the first $r+s$ terms of our matrix are constant, we might as well
replace $\Gamma$ by $\Gamma^{\prime}$ and thereby assume that the 
$\Gamma_i \in \Ker(\Ad_e)$ for all $i>-r$ (just to simplify the notation
with $\Gamma^{\prime}$).

\subsection{}

For $j \in \bZ$, let $\Gamma_i^{(j)}$ be the degree $j$ component
of $\Gamma_i$ with respect to the grading defined by $H$, so:

\[
\begin{gathered}
[h,\Gamma_i^{(j)}] = j \Gamma_i^{(j)}, \text{ or equivalently, }
\Ad_{H(\lambda)}(\Gamma_i^{(j)}) = \lambda^j \Gamma_i^{(j)} \hspace{.25cm} 
(\lambda \in \bG_m).
\end{gathered}
\]

\noindent For example, $\Gamma_{-r} = \Gamma_{-r}^{(-2)}$. Note that
for $i>-r$, since $\Gamma_i \in \Ker(\Ad_e)$, we have 
$\Gamma_i^{(j)} \neq 0$ only for $j\geq 0$.

Let $\alpha \in \bQ^{>0}$ be defined as:\footnote{
We use the standard convention that $\alpha = \infty$ 
if the set indexing this minimum
is empty.}

\[
\alpha \coloneqq \underset{
\Gamma_i^{(j)} \neq 0, \, -r<i<s
}
{\min} \, \frac{i+r}{j+2}.
\]

\noindent We treat the cases $\alpha \geq \frac{r-1}{2}$ and $\alpha<\frac{r-1}{2}$ separately 
in \S \ref{ss:nilp-type1} and \S \ref{ss:jumps-finish}
respectively. 




\subsection{}

At this point, we impose our conditions on $N_{n,r,\delta}$. 
The reader may safely skip these formulae right now:
we are only including them now for the sake of concreteness.

Let $j_{\delta}$ the maximal degree in the grading of $\gl_n$ defined by
$H$ for \emph{some} nilpotent $f$ with orbit having dimension 
$\delta$.\footnote{Of course, $j_{\delta}$ is bounded by 
the maximal degree in the principal grading of $\gl_n$. We could
replace $j_{\delta}$ everywhere by this constant, but we are just
trying to be somewhat more economical by retaining the
dependence on $\delta$.}
This constant is well-defined e.g. because there are only finitely many 
nilpotent orbits.

Below, we will show that as long as:

\begin{equation}\label{eq:ndelta-eqns}
\begin{gathered}
N_{n,r,\delta} \geq -1+ j_{\delta}\cdot \frac{r-1}{2}, \\
N_{n,r,\delta} \geq
 (j_{\delta}+2) \cdot j_{\delta} \cdot \frac{r-1}{2} +
N_{n^{\prime},(j_{\delta}+2)\cdot (r-1) + 1}
\text{ for }n^{\prime}<n, \text{ and} \\
N_{n,r,\delta} \geq  (j_{\delta}+2) \cdot j_{\delta} \cdot \frac{r-1}{2} + 
N_{n,(j_{\delta}+2)\cdot (r-1) + 1,\delta^{\prime}} \text{ for }
\delta<\delta^{\prime} \leq \dim GL_n.
\end{gathered}
\end{equation}

\noindent the connection is Fredholm. Note that there are only
finitely many conditions listed here, and they are of the
desired inductive nature, and so if we can show this,
then we have completed the proof of the theorem.

\subsection{}\label{ss:nilp-type1}

First, suppose that $\alpha \geq \frac{r-1}{2}$.

By Lemma \ref{l:extn-scalars}, it suffices to show our connection is Fredholm
after adjoining $t^{\frac{1}{2}}$.

We then claim that after performing a gauge transformation by 
$H(t^{\frac{1-r}{2}})$, we obtain a connection with regular singularities.

Indeed, first note that $d\log(H(t^{\frac{1-r}{2}}))$ has a regular singularity,
so it suffices to show that $\Ad_{H(t^{\frac{1-r}{2}})}(\Gamma)$ has a pole
of order $\leq 1$.

We clearly have:

\[
\Ad_{H(t^{\frac{1-r}{2}})}(\Gamma_{-r}t^{-r}) =
t^{r-1} \cdot \Gamma_{-r}t^{-r} = 
 \Gamma_{-r} t^{-1}.
\]

Then for $-r<i<s$, we have:

\[
\Ad_{H(t^{\frac{1-r}{2}})}(\Gamma_i t^i) = 
\sum_j \Gamma_i^{(j)} t^{i+\frac{j(1-r)}{2}}.
\]

\noindent We then claim that $\Gamma_i^{(j)} \neq 0$ and
the definition of $\alpha$ implies that:

\[
i+\frac{j(1-r)}{2} \geq -1
\]

\noindent as desired. Indeed, for any such pair $(i,j)$, we
have:

\[
\frac{i+r}{j+2} \geq \alpha \geq \frac{r-1}{2}
\]

\noindent and rearranging terms we get:

\[
i+r \geq j\frac{r-1}{2} + (r-1)
\]

\noindent which is obviously equivalent to the desired inequality.

Finally, for $s\leq i$, recall from \eqref{eq:ndelta-eqns} that
$s > -1+j\frac{r-1}{2}$ for any degree $j$ appearing
in the grading of $\gl_n$ defined by $H$.
Therefore, for $i$ in this range, we have:

\[
i + \frac{j(1-r)}{2} \geq s + \frac{j(1-r)}{2} \geq -1
\]

\noindent as desired.

Therefore, we see that the resulting connection is regular singular. Moreover,
the above shows that the leading term of this regular singular
connection is determined entirely by
$\Gamma_i$ for $i<s$ and therefore
has coefficients in $F \subset A$ (noting that $d\log(H(t^{\frac{r-1}{2}}))$ 
also has coefficients in $F \subset A$). This completes the argument
by \S \ref{ss:rs-pf}.

\subsection{}\label{ss:jumps-finish}

Finally, we treat the case where $\alpha < \frac{r-1}{2}$.

Applying Lemma \ref{l:extn-scalars} again,
it suffices to see that our connection is Fredholm after adjoining
$t^{\alpha}$.  

Performing a gauge transformation by $H(t^{-\alpha})$, we
claim that we obtain a connection with a pole of order $\leq -r+2\alpha$.
Indeed, this follows exactly as in \S \ref{ss:nilp-type1}:

\begin{itemize}

\item The $d\log$ term only affects the coefficient of $t^{-1}$.

\item 

$\Ad_{H(t^{-\alpha})} \Gamma_{-r} t^{-r} = \Gamma_{-r} t^{-r+2\alpha}$.

\item

For $-r<i<s$, we have

$\Ad_{H(t^{-\alpha})} (\Gamma_i t^i) = 
\sum_j \Gamma_i^{(j)} t^{i-j\alpha}$, and for 
$\Gamma_i^{(j)} \neq 0$, we have:

\[
\frac{i+r}{j+2} \geq \alpha \Rightarrow 
i+r \geq j\alpha + 2\alpha \Rightarrow
i-j\alpha \geq -r+2\alpha.
\]

\item For $s\leq i$, the same argument as in
\S \ref{ss:nilp-type1} shows that the corresponding terms can
only change the coefficients of $t^j$ for $j \geq -1$.

\end{itemize}

Note that by assumption, we presently have $-r+2\alpha < -1$.

For any pair $(i,j)$ with $-r<i<s$, $\Gamma_i^{(j)} \neq 0$,
and:

\[
\frac{i+r}{j+2} = \alpha
\]

\noindent we have:

\[
i - j\alpha = i - (j+2)\alpha + 2\alpha = -r + 2\alpha.
\]

\noindent Therefore, the leading term (i.e., $t^{-r+2\alpha}$-coefficient) of our resulting 
connection is:

\begin{equation}\label{eq:slodowy}
\Gamma_{-r} + \sum_{\frac{i+r}{j+2} = \alpha, \, -r<i<s} \Gamma_i^{(j)}.
\end{equation}

Note that this sum cannot be $\Gamma_{-r}$: indeed, there is at least
one summand on the right (since $\alpha <\frac{r-1}{2} < \infty$); 
moreover, a summand $\Gamma_i^{(j)}$ contributes
purely in degree $j$ (with respect to the grading defined by $H$), 
and since for given $j$ there is at most one pair $(i,j)$ 
with $\frac{i+r}{j+2} = \alpha$, there can be no cancellations within the given degree $j$. 

Therefore, this sum lies in the Slodowy slice for $\Gamma_{-r}$. 
By Lemma \ref{l:nilp-orbits}, it is either nilpotent with a larger dimensional
orbit that $\Gamma_{-r}$, or it is not nilpotent, in which case we should bring
\S \ref{ss:nilp-redn} to bear. 

\begin{rem}

The reader should be better off arguing for themselves that we are done at this
point, since this is essentially clear. 
But we include the final bit of the argument (difficult though it may be to
read) for completeness.

\end{rem}

Let $\alpha = \frac{a}{b}$ with $a$ and $b$ coprime.
Note that $b$ is bounded in terms of $\delta$ alone:
$b$ divides $j+2$ for some $0\leq j \leq j_{\delta}$, and
therefore $b \leq j_{\delta}+2$.

Therefore, the order of pole in the extension $A((t^{\frac{1}{b}}))$ of $A((t))$
of our resulting connection is bounded in terms of $\delta$. More precisely,
we have:

\[
t^{-r+2\alpha} dt = b t^{-r+2\alpha+1-\frac{1}{b}} d t^{\frac{1}{b}}
\]

\noindent so that our pole has order (in the $t^{\frac{1}{b}}$-valuation) 
at most:

\[
br-2a-b+1 \leq b(r-1) + 1 \leq (j_{\delta}+2) \cdot (r-1) + 1.
\]

Moreover, the coefficients before (i.e., with lower order of
zero/higher order of pole):

\[
t^{s-j_{\delta}\alpha} dt = (t^{\frac{1}{b}})^{bs-b j_{\delta} \alpha+b-1} d t^{\frac{1}{b}} 
\]

\noindent lie in $\gl_n(F) \subset \gl_n(A)$, as is clear from 
using our same $\Ad_h$-eigenspace decomposition of
$\gl_n$. Note that we can bound
the order of zero here from below independently of $\alpha$:

\[
\begin{gathered}
bs-b j_{\delta} \alpha+b-1 >
bs - b j_{\delta} \cdot \frac{r-1}{2}  + b - 1 =
b(s+1)-1 - b j_{\delta} \cdot \frac{r-1}{2}  \geq \\
s+1 -1 - b j_{\delta} \cdot \frac{r-1}{2} \geq 
s - (j_{\delta}+2) \cdot j_{\delta} \cdot \frac{r-1}{2}.
\end{gathered}
\]

Therefore, if:

\[
\begin{gathered}
N_{n,r,\delta} \geq
 (j_{\delta}+2) \cdot j_{\delta} \cdot \frac{r-1}{2} +
N_{n^{\prime},(j_{\delta}+2)\cdot (r-1) + 1}
\text{ for }n^{\prime}<n, \text{ and} \\
N_{n,r,\delta} \geq  (j_{\delta}+2) \cdot j_{\delta} \cdot \frac{r-1}{2} + 
N_{n,(j_{\delta}+2)\cdot (r-1) + 1,\delta^{\prime}} \text{ for }
\delta<\delta^{\prime} \leq \dim GL_n.
\end{gathered}
\]

\noindent then we obtain the claim. Indeed, the former condition
takes care of the case where the leading term of 
$\Gauge_{H(t^{\alpha})}(\Gamma dt)$ has a non-nilpotent leading term
by \S \ref{ss:nilp-redn}; and the latter condition takes care of the
case where it has nilpotent leading term, by the above analysis.

\subsection{Example: n = 2, r = 2}\label{ss:bv-example}

It is instructive to see what is going on above in a simpler setup.
Suppose that we have a rank 2 connection with a pole of order 2
and with nilpotent leading term. We claim that $N_{2,2} = 1$ 
suffices.\footnote{This is a more precise estimate than
we obtained from the crude estimates of \S \ref{ss:jumps-finish}.}

Up to (constant) change of
basis, we can write:

\[
\Gamma dt = 
\begin{pmatrix} 0 & 0 \\ 1 & 0 \end{pmatrix} \frac{dt}{t^2} +
\Gamma_{-1} \frac{dt}{t} + \Gamma_0 dt + \text{lower order terms}.
\]

Suppose that 
$\Gamma_{-1} = \begin{pmatrix} a & b \\ c & d \end{pmatrix}$
and $\Gamma_0 = \begin{pmatrix} \ast & e \\ \ast & \ast \end{pmatrix}$
where $\ast$ is indicates that the term is irrelevant for our needs below.

There are two cases: when $b = 0$
and when $b \neq 0$. Note that these correspond 
to \S \ref{ss:nilp-type1} 
and \ref{ss:jumps-finish} respectively.

In the former case, we apply a gauge transformation by:

\[
\begin{pmatrix}
t^{-\frac{1}{2}} & 0 \\
0 & t^{\frac{1}{2}}
\end{pmatrix}
\]

\noindent to obtain:

\[
\begin{pmatrix} 
a + \frac{1}{2} & e \\ 1 & d - \frac{1}{2} 
\end{pmatrix} \frac{dt}{t} + \text{lower order terms}.
\]

\noindent So we see that we get a regular singular connection
whose leading terms depends only on the fixed matrices
$\Gamma_{-2}$, $\Gamma_{-1}$, and $\Gamma_0$, and therefore
it is Fredholm. (Note that if $b$ were $\neq 0$, it would contribute
to a non-zero coefficient of $\frac{dt}{t^2}$ here.)

If $b \neq 0$, then we instead apply a gauge transformation
by the matrix:

\[
\begin{pmatrix}
t^{-\frac{1}{4}} & 0 \\
0 & t^{\frac{1}{4}}
\end{pmatrix}
\]

\noindent to obtain a connection of the form:

\[
\begin{pmatrix}
0 & b \\
1 & 0
\end{pmatrix} \frac{dt}{t^{\frac{3}{2}}} + 
\begin{pmatrix}
a+\frac{1}{4} & 0 \\
0 & d-\frac{1}{4}
\end{pmatrix} \frac{dt}{t} +
\text{lower order terms}.
\]

\noindent Since $b \neq 0$, the leading term of this connection
is semisimple. Therefore, our connection can be written as
a sum of two rank 1 connections whose polar parts are
constant. Therefore, the connection is Fredholm by reduction
to the rank 1 case.

\subsection{Coda}

We need some refinements to the above results, making them
more effective. These refinements will be used later in the paper, and
may be safely skipped for the time being.

\subsection{}

Let $G$ be an affine algebraic group.
For $r>0$ and $s \geq 0$,  Theorems \ref{t:artin} and \ref{t:jumps}
imply that the geometric fibers of the map:

\[
t^{-r}\fg[[t]]dt/\cK_{r+s} \to t^{-r}\fg[[t]]dt/t^s\fg[[t]]dt
\]

\noindent are Artin stacks for $s\gg 0$.

\begin{prop}\label{p:bdd-dimns}

The dimensions of these fibers are bounded uniformly in terms of
$G$ and $r$. 

\end{prop}

Note that these Artin stacks are smooth with tangent complex:

\[
t^{r+s} \fg[[t]] \rar{-\nabla} t^{s}\fg[[t]]dt 
\]

\noindent for $\Gamma dt \in t^{-r}\fg[[t]]dt$ a point and 
$\nabla \coloneqq d - \Gamma dt$. The dimension of 
the Artin stacks above is the Euler characteristic of this complex
(for $\Gamma dt$ any connection in the relevant fiber).
Therefore, Proposition \ref{p:bdd-dimns} follows from the next result:

\begin{lem}\label{l:bdd-dimns-lin}

There is a constant $C_{n,r}$ with the following property:

Let $(V,\nabla)$ be any rank $n$ differential module over\footnote{Of course,
$k$ can be any field of characteristic zero, not just our ground field. The result also
immediately extends to any Fredholm connection over any
commutative ring.} $k((t))$,
and let $L \subset V$ be a $k[[t]]$-lattice such that $\nabla$ maps into
$t^{-r}L dt$.

Then the absolute value of the Euler characteristic of the complex:

\[
\nabla: \Lambda \to t^{-r}\Lambda dt
\]

\noindent is at most $C_{n,r}$.

\end{lem}

\begin{proof}

First, note that $\dim(\Ker(\nabla))$ is at most $n$, and therefore
we need to bound $\dim(t^{-r}\Lambda dt/\nabla(\Lambda))$ 
in terms $n$ and $r$ alone.

By \cite{frenkel-zhu} Lemma 8, we have:

\[
\codim(\nabla^{-1}(\La dt) \cap \Lambda \subset \nabla^{-1}(\La dt)) \leq 
\codim(\nabla^{-1}(\Lambda dt) \cap \Lambda \subset \La).
\]

\noindent Moreover, the right hand side is bounded in terms of $r$ and $n$:
it is less than or equal to $rn$. Indeed, this follows from the embedding:

\[
\La /(\nabla^{-1}(\Lambda dt) \cap \Lambda ) \overset{\nabla}{\into}
t^{-r}\Lambda dt/\Lambda dt 
\]

\noindent so it is bounded by $r n$.

Now recall (see e.g. \cite{bbe} \S 5.9) that $H_{dR}^1(V,\nabla)$ is at most
$n$-dimensional. Therefore, $\nabla(\nabla^{-1}(\Lambda dt)) \subset \Lambda dt$
as codimension at most $n$. Combining this with the above, we see that:

\[
\nabla(\nabla^{-1}(\Lambda dt) \cap \Lambda ) \subset \Lambda dt
\]

\noindent has codimension $\leq (r+1) n$, and therefore has codimension
$\leq (2r+1) n$ in $t^{-r}\Lambda dt$.

\end{proof}

\subsection{}

Finally, we record the following consequence of
the Babbitt-Varadarajan method. Note that this
result will not be needed until \S \ref{s:z-geom}.

\begin{prop}\label{p:coda-2}

In the setting of Theorem \ref{t:jumps}, 
the universal connection $\nabla$ on $S = \Spec(A)$ 
(where $S$ is a fiber at a field-valued point
of the map $t^{-r}\gl_n[[t]]dt \to (t^{-r}\gl_n[[t]]dt/t^s\gl_n[[t]]dt)$)
has $H_{dR}^1$ which can be generated by at most $n$ elements
(as an $A$-module).

\end{prop}

\begin{proof}

We will show this result explicitly in our favorite special cases, and then
use the Babbitt-Varadarajan method to show that it follows
in general.

\step

First, if the connection has a pole of order $>-1$ and the leading
term is invertible, then the complex of
de Rham cohomology vanishes, and the claim is stupid.

\step

Next, we treat the case where $r = -1$. We will show that
for any connection of the form:
 
\[
d+\Gamma_{-1} \frac{dt}{t} + \text{ lower order terms}
\]

\noindent where $\Gamma_{-1} \in \gl_n(k)$,
$H_{dR}^1$ can be generated by at most $\dim(G)$ elements
(we are using the
ground field $k$ just for simplicity: this holds for any field).
Note that applying this to the universal such connection gives the claim
of the proposition in this case.

Indeed, as in Example \ref{e:rs-evalues}, let
$\Lambda = A[[t]]^{\oplus n}$, and
recall that for the associated $t$-adic filtrations,
$\nabla$ shifts degrees by one and on the
induces the map:

\[
t^N\Lambda/
t^{N+1} \Lambda = A^{\oplus n} 
\xar{\Gamma_{-1} + N\cdot \id}
A^{\oplus n} =
t^{N-1} \Lambda dt/ t^N \Lambda dt
\]

\noindent on passing to the associated graded.
Note that the cokernel of this map is a free $A$-module
of rank equal to the dimension of the ($-N$)-eigenspace
of $\Gamma_{-1} \in \gl_n$ 
(which is a matrix with field-valued coefficients). 

Splitting the surjection:

\[
(t^{N-1} \Lambda dt/ t^N \Lambda dt) \onto 
\Coker(\Gamma_{-1} + N\cdot \id) \in A\mod^{\heart}
\]

\noindent we see that we can find an $A$-free
module $\sF$ of rank $j\leq n$ that is a direct
summand of $A((t))^{\oplus n}$, equipped with
a basis $v_1,\ldots,v_j$ whose symbols
provide a basis for
$ \oplus_N \Coker(\Gamma_{-1} + N\cdot \id) $
(e.g., we can choose the $v_j$ to lie in
$t^{N_j} A^{\oplus n} \subset t^{N_j} A[[t]]^{\oplus n}$
for some $N_j$).

We claim that $\sF \subset A((t))^{\oplus n}dt$
maps surjectively onto $H_{dR}^1$. 

Indeed, suppose that $\omega \in A((t))^{\oplus n} dt$. 
We can suppose that $\omega \in t^{N-1}\Lambda dt$
for some fixed $N$. Then we can find $\vph_0 \in \sF$
so that the images of $\omega$ and $\vph_0$ in 
$\Coker(\Gamma_{-1} + N\cdot \id)$ coincide.
By our earlier symbol computation, this means that:

\[
\omega - \vph_0 \in t^{N}\Lambda dt + \nabla(t^N\Lambda).
\]

\noindent Therefore, the de Rham cohomology class represented
by $\omega - \vph_0$ can be represented by a class in $t^N\Lambda dt$.

By induction, we see that we can find a vector 
$\vph_{\ell} \in \sF$ so that the de Rham cohomology class
of $\omega - \vph_{\ell}$ can be represented by an element
of $t^{N+\ell} \Lambda dt$ for any $\ell$. But for $\ell$ large
enough, $t^{N+\ell}\Lambda dt$ maps onto zero in de Rham cohomology,
so $\omega$ and $\vph_{\ell}$ represent the same class in 
de Rham cohomology, as desired.

\step 

The general case is a consequence of the proof of Theorem \ref{t:jumps}:

Babbitt-Varadarajan reduction theory, as applied
above, shows that after possibly adjoining $t^{\frac{1}{e}}$,
any connection with sufficiently fixed leading terms
can be written as a sum of connections gauge
equivalent to rank 1 connections with fixed polar parts and
regular singular connections with constant polar part.

Since de Rham cohomology
is a summand of the de Rham cohomology
obtained after adjoining $t^{\frac{1}{e}}$, we obtain the result
(since if a module is generated by $n$ elements, any direct
summand of it certainly is also).

\end{proof}


\section{Tameness}\label{s:tame}


\subsection{}

In this section, we study the action of $G(O)$ on gauge forms,
and show that through the lens of homological algebra, this
action has many favorable properties.

In \S \ref{ss:gpact-start}-\ref{ss:gpact-finish}, we study some general
aspects of (weak) actions of 
infinite type algebraic groups on categories. In general, this subject exhibits
pathological behavior that is not present in the finite type situation.

However, we will introduce the notion of a \emph{tame}\footnote{Given
the interests of this paper, we explicitly state that this
word has nothing to do with the notion of tame ramification in Galois/$D$-module
theory.} 
action of such a group, where such behavior is not present. The basic dichotomy
the reader should keep in mind is that the regular representation is tame,
while the trivial representation is not. 

Then the main result of this section, Theorem \ref{t:tame}, is that the 
action of $G(O)$ on gauge forms is tame. 

\subsection{}\label{ss:gpact-start}

Let $\cG$ be an affine group scheme (possibly of infinite type) over $k$.

Then $\QCoh(\cG)$ inherits a usual \emph{convolution} monoidal structure by
$\sF \otimes \sG \coloneqq m_*(\sF \boxtimes \sG)$ for $m: \cG \times \cG \to \cG$
the multiplication. This convolution structure commutes with colimits in each variable
separately, and therefore defines a structure of algebra on 
$\QCoh(\cG) \in \DGCat_{cont}$. 

For $\sC \in \QCoh(\cG)\mod \coloneqq \QCoh(\cG)\mod(\DGCat_{cont})$, we
define the \emph{(weak) invariant} and \emph{(weak) coinvariant} categories as:

\[
\begin{gathered}
\sC^{\cG,w} \coloneqq \underset{\bDelta} \lim \, \big(\sC \rightrightarrows 
\QCoh(\cG) \otimes \sC \ldots\big) = 
\TwoHom_{\QCoh(\cG)\mod}(\Vect,\sC)
\\
\sC_{\cG,w} \coloneqq \underset{\bDelta^{op}} \colim \, \big( \ldots
\QCoh(\cG) \otimes \sC \rightrightarrows \sC = 
\sC \underset{\QCoh(\cG)}{\otimes} \Vect\big).
\end{gathered}
\]

\noindent Here $\Vect$ is induced with the \emph{trivial} $\QCoh(\cG)$-action,
i.e., it is induced by the monoidal functor $\Gamma: \QCoh(\cG) \to \Vect$.

\begin{example}\label{e:inv-desc}

If $\cY$ is a prestack with an action of $\cG$, then 
$\QCoh(\cG)$ acts on $\QCoh(\cY)$. 
Moreover, $\QCoh(\cY)^{\cG,w}$ identifies with 
$\QCoh(\cY/\cG)$.\footnote{The main point here is that
$\QCoh(\cY \times \cG) = \QCoh(\cY) \otimes \QCoh(\cG)$, which 
is a general consequence of the fact that $\QCoh(\cG)$ is dualizable
(being compactly generated).}

\end{example}

\begin{example}

We have canonical identifications:

\[
\begin{gathered}
\QCoh(\cG)^{\cG,w} \simeq \Vect \\
\QCoh(\cG)_{\cG,w} \simeq \Vect.
\end{gathered}
\]

\noindent Indeed, the former is a special case of Example \ref{e:inv-desc},
and the latter equivalence is tautological: we are computing
$\QCoh(\cG) \otimes_{\QCoh(\cG)} \Vect$.

\end{example}

\subsection{The non-renormalized category of representations}

We use the notation $\QCoh(\bB \cG)$ for $\Vect^{\cG,w}$.
Of course, this DG category is (tautologically) $\QCoh$ of the
prestack $\bB \cG$.

Note that $\QCoh(\bB \cG) = \TwoHom_{\QCoh(\cG)}(\Vect,\Vect)$ as monoidal
categories (where the former has the usual tensor product of quasi-coherent sheaves
as its monoidal structure). Therefore,
$\sC^{\cG,w}$ and $\sC_{\cG,w}$ have canonical $\QCoh(\bB G)$-module
structures for any $\sC \in \QCoh(\cG)\mod$.

\begin{rem}

For any affine scheme $\cY$ with an action of $\cG$, $\QCoh(\cY/\cG)$ inherits
a canonical $t$-structure characterized by the fact that $\QCoh(\cY/\cG) \to \QCoh(\cY)$
is $t$-exact; moreover, this $t$-structure is left (and right) complete. Indeed,
these facts follow by the usual argument for Artin stacks (c.f. \cite{dennis-dag}).
In particular, $\QCoh(\bB \cG)$ has a canonical left complete $t$-structure.

\end{rem}

\begin{rem}

For $\cG$ of infinite type, we emphatically do not use the 
notation $\Rep(\cG)$
for $\QCoh(\bB \cG)$ out of deference to Gaitsgory, who suggests
(c.f. \cite{dmod-aff-flag}) to use this notation instead for a ``renormalized"
form of this category in which bounded complexes of finite-dimensional
representations are declared to be compact. However, this
renormalized form will not play any role in this paper.

\end{rem}

\subsection{Averaging}

By the Beck-Chevalley theory, the functor $\Oblv: \sC^{\cG,w} \to \sC$
admits a continuous right adjoint $\Av_*^{\cG,w} = \Av_*^w$, and is
comonadic. Moreover, the comonad $\Oblv \Av_*^w$ on $\sC$ is given by
convolution with $\sO_{\cG} \in \QCoh(\cG)$.

\subsection{}

We have the following basic result.

\begin{prop}\label{p:inv-ff}

\begin{enumerate}

\item

For every $\sC \in \QCoh(\cG)\mod$, the natural functor:

\[
\sC^{\cG,w} \underset{\QCoh(\bB \cG)}{\otimes} \Vect \to
\sC
\]

\noindent is an equivalence.

\item 

The invariants functor:

\[
\QCoh(\cG)\mod \to \QCoh(\bB \cG)\mod
\]

\noindent is fully-faithful.

\end{enumerate}

\end{prop}

\begin{proof}

For the first part:

We claim that the averaging functor $\sC \to \sC^{\cG,w}$ is monadic.
Indeed, its composite with the forgetful functor is tensoring with $\sO_{\cG}$,
and therefore it is obviously conservative. Moreover, it commutes with all colimits,
so this gives the claim.

Now let $\pi$ denote the map $\Spec(k) \to \bB \cG$ and consider
$\pi_*(k)$ as an algebra object in $\QCoh(\bB \cG)$. Note that this
algebra object induces the monad above. Therefore, we have: 

\[
\sC^{\cG,w} \underset{\QCoh(\bB \cG)}{\otimes} \Vect =
\sC^{\cG,w} \underset{\QCoh(\bB \cG)}{\otimes} \pi_*(k)\mod(\QCoh(\bB \cG)) =
\pi_*(k)\mod(\sC^{\cG,w}) = \sC
\]

\noindent as desired.\footnote{The penultimate equality
follows from \cite{dgcat} Proposition 4.8.1.}

Now the second part follows easily from the first. 
Indeed, this functor admits the left adjoint $-\otimes_{\QCoh(\bB \cG)} \Vect$,
and we have just checked that the counit for the adjunction is an equivalence.

\end{proof}

\begin{rem}

This is a different argument from the one given in \cite{shvcat},
which relied on rigidity.

\end{rem}

\subsection{Tameness}

The functor $\Av_*^w:\sC \to \sC^{\cG,w}$ induces a \emph{norm} functor:

\[
\Nm: \sC_{\cG,w} \to \sC^{\cG,w}.
\]

\begin{defin}

$\sC$ is \emph{tame} (with respect to the $\cG$ action) if
the above morphism $\Nm$ is an equivalence.

\end{defin}

\begin{example}

$\QCoh(\cG) \in \QCoh(\cG)\mod$ is tame: indeed, both invariants
and coinvariants are $\Vect$, the norm functor is obviously compatible with
respect to this.

\end{example}

\begin{example}

In \cite{shvcat} \S 7.3, it is shown that for $\cG = \prod_{i=1}^{\infty} \bG_a$,
$\Vect$ is \emph{not} tame.

\end{example}

\begin{example}\label{e:fintype-tame}

If $\cG$ is finite type, then every $\sC$ is tame. Indeed, this follows from 
\cite{shvcat} Theorem 2.2.2 and Proposition 6.2.7.

\end{example}

Generalizing Example \ref{e:fintype-tame}, we obtain the following result:

\begin{lem}\label{l:tame-cptopen}

For $\cG$ acting on $\sC$ and
$\vareps:\cG \to \cG^{\prime}$ a flat morphism with $\cG^{\prime}$ finite
type and $\cK \coloneqq \Ker(\vareps)$, $\sC$ is tame with respect to 
$\cG$ if and only if $\sC$ is tame with respect to $\cK$ (with respect to the
induced action).

\end{lem}

\subsection{}

We have the following result, which relates the failure
of tameness to the failure of the invariants functor to be a Morita equivalence.

\begin{prop}

For every $\sC \in \QCoh(\cG)\mod$, the norm functor induces an equivalence:

\[
\sC_{\cG,w} \underset{\QCoh(\bB\cG)}{\otimes} \Vect \isom
\sC^{\cG,w} \underset{\QCoh(\bB\cG)}{\otimes} \Vect.
\]

\end{prop}

\begin{proof}

By Proposition \ref{p:inv-ff}, the right hand side identifies with $\sC$. In particular,
both sides commutes with colimits in $\sC$ (since this is obvious for the left hand side),
and therefore we reduce to the case $\QCoh(\cG)$, where the result is clear.

\end{proof}

\subsection{Tameness and 1-affineness}\label{ss:gpact-finish}

Suppose $\cY$ is a prestack with an action of $\cG$.

\begin{defin}

$\cY$ is \emph{tame} (with respect to the action of $\cG$) if 
$\QCoh(\cY)\in \QCoh(\cG)\mod$ is tame.

\end{defin}

\begin{prop}\label{p:tame-1aff}

In the above notation, suppose that $\cY$ is 1-affine and tame. 
Then $\cY/\cG$ is 1-affine.

\end{prop}

\begin{proof}

Because $\pi:\cY \to \cY/\cG$ is a 1-affine morphism, we obtain:

\[
\ShvCat_{/(\cY/\cG)} \simeq 
\cG \ltimes \QCoh(\cY)\mod
\]

\noindent where $\cG \ltimes \QCoh(\cY)$ is the
semidirect (or crossed) product of $\QCoh(\cG)$ with $\QCoh(\cY)$
(as a mere object of $\DGCat_{cont}$, this category is $\QCoh(\cG) \otimes \QCoh(\cY)$).
That is, for $\mathsf{C} \in \ShvCat_{/\cY/\cG}$, 
$\Gamma(\cY,\pi^*(\mathsf{C}))$ is acted on by $\QCoh(\cY)$ and
$\QCoh(\cG)$ satisfying the compatibility that the semi-direct product of
these two algebras acts;
moreover, this gives an equivalence with categories acted on by the semi-direct product.

We should show that the functor:

\begin{equation}\label{eq:semidirect-inv}
\begin{gathered}
\cG \ltimes \QCoh(\cY)\mod \to \QCoh(\cY/\cG)\mod = \QCoh(\cY)^{\cG,w}\mod \\
\sC \mapsto 
\sC^{\cG,w} = 
\TwoHom_{\cG \ltimes \QCoh(\cY)}(\QCoh(\cY),\sC) 
\end{gathered}
\end{equation}

\noindent is an equivalence. This functor admits the left
adjoint:

\[
\big(\sD \in \QCoh(\cY/\cG)\mod \big) \mapsto 
\sD \underset{\QCoh(\cY/\cG)}{\otimes} \QCoh(\cY).
\]

First, we claim (as in Proposition \ref{p:inv-ff}) that \eqref{eq:semidirect-inv}
is fully-faithful. Indeed, first note that for $\sD \in \QCoh(\cY/\cG)\mod$, we have:

\[
\sD \underset{\QCoh(\bB \cG)}{\otimes} \Vect \isom
\sD \underset{\QCoh(\cY/\cG)}{\otimes} \QCoh(\cY) 
\]

\noindent since it suffices to check that this map is an isomorphism
for $\QCoh(\cY/\cG)$ (since both sides commute with colimits),
and there it follows from Proposition \ref{p:inv-ff}. Then the claim
of fully-faithfulness follows directly from Proposition \ref{p:inv-ff}.

It therefore suffices now to show that our left adjoint is conservative. This follows
from the calculation:

\[
\begin{gathered}
\big(\sD \underset{\QCoh(\cY/\cG)}{\otimes} \QCoh(\cY) \big) 
\underset{\QCoh(\cG)}{\otimes} \Vect =
\sD \underset{\QCoh(\cY/\cG)}{\otimes} \QCoh(\cY)_{\cG,w} = \\
\sD \underset{\QCoh(\cY/\cG)}{\otimes} \QCoh(\cY)^{\cG,w} =
\sD \underset{\QCoh(\cY/\cG)}{\otimes} \QCoh(\cY/\cG) =
\sD
\end{gathered}
\]

\noindent for any $\sD \in \QCoh(\cY/\cG)\mod$.

\end{proof}

\subsection{}

Let $G$ be an affine algebraic group for the remainder of this section.
We now formulate the main result of this section.

\begin{thm}\label{t:tame}

For any $r \geq 0$, $t^{-r}\fg[[t]]dt$ is tame with
respect to the $G(O)$-action on it.

\end{thm}

\begin{rem}

Since $\QCoh(\fg((t))dt)$ is a colimit/limit of the categories
$\QCoh(t^{-r}\fg[[t]]dt)$, we immediately obtain that $\fg((t))dt$ 
is tame with respect to the action of $G(O)$ as well.

\end{rem}

From Proposition \ref{p:tame-1aff}, we obtain:

\begin{cor}\label{c:gauge/cpt-1aff}

The quotient $t^{-r}\fg[[t]]dt/G(O)$ is 1-affine.

\end{cor}

The proof of Theorem \ref{t:tame} will occupy the remainder of this section.

\subsection{} 

The main point in proving Theorem \ref{t:tame} is the following.

\begin{prop}\label{p:tame-amplitude}

The global sections functor $\Gamma: \QCoh(t^{-r}\fg[[t]]dt/G(O)) \to \Vect$
has finite cohomological amplitude.

\end{prop}

The proof of Proposition \ref{p:tame-amplitude} will be given
in \S \ref{ss:gamma-bdd-below}-\ref{ss:tame-amp-finish}.

\subsection{}\label{ss:gamma-bdd-below}

First, we observe that
$\Gamma$ commutes with colimits bounded uniformly from below.\footnote{A slightly
different version of this argument follows by noting that the
same fact occurs on $\bB G(O)$, and reducing to that
case using affineness of the morphism $t^{-r}\fg[[t]]dt/G(O) \to \bB G(O)$; 
however the corresponding fact for $\bB G(O)$ is proved by 
exactly the same argument as that presented here.} 

Indeed, this follows from $t$-exactness and comonadicity of the functor
$\QCoh(t^{-r}\fg[[t]]dt/G(O)) \to \QCoh(t^{-r}\fg[[t]]dt)$ by a well-known argument.
We include this argument below for completeness.

Let $\pi$ denote the structure map $t^{-r}\fg[[t]]dt \to t^{-r}\fg[[t]]dt/G(O)$.
Suppose $\sF \in \QCoh(t^{-r}\fg[[t]]dt/G(O))^{\geq 0}$. By comonadicity,
we have:

\[
\sF = \underset{[m] \in \bDelta}{\lim} \, (\pi_*\pi^*)^{m+1}(\sF).
\]

\noindent Since $\Gamma$ commutes with limits (being a right adjoint), we
have:

\[
\Gamma(\sF) = 
\underset{[m] \in \bDelta}{\lim} \, \Gamma(t^{-r}\fg[[t]]dt/G(O), (\pi_*\pi^*)^{m+1}(\sF)) =
\underset{[m] \in \bDelta}{\lim} \, \Gamma(t^{-r}\fg[[t]]dt, \pi^*(\pi_*\pi^*)^{m}(\sF)).
\]

\noindent Therefore, we have:

\[
\begin{gathered}
\tau^{\leq n} \Gamma(\sF) = 
\tau^{\leq n} \underset{[m] \in \bDelta}{\lim} \, \Gamma(t^{-r}\fg[[t]]dt, \pi^*(\pi_*\pi^*)^{m}(\sF)) = 
\tau^{\leq n} \underset{[m] \in \bDelta_{\leq n+1}}{\lim} \, 
\Gamma(t^{-r}\fg[[t]]dt, \pi^*(\pi_*\pi^*)^{m}(\sF)).
\end{gathered}
\]

\noindent where $\bDelta_{\leq n+1} \subset \bDelta$ is the subcategory
consisting of totally ordered finite sets of size at most $n+1$. Here
the crucial second equality follows from the fact that 
$\tau^{\leq n} \lim_{\bDelta}$ computes the limit in the 
$(n+1)$-category $\Vect^{[0,n]}$. 

But now this expression visibly
commutes with filtered colimits, since:

\begin{itemize}

\item The $t$-structure on $\Vect$
is compatible with filtered colimits. 

\item $\Gamma$  on 
$t^{-r}\fg[[t]]dt$ commutes with all colimits (by affineness).

\item $\pi_*$ commutes with colimits, since $\pi$ is affine.

\item Our limit is now finite.

\end{itemize}

Finally, we obtain the claim from right completeness of the $t$-structure on $\Vect$.

\subsection{}

Next, we reduce to showing that $\Gamma$ has finite cohomological
amplitude\footnote{To be totally clear: a DG functor
$\sC \to \sD$ between DG categories with $t$-structures 
has \emph{finite cohomological amplitude} if there is an integer $\delta$
for which $F(\sC^{\leq 0}) \subset \sD^{\leq \delta}$ and 
$F(\sC^{\geq 0}) \subset \sD^{\geq -\delta}$.}
when restricted to $\QCoh(t^{-r}\fg[[t]]dt/G(O))^+$, i.e., when
restricted to objects bounded from below.

This follows from the following general lemma (which Gaitsgory
informs us appeared already in a simpler form as \cite{finiteness} Lemma 2.1.4).

\begin{lem}\label{l:postnikov}

Suppose $\sC$ and $\sD$ are DG categories 
equipped with left complete $t$-structures.

Suppose $F: \sC \to \sD$ is a (possibly non-continuous) DG functor.
Suppose that
$F|_{\sC^+}$ has finite cohomological amplitude. 

\begin{enumerate}

\item 

The following conditions are equivalent:

\begin{enumerate}

\item\label{i:coh-amp} $F$ has finite cohomological amplitude.

\item\label{i:postnikov} $F$ commutes with left Postnikov towers, i.e.:
for any $\sF \in \sC$, the morphism:

\[
F(\sF) = F( \underset{n}{\lim} \, \tau^{\geq -n} \sF) \to
\underset{n}{\lim} \,  F(\tau^{\geq -n} \sF)
\]

\noindent is an isomorphism.

\end{enumerate}

\item 

Suppose moreover that $\sC$ and $\sD$ are cocomplete,
the $t$-structures on them are compatible with
filtered colimits, and that 
$F|_{\sC^+}$ commutes with colimits bounded uniformly from
below. 

Then the equivalent conditions above imply that $F$ is continuous.

\end{enumerate}

\end{lem}

\begin{proof}

To see that \eqref{i:coh-amp} implies \eqref{i:postnikov}:

We compute:

\begin{equation}\label{eq:post}
\underset{n}{\lim} \,  F(\tau^{\geq -n} \sF) =
\underset{n}{\lim} \, \underset{m}{\lim} \,  \tau^{\geq -m} F(\tau^{\geq -n} \sF) =
\underset{m}{\lim} \, \underset{n}{\lim} \, \tau^{\geq -m} F(\tau^{\geq -n} \sF).
\end{equation}

For fixed $m$, we have:

\[
\underset{n}{\lim} \, \tau^{\geq -m} F(\tau^{\geq -n} \sF) = 
\tau^{\geq -m} F(\sF)
\]

\noindent since the limit stabilizes to $F(\sF)$ for $m$ large enough by the
boundedness of the cohomological amplitude of $F$.

Therefore, we compute the right hand side of \eqref{eq:post} as:

\[
\underset{m}{\lim} \, \tau^{\geq -m} F(\sF) = F(\sF)
\]

\noindent by left completeness of the $t$-structure on $\sD$.

Now suppose that \eqref{i:postnikov} holds.

We begin by reversing the above logic to show that the hypothesis \eqref{i:postnikov} 
implies that for every $\sF\in \sC$ and every $m$, we have:

\begin{equation}\label{eq:post-2}
\tau^{\geq -m} F(\sF) \isom
\underset{n}{\lim} \, \tau^{\geq -m} F(\tau^{\geq -n} \sF).
\end{equation}

Because $F|_{\sC^+}$ assumed to be cohomologically bounded,
the limit on the right stabilizes. Let $\sG_m \in \sD$ denote this limit. 
Note that $\sG_m \in \sD^{\geq -m}$ since this subcategory is closed under limits. 
Moreover, $\tau^{\geq -m}(\sG_{m+1}) = \sG_m$ since
each equals $ \tau^{\geq -m} F(\tau^{\geq -n} \sF)$ for $n$ large enough
(i.e., because the limits defining $\sG_{m}$ and $\sG_{m+1}$ stabilize).

Since $m \mapsto \sG_m$ is compatible under truncations, this system
is equivalent to the datum of the object $\sG = \lim_m \sG_m \in \sD$,
by left completeness of the $t$-structure of $\sD$. Therefore, it suffices
to show that:

\[
F(\sF) \to \underset{m}{\lim} \, \sG_m \isom 
\underset{m}{\lim} \, \underset{n}{\lim} \, \tau^{\geq -m} F(\tau^{\geq -n} \sF)
\]

\noindent is an isomorphism. But this is clear, since the right hand side equals:

\[
\underset{n}{\lim} \, \underset{m}{\lim} \, \tau^{\geq -m} F(\tau^{\geq -n} \sF) =
\underset{n}{\lim} \, F(\tau^{\geq -n} \sF) \overset{\eqref{i:postnikov}}{=}
F(\sF).
\]

Now, since we noted that the limit in the right hand side of \eqref{eq:post-2}
stabilizes, this makes it clear that $F$ is cohomologically bounded,
so we see that \eqref{i:postnikov} implies \eqref{i:coh-amp}.

Finally, suppose that 
$F|_{\sC^+}$ commutes with colimits uniformly 
bounded from below and that the $t$-structures are compatible
with filtered colimits. We will show that \eqref{i:postnikov} implies that $F$ is continuous.

We need to show that for
$i \mapsto \sF_i$ a \emph{filtered} diagram, we have:

\[
\underset{i}{\colim} \, F(\sF_i) \isom F(\underset{i}{\colim} \, \sF_i).
\]

\noindent It suffices to show this after applying $\tau^{\geq -m}$ for any
integer $m$. 
Since the $t$-structures are compatible with filtered colimits,
the left hand side of the above becomes $\colim_i \tau^{\geq -m} F(\sF_i)$
upon truncation.

We now compute the right hand side as:

\[
\tau^{\geq -m} F(\underset{i}{\colim} \, \sF_i) \overset{\eqref{eq:post-2}}{=}
\underset{n}{\lim} \, \tau^{\geq -m} F(\tau^{\geq -n} \underset{i}{\colim} \, \sF_i).
\]

\noindent Using the compatibility between filtered colimits and $t$-structures,
and the commutation of $F$ with uniformly bounded colimits, this term now becomes:

\[
\underset{n}{\lim} \, \underset{i}{\colim} \, \tau^{\geq -m} F(\tau^{\geq -n} \sF_i).
\]

\noindent Observe that this limit is eventually constant, since $F|_{\sC^+}$ has
bounded amplitude. This justifies interchanging the limit and the colimit, so that
we can now compute further:

\[
\underset{i}{\colim} \, \underset{n}{\lim} \, \tau^{\geq -m} F(\tau^{\geq -n} \sF_i) 
\overset{\eqref{eq:post-2}}{=} 
\underset{i}{\colim} \, \tau^{\geq -m} F(\sF_i)
\]

\noindent as desired.

\end{proof}

As an auxiliary remark, Lemma \ref{l:postnikov} allows us to deduce the
following result from Proposition \ref{p:tame-amplitude}:

\begin{cor}\label{c:cpt}

$\Gamma$ on $t^{-r}\fg[[t]]dt/G(O)$ commutes with colimits, i.e.,
the structure sheaf of $t^{-r}\fg[[t]]dt/G(O)$ is compact. 

\end{cor}

We emphasize that Corollary \ref{c:cpt} is a consequence of the
(yet unproved) Proposition \ref{p:tame-amplitude}, so of course
we will not appeal to it in the course of the argument below.

We will return to Corollary \ref{c:cpt} in \S \ref{s:cpt-gen}.

\begin{rem}

If Corollary \ref{c:cpt} appears innocuous 
given \S \ref{ss:gamma-bdd-below}, do note
that it fails (as does Proposition \ref{p:tame-amplitude}) for $\bB G(O)$
in place of $t^{-r}\fg[[t]]dt/G(O)$, unlike \S \ref{ss:gamma-bdd-below}. 
The reader might also glance at
\cite{dmod-aff-flag} to see these homological subtleties 
play out in the context of Kac-Moody representations.

\end{rem}

\subsection{}

By Theorems \ref{t:artin} and \ref{t:jumps}, we can therefore
can choose $s$ so that the geometric fibers of the map:\footnote{Implicitly, we are
assuming $r>0$ here that the polar part map makes sense; of course this
is fine for our purposes.}

\[
t^{-r}\fg[[t]]dt/\cK_{r+s} \to t^{-r}\fg[[t]]dt/t^s\fg[[t]]dt
\]

\noindent are smooth Artin stacks.

Now observe that it suffices to prove Proposition \ref{p:tame-amplitude}
after replacing $G(O)$ by any congruence subgroup; we will
prove the homological boundedness for global sections
on $t^{-r}\fg[[t]]dt/\cK_{r+s}$ instead.

\subsection{}

Next, we reduce to showing that there is a \emph{uniform} bound (see below)
on the possible cohomological amplitude of 
$\Gamma(t^{-r}\fg[[t]]dt/\cK_{r+s},\sF)$ whenever $\sF$ is set-theoretically
supported\footnote{For a (possibly non-closed) point $\eta \in S$ 
for $S$ a smooth (finite type) scheme, we will say that $\sF$ is
supported on this point if it can be written as a filtered colimit
of sheaves $i_{\eta,*}(\sG)$ for $\sG \in \QCoh(\eta)$ (i.e.,
the DG category of vector spaces with coefficients
in the residue field of $\eta$) for $i_{\eta}: \eta \into S$.

We warn that this is a bad definition for general $S$ and is being
introduced in an ad hoc way to make the language easier in
our current setting.} 
on the fiber of a schematic\footnote{
We use the term \emph{schematic point}
to mean the generic point of an integral subscheme, i.e., a point in the
topological space underlying a scheme in the locally ringed spaces perspective on the
subject.} 
point. 

\begin{rem}

Here by \emph{uniform bound}, we mean that the bound 
should be independent of which fiber we choose.
It is an immediate consequence of Theorem \ref{t:artin} that there is a bound
fiber-by-fiber, but what will remain to show after this
subsection is that we can bound the relevant 
cohomological dimensions uniformly. 

\end{rem}

For notational simplicity, let $\cP$ denote the affine space $t^{-r}\fg[[t]]dt/t^s\fg[[t]]dt$.
The Cousin complex then provides a bounded resolution of $\sO_{\cP}$ by
direct sums of sheaves of the form $\mathscr{Cous}_{\eta}[-\on{ht}(\eta)]$ 
for $\eta \in \cP$ a schematic point and $\mathscr{Cous}_{\eta} \in \QCoh(\cP)^{\heart}$
set-theoretically supported on $\eta$.

Tensoring with the Cousin resolution then implies that any $\sF$ bounded from
below admits a finite filtration where the associated graded terms
are direct sums of sheaves of the form $\sF_{\eta}$ where $\sF_{\eta}$ is supported
on a point $\eta \in \cP$, and where the $\sF_{\eta}$ are bounded
uniformly (in $\eta$) from below.

Because the filtration is finite, and because 
$\Gamma(t^{-r}\fg[[t]]dt/\cK_{r+s},-)$ commutes with direct sums bounded uniformly,
we obtain the desired reduction.

\subsection{Completion of the proof of Proposition \ref{p:tame-amplitude}}\label{ss:tame-amp-finish}

Again using the fact that $\Gamma$ commutes with colimits bounded
uniformly from below, we see that it suffices to treat those
sheaves $\sF$ that are pushed forward from the fiber of
$t^{-r}\fg[[t]]dt/\cK_{r+s}$ at some schematic point in 
$t^{-r}\fg[[t]]dt/t^s\fg[[t]]dt$. 

Therefore, it suffices to show that the cohomological dimension
of the global sections functor is bounded uniformly on the fibers
the map $t^{-r}\fg[[t]]dt/\cK_{r+s} \to t^s\fg[[t]]dt$.

Recall that the dimensions of these fibers are bounded \emph{uniformly}:
this is Proposition \ref{p:bdd-dimns}. Moreover, the dimensions of the automorphism
group at a point is uniformly bounded by $\dim(\fg)$,
since the Lie algebra of the stabilizer subgroup at a point embeds into 
$H_{dR}^0$ of the corresponding connection.

Therefore, the claim follows from the next
result.\footnote{There is a slight finiteness issue to clarify here, since
our stacks occur a priori as quotients of infinite type affine schemes
by prounipotent groups. However, since the quotient is an Artin stack,
when we quotient by a small enough congruence subgroup, we obtain a scheme,
and then by Noetherian descent \cite{thomason-trobaugh} Proposition C.6,
we obtain that the quotient by a small enough congruence subgroup $\cK_N$ 
is finite type affine. Then our stack of interest
is the quotient of this finite type affine by the residual 
$\cK_{r+s}/\cK_N$-action, as desired.}

\begin{prop}

Suppose that $X$ is an affine scheme of finite type and $K$ is a unipotent group
acting on $X$.\footnote{We really
need to extend scalars here and work over base-fields other than the
arbitrary (characteristic zero) ground field $k$.
Of course, this is irrelevant, so for simplicity we just let $k$ denote our
ground field in this proposition.}

Define the stack $\sX \coloneqq X/K$.
Then the cohomological dimension of the global sections functor
$\Gamma(\sX,-)$ is bounded by:

\[
\dim(\sX)+
\underset{x \in \sX \text{ a geometric point}}{\max} 
2\dim(\Aut_{\sX}(x)).
\]

\end{prop}

\begin{proof}

First, note that by replacing $X$ by $X^{red}$, we can assume that $X$ is reduced.
The action of $K$ also preserves irreducible components, so we can assume $X$
is integral.

In the case that $K$ acts transitively on $X$, the result is clear, since (perhaps after
innocuously extending our base-field $k$) $\sX = \bB(\Stab_{K}(x))$ for a 
$k$-point $x \in X$, and then the cohomological dimension
is bounded by the dimension of this stabilizer.

Let $U \subset X$ be a non-zero $K$-stable open subscheme. We claim
that $U$ contains a non-zero $K$-stable \emph{affine} open subscheme.

Indeed, let $Z \subset X$ be the reduced complement to $U$.
By unipotence of $K$, we can find a non-zero 
$K$-invariant function $f$ vanishing on $Z$, since the ideal of functions
vanishing on $Z$ is a non-zero $K$-representation. 

Therefore, we can find an affine $K$-stable open $\emptyset \neq V \subset X$ such that
the group scheme of stabilizers is smooth over $V$ (since we can find some $K$-stable
open for which this is true, namely, the open of \emph{regular} points, for which stabilizer
has the minimal dimension).

We can then compute $\Gamma(V/K,-)$ by pulling a quasi-coherent sheaf back to
$V$ and then taking Lie algebra cohomology with respect to the Lie algebra
of the stabilizer group scheme. Since $V$ is affine, this means that the
cohomological dimension of $V/K$ is at most the dimension of this
stabilizer, which we bound as:

\[
\underset{x \in V/K}{\max} \, \dim (\Aut_{\sX}(x)) \leq  
\underset{x \in \sX}{\max} \, \dim (\Aut_{\sX}(x)) \leq  
\dim(\sX) + \underset{x \in \sX}{\max} \, 2\dim (\Aut_{\sX}(x)) .
\]

\noindent (Here we begin a convention throughout this argument
that e.g. $\max_{x \in \sX}$ assumes $x$ a geometric point.)

Now let $Y \subset X$ be the reduced complement to $V$, and note
that $\dim(Y) < \dim(X)$.
By Noetherian descent, we can assume the result holds for $\sY \coloneqq Y/K$. 
Then by an easy argument (c.f. \cite{finiteness} Lemma 2.3.2), the cohomological
dimension of $X/K$ is then bounded by:\footnote{For the comparison with
\cite{finiteness}, note that $V/K \into X/K$ is affine, so the pushforward is
$t$-exact.}

\[
\begin{gathered}
\max \big\{
\dim(\sX) +
\underset{x \in \sX}{\max} \,
 2\dim(\Aut_{\sX}(x)), \,
1 + \dim(\sY) + \underset{y \in \sY}{\max} \, 
2\dim(\Aut_{\sY}(x))
\big\} \leq \\
\dim(\sX)  + 
\underset{x \in \sX}{\max} 
\dim(2\Aut_{\sX}(x))
\end{gathered}
\]

\noindent as desired.

\end{proof}

\subsection{Digression: $t$-structures and monadicity}

We digress for a moment to discuss the relationship between $t$-structures
and monadicity. This is well-known, and appears in some form in both
\cite{higheralgebra} and \cite{shvcat}, but we include it here because 
Corollary \ref{c:monad-tstr-iterates} is not quite stated in a conveniently referenced way in either
source.

We will use the following general (and well-known) results to check monadicity.

\begin{lem}\label{l:geom-real-tstr}

Suppose that
$\sC$ and $\sD$ are (possibly non-cocomplete) DG categories equipped
with left complete $t$-structures.
Suppose that $G:\sC \to \sD$ is right $t$-exact up to cohomological shift.

Then $\sC$ admits geometric realizations of simplicial diagrams bounded uniformly from
above, and $G$ preserves these geometric realizations. I.e., the functor 
$\sC^{\leq 0} \to \sD$ commutes with geometric realizations.

\end{lem}

\begin{proof}

We can assume $G$ is right $t$-exact, and then this result is 
\cite{higheralgebra} Lemma 1.3.3.11 (2).

\end{proof}

\begin{cor}\label{c:monad-tstr-iterates}

Suppose that $G: \sC \to \sD$ is a conservative DG functor 
between DG categories with left and right complete $t$-structures.
Suppose moreover that $G$ admits a left adjoint $F$ such that:

\begin{itemize}

\item The functors $F$ and $G$ are of bounded cohomological amplitude.

\item The upper amplitude of the functors $(FG)^n$ is bounded independently of $n$,
i.e., there is an integer $N$ such that each functor $(FG)^n$ $(n\geq 0)$ maps
$\sC^{\leq 0}$ into $\sC^{\leq N}$.

\end{itemize}

Then $G$ is monadic.

\end{cor}

\begin{rem}

For example, the last two hypotheses hold if $F$ and $G$ are both $t$-exact.
However, we will see below that one can sometimes arrange this
situation \emph{without} $t$-exactness of both functors. 

\end{rem}

\begin{proof}[Proof of Corollary \ref{c:monad-tstr-iterates}]

By the Barr-Beck theorem, it suffices to show that for every $\sF \in \sC$, the map
(coming from the bar resolution):

\[
\underset{[n]\in \bDelta^{op}} {\colim} \, (FG)^{n+1} (\sF) \to \sF
\]

\noindent is an isomorphism.

We have that $G|_{\sC^{\leq 0}}$ commutes with
geometric realizations by Lemma \ref{l:geom-real-tstr}.

For $\sF \in \sC$ as above and for any $m \geq 0$, we claim that the map:

\[
\underset{[n] \in \bDelta^{op}}{\colim} \, (FG)^{n+1} \tau^{\leq m} \sF \to
\tau^{\leq m} \sF
\]

\noindent is an isomorphism. Indeed, by conservativeness,
it suffices to see this after applying $G$. This geometric realization is of terms
bounded uniformly from above by hypothesis on the endofunctors $(FG)^{n+1}$, 
so the colimit commutes with $G$; then since this simplicial
diagram is $G$-split, we obtain the claim.

Because $(FG)^{n+1}$ is cohomologically bounded for each $n$ (since
$F$ and $G$ are),
for every $\sF \in \sC$ we have:

\[
(FG)^{n+1} (\sF) =
\underset{m}{\colim} \,
(FG)^{n+1} \tau^{\leq m} \sF
\] 

\noindent by (the dual to) Lemma \ref{l:postnikov}. 
Combining this with the above, we compute:

\[
\begin{gathered}
\underset{[n] \in \bDelta^{op}}{\colim} \,
(FG)^{n+1} (\sF) =
\underset{[n] \in \bDelta^{op}}{\colim} \, \underset{m}{\colim} \,
(FG)^{n+1} \tau^{\leq m} \sF = 
 \underset{m}{\colim} \,
\underset{[n] \in \bDelta^{op}}{\colim} \,
(FG)^{n+1} \tau^{\leq m} \sF = \\
\underset{m}{\colim} \, \tau^{\leq m} \sF = \sF
\end{gathered}
\] 

\noindent as desired. 

\end{proof}

\subsection{}\label{ss:tame-pf}

We now prove Theorem \ref{t:tame}.

\begin{proof}[Proof of Theorem \ref{t:tame}]

Let $\pi$ denote the morphism $t^{-r}\fg[[t]]dt \to t^{-r}\fg[[t]]dt/G(O)$.

\step 

By the Beck-Chevalley method (c.f. \cite{shvcat} Appendix C), the 
\emph{non-continuous} right adjoint to the
canonical functor $\QCoh(t^{-r}\fg[[t]]dt) \to \QCoh(t^{-r}\fg[[t]]dt)_{G(O),w}$
is monadic.

Let $\pi^?$ denote the (non-continuous) right adjoint 
$\pi^?$ to $\pi_*$. One easily checks (c.f. \cite{shvcat} 
\S 6.3, especially Proposition 6.3.7) that the monads on 
$\QCoh(t^{-r}\fg[[t]]dt)$ corresponding to $\pi^?\pi_*$ and to
the monad coming from coinvariants coincide, compatibly with
the norm functor comparing invariants and coinvariants.

Therefore, it suffices to show that $\pi^?$ is monadic.

\step

We claim that $\pi^?$ is conservative. It is equivalent to say that $\pi_*$ generates
the target under colimits, which is the form in which we will check this result.

First, note that we can replace $G(O)$ by any congruence subgroup. Indeed,
by \cite{shvcat} Proposition 6.2.7 and 1-affineness of the classifying stack
of a finite type algebraic group $\vG$, $\Av_*^{\vG,w}: \sC \to \sC^{\vG,w}$ generates
the target under colimits for any $\QCoh(\vG)$-module category $\sC$.

Therefore, by Theorems \ref{t:artin} and \ref{t:jumps}
it suffices to show this result for $t^{-r}\fg[[t]]dt/\cK_{r+s}$ where we
have a polar part map to $t^{-r}\fg[[t]]dt/t^s\fg[[t]]dt$ with geometric fibers
smooth Artin stacks. 

Note that $\pi_*$ is a morphism of $\QCoh(t^{-r}\fg[[t]]dt/t^s\fg[[t]]dt)$-module
categories. Moreover, as for any smooth scheme, 
$\QCoh(t^{-r}\fg[[t]]dt/t^s\fg[[t]]dt)$ is generated under colimits by skyscraper
sheaves at geometric points. Therefore, it suffices to show that
the morphism $\pi_*$ generates under colimits when restricted to any geometric
fiber.

But on these geometric fibers, the map $\pi$ base-changes to
$\widetilde{\pi}: S \to S/K$ where $S$ is an affine scheme
(the geometric fiber of $t^{-r}\fg[[t]]dt$), $K$ is a prounipotent group scheme
and $S/K$ is a smooth (finite-dimensional) Artin stack. 

By Noetherian descent (\cite{thomason-trobaugh} Proposition C.6),
there is a normal compact open subgroup $K_0 \subset K$ such
that $S/K_0$ is an affine scheme of finite type. Then
$S \to S/K_0$ is a $K_0$-torsor. Moreover, this torsor
is necessarily trivial because $K_0$ is prounipotent
and $S/K_0$ is affine. Therefore, the pushforward obviously generates under
colimits in this case. 

Then $S/K_0 \to S/K$ generates under colimits
using the fact that $K/K_0$ is finite type, and using \cite{shvcat} Proposition 6.2.7 once again. 

\step

Next, we claim that $\pi^?$ has bounded cohomological amplitude.

Note that:

\[
\ul{\Hom}_{\QCoh(t^{-r}\fg[[t]]dt)}(\sO_{t^{-r}\fg[[t]dt},\pi^?(\sF)) =
\ul{\Hom}_{\QCoh(t^{-r}\fg[[t]]dt/G(O)}(\pi_*(\sO_{t^{-r}\fg[[t]dt}),\sF)
\in \Vect
\]

\noindent so it suffices to show that the latter complex is bounded.

But observe that $\pi_*(\sO_{t^{-r}\fg[[t]dt})$ is an ind-finite dimensional
vector bundle under a countable direct limit: indeed, it is obtained by
pullback from the regular representation in $\QCoh(\bB G(O))$ 
(i.e., the pushforward of $k \in \Vect$ under $\Spec(k) \to \bB G(O)$).

By Proposition \ref{p:tame-amplitude} (and a dualizability argument),
$\ul{\Hom}_{\QCoh(t^{-r}\fg[[t]]dt/G(O)}$ out of any finite rank vector bundle
has cohomological amplitude bounded independently of the vector bundle.

Note $\ul{\Hom}_{\QCoh(t^{-r}\fg[[t]]dt/G(O)}(\pi_*(\sO_{t^{-r}\fg[[t]dt}),-)$
is a countable limit of such functors. Since the limit is countable, 
the ``$R\lim$" aspect can increase the cohomological amplitude by at most $1$,
so this functor also has bounded cohomological amplitude. 

\step

Finally, we claim that $(\pi_*\pi^?)^n$ has cohomological amplitude
bounded independently of $n$. Note that this suffices by Corollary \ref{c:monad-tstr-iterates}.

By the previous step, and since $\pi_*$ is $t$-exact
(by affineness), it suffices to show that $\pi^?\pi_*$ is $t$-exact. 
Observe that:

\begin{equation}\label{eq:?-basechange}
\pi^?\pi_* = \act_* p_2^?
\end{equation}

\noindent where the maps here are 
$G(O) \times t^{-r}\fg[[t]]dt \overset{p_2}{\underset{\act}{\rightrightarrows}} t^{-r}\fg[[t]]dt$.
Indeed, the identity \eqref{eq:?-basechange} follows immediately 
from the base-change between upper-* and lower-* functors.

Since $\act_*$ is $t$-exact (by affineness of $G(O)$), we need to see that
$p_2^?$ is $t$-exact. This follows by explicit calculation:
it is equivalent to see that: 

\[
\ul{\Hom}_{\QCoh(t^{-r}\fg[[t]]dt)}(p_{2,*}(\sO_{G(O)\times t^{-r}\fg[[t]]dt}),-) =
\ul{\Hom}_{\QCoh(t^{-r}\fg[[t]]dt)}(\Gamma(G(O),\sO_{G(O)}) \otimes \sO_{t^{-r}\fg[[t]]dt},-)
\]

\noindent is $t$-exact. Clearly $\Gamma(G(O),\sO_{G(O)}) \otimes \sO_{t^{-r}\fg[[t]]dt}$
is a free quasi-coherent sheaf on an affine scheme, giving the claim.

\end{proof}


\section{Compact generation}\label{s:cpt-gen}


\subsection{}

In this section, we discuss applications of the results
of \S \ref{s:tame} to questions of compact generation.

This material logically digresses from the overall goal of
proving our main theorem (c.f. \S \ref{ss:mainthm-statement}).
However, it is a simple
application of the results of the previous section, so we include it here.

\subsection{}

First, we give the following result, valid for any affine algebraic
group $G$.

\begin{prop}\label{p:tame-cpt-gen}

For every $r \geq 0$, $\QCoh(t^{-r}\fg[[t]]dt/G(O))$ is compactly
generated. Moreover, an object in this category
is compact if and only if it is perfect. 

In particular, $\QCoh(t^{-r}\fg[[t]]dt/G(O))$ is rigid symmetric monoidal.

\end{prop}

\begin{proof}

By dualizability, perfect objects are compact if and only 
if the structure sheaf is compact, and this holds
in our case by Corollary \ref{c:cpt}.

Now let $\rho$ denote the structure map 
$t^{-r}\fg[[t]]dt/G(O) \to \bB G(O) \to \bB G$. We claim that
the objects $\rho^*(V)$ for $V \in \Rep(G)$ a bounded complex
of finite-dimensional representations form a set of compact generators.
Obviously this would imply that $\QCoh(t^{-r}\fg[[t]]dt/G(O))$ is compactly
generated by perfect objects, and therefore the two notions would
coincide as desired.

We now claim that these objects generate under colimits.
Indeed, as in the proof of Theorem \ref{t:tame}, 
$\QCoh(t^{-r}\fg[[t]]dt/G(O))$ is generated under colimits
and shifts by $\pi_*(\sO_{t^{-r}\fg[[t]]dt/G(O)}) = 
\rho^*(\sO_{G(O)})$. Remarks on the notation:
as in \S \ref{s:tame}, $\pi$ is the projection map
$t^{-r}\fg[[t]]dt \to t^{-r}\fg[[t]]dt/G(O)$, and we
are letting $\sO_{G(O)}$ denote the regular representation
of $G(O)$.

Then observe that $\sO_{G(O)}$ is the colimit of 
$\sO_{G(O)/\cK}$ for $\cK$ a congruence subgroup. 
Since $\Ker(G(O)/\cK \to G)$ is unipotent, we then see that
$\sO_{G(O)/\cK}$ lies in the category generated under colimits
by representations of $G$.

\end{proof}

\subsection{}

We readily deduce the following.

\begin{cor}\label{c:cptgen-int}

$\QCoh(\fg((t))dt/G(O))$ is compactly generated.

\end{cor}

\begin{proof}

The structure maps $t^{-r}\fg[[t]]dt/G(O) \to t^{-r-1}\fg[[t]]dt/G(O)$ are
regular embeddings, and therefore the restriction for quasi-coherent
sheaves admits a \emph{left} adjoint. 
Therefore, we have:

\[
\QCoh(\fg((t))dt/G(O)) = 
\underset{r}{\lim} \, \QCoh(t^{-r}\fg[[t]]dt/G(O)) =
\underset{r}{\colim} \, \QCoh(t^{-r}\fg[[t]]dt/G(O))
\]

\noindent where the colimit is under the left adjoints.
Since each of the structure functors in this colimit
obviously preserves compacts (being left adjoints to continuous functors), 
this implies that
the colimit is compactly generated as well 
(since $\Ind:\DGCat \to \DGCat_{cont}$ is a left adjoint, so commutes
with colimits).

\end{proof}

\subsection{}

Finally, we deduce the following.

\begin{thm}\label{t:cpt-gen}

For $G$ reductive, $\QCoh(\LocSys_G(\o{\cD}))$ is compactly
generated.

\end{thm}

\begin{proof}

The map $p:\fg((t))dt/G(O) \to \LocSys_G(\o{\cD}) = \fg((t))dt/G(K)$
is a $G(K)/G(O)$-fibration, i.e., a $\Gr_G$-fibration up to 
sheafification. Therefore, this map is ind-proper (up to sheafification).

Identifying $\QCoh$ and $\IndCoh$ for $\Gr_G$ and in
\cite{indschemes}, we see that the pullback $p^*$ admits
a \emph{left} adjoint. Clearly $p^*$ is continuous and conservative,
so this gives the result from Corollary \ref{c:cptgen-int}.

\end{proof}


\section{Infinitesimal analysis: reductions}\label{s:z-red}


\subsection{Motivation}

Recall from \cite{shvcat} that none of
$\bA^{\infty} \coloneqq \colim_n \bA^n$,
$(\bA^{\infty})_0^{\wedge}$, and $\bB (\bA^{\infty})_0^{\wedge}$
is 1-affine.

In particular, the first of these results
means that while we have shown 
$t^{-r}\fg[[t]]dt/G(O)$ is 1-affine for each
$r$, there is no hope for $\fg((t))dt/G(O)$ to be 1-affine
(e.g., consider $G = \bG_m$ or $\bG_a$ to explicitly see
$\bA^{\infty}$).

However, $\bA_{dR}^{\infty}$ \emph{is} 1-affine by
\cite{shvcat}. Therefore, considering $\bG_m$ and $\bG_a$,
this indicates that we should expect that 
for $\widehat{G(O)}$ the formal completion
of $G(O)$ in $G(K)$, $\fg((t))dt/\widehat{G(O)}$ is 1-affine.

This section and the next are devoted to showing the following:

\begin{thm}\label{t:mod-g(o)-hat}

$\fg((t))dt/\widehat{G(O)}$ is 1-affine.

\end{thm}

Since, as we recalled above,
$\bB (\bA^{\infty})_0^{\wedge}$ is also not 1-affine, 
this result is a global way of encoding the finite-dimensionality
of de Rham cohomology at each point.

\subsection{Notation}

This material would be unbearable without some notational 
shorthands. Therefore, throughout through \S \ref{s:z-geom}, we let
$\sY_r = t^{-r}\fg[[t]]dt$ for $r>0$, 
and we let $\sZ_r$ denote the formal 
completion of $t^{-r}\fg[[t]]dt$ in $\fg((t))dt$.

Note that $\widehat{G(O)}$ acts on $\sZ_r$ via gauge transformations.

We let $\pi$ denote the morphism $\sY_r/G(O) \to \sZ_r/\widehat{G(O)}$.

\subsection{Structural remarks}

In this section, we will give some reductions toward proving
Theorem \ref{t:mod-g(o)-hat}. We remark from the beginning
that the hard work in proving Theorem \ref{t:mod-g(o)-hat}
will be given in the subsequent sections: in particular,
the reduction steps in the present section will not\footnote{
Except in as much as we refer to Theorem \ref{t:tame}. But
note that Theorem \ref{t:mod-g(o)-hat} is simply a formal consequence
of Theorem \ref{t:tame}, i.e., we really do need to turn back
to the geometry.
}
rely on the results
of \S \ref{s:locsys}. The role of the present section is therefore
mostly organizational: it is intended to motivate
the objectives and methods of \S \ref{s:z-tens}. 

In \S \ref{ss:tz-formulation}, 
we will reduce Theorem \ref{t:mod-g(o)-hat} 
to Theorem \ref{t:z}, which says that $\sZ_r/\widehat{G(O)}$
is 1-affine.
This reduction step
exactly imitates the deduction from \cite{shvcat} that $\bA_{dR}^{\infty}$
is 1-affine from the fact that $\bA_{dR}^n$ is.

We then obviously should try to 
push Theorem \ref{t:tame} (that $\sY_r/G(O)$ is
1-affine) as far as we can. The good news is that $\pi$ is
a nilisomorphism, so this feels possible. The bad news is that
that is not nearly enough, since $0 \into (\bA^{\infty})_0^{\wedge}$
and $(\bA^{\infty})_0^{\wedge} \to 0$ are also (while
this formal completion is not 1-affine),
and $\pi$ is a composition of morphisms that look similar to these
maps.

To this end, we will formulate Proposition \ref{p:zr-main}
in \S \ref{ss:zr-main-statement}. This result is about calculating
some tensor products of DG categories (and the map $\pi$), 
and is the main result of \S \ref{s:z-tens}. 
In \S \ref{ss:shvcat-dig-start}-\ref{ss:shvcat-dig-finish},
we will show how this
tensor product calculation implies Theorem \ref{t:z}.

\subsection{DAG}

Throughout our discussion of $\sZ_r$, we will assume we are properly 
in the setting of derived algebraic geometry. E.g., the reader should 
understand all fiber products as derived fiber products
(but most of them will be classical anyway).

\subsection{First reduction}\label{ss:tz-formulation}

The following result readily implies Theorem \ref{t:mod-g(o)-hat}

\begin{thm}\label{t:z}

$\sZ_r/\widehat{G(O)}$ is 1-affine.

\end{thm}

\begin{proof}[Proof that Theorem \ref{t:z} implies Theorem \ref{t:mod-g(o)-hat}]

The argument closely follows the proof of 1-affineness
of $\colim_n (\bA^n)_{dR}$ from \cite{shvcat} \S 12:

First, note that $\Gamma: \ShvCat_{/\fg((t))dt/\widehat{G(O)}} \to \DGCat_{cont}$
commutes with colimits and is a morphism of categories
tensored over $\DGCat_{cont}$. Indeed,
for a sheaf of categories $\mathsf{C}$ on $\fg((t))dt/\widehat{G(O)}$,
we have:

\begin{equation}\label{eq:mod-g(o)-hat-gamma}
\Gamma(\fg((t))dt/\widehat{G(O)}, \mathsf{C}) = 
\underset{r}{\lim} \, \Gamma(\sZ_r/ \widehat{G(O)},\mathsf{C})
\end{equation}

\noindent and each of the structure maps in this limit admits a
left adjoint (since that is true for 
the restriction $\QCoh(\sZ_{r+1}/\widehat{G(O)}) \to \QCoh(\sZ_r/\widehat{G(O)})$).
Therefore, this limit can also be written as a colimit, implying
the desired results.

The commutation with colimits and compatibility tensoring
by (cocomplete) DG categories formally
implies that 

\[
\Loc: \QCoh(\sZ_r/\widehat{G(O)})\mod \to \ShvCat_{/\sZ_r/\widehat{G(O)}}
\]

\noindent is fully-faithful, since we readily see that $\Gamma \circ \Loc$
is the identity.

Before\footnote{This is
the only part where we seriously distinguish between
$\fg((t))dt/\widehat{G(O)}$ and $\fg((t))dt/G(O)$.
I.e., the analysis to this point
could have been given for $\fg((t))dt/G(O)$, using $\sY_r/G(O)$ in
place of $\sZ_r/\widehat{G(O)}$ throughout.}
showing that $\Gamma$ is conservative,
we claim that for $r^{\prime} \geq r$, we have:

\[
\QCoh(\sZ_r/\widehat{G(O)}) 
\underset{\QCoh(\fg((t))dt/\widehat{G(O)})}{\otimes}
\QCoh(\sZ_{r^{\prime}}/\widehat{G(O)}) \isom 
\QCoh\big((\sZ_r \underset{\fg((t))dt}{\times} \sZ_{r^{\prime}})
/\widehat{G(O)}\big) =
\QCoh(\sZ_r/\widehat{G(O)})
\]

\noindent an isomorphism (where we remind that 
the fiber product appearing there
is the derived one).
Indeed, this follows readily from the fact that
that the restriction 
$\QCoh(\fg((t))dt/\widehat{G(O)}) \to \QCoh(\sZ_{r^{\prime}}/\widehat{G(O)})$
admits a fully-faithful left adjoint, and since the kernel of this
restriction functor acts trivially on $\QCoh(\sZ_r/\widehat{G(O)})$.

We now show that $\Gamma$ is conservative.
For $\mathsf{C} \in \ShvCat_{/\fg((t))dt/\widehat{G(O)}}$, it suffices
to check that:

\[
\Gamma(\fg((t))dt/\widehat{G(O)},\mathsf{C}) 
\underset{\QCoh(\fg((t))dt/\widehat{G(O)})} {\otimes}
\QCoh(\sZ_r/\widehat{G(O)}) \to
\Gamma(\sZ_r/\widehat{G(O)},\mathsf{C})
\] 

\noindent is an isomorphism. Indeed, since the functor of restriction to every
$\sZ_r/\widehat{G(O)}$ is tautologically conservative
(i.e., the restriction functor to $\prod_r \ShvCat_{/\sZ_r/\widehat{G(O)}}$),
this would show that $\Gamma$ is conservative.

We will do this using the expression \eqref{eq:mod-g(o)-hat-gamma}
for the left hand side. We compute (using filteredness of the colimit):

\[
\begin{gathered}
\Gamma(\fg((t))dt/\widehat{G(O)},\mathsf{C}) 
\underset{\QCoh(\fg((t))dt/\widehat{G(O)})} {\otimes}
\QCoh(\sZ_r/\widehat{G(O)}) = \\
\underset{r^{\prime} \geq r} {\colim} \, 
\Gamma(\sZ_{r^{\prime}}/\widehat{G(O)},\mathsf{C}) 
\underset{\QCoh(\fg((t))dt/\widehat{G(O)})} {\otimes}
\QCoh(\sZ_r/\widehat{G(O)}) = \\
\underset{r^{\prime} \geq r} {\colim} \, 
\Gamma(\sZ_{r^{\prime}}/\widehat{G(O)},\mathsf{C}) 
\underset{\QCoh(\sZ_{r^{\prime}}/\widehat{G(O)})} {\otimes}
\QCoh(\sZ_{r^{\prime}}/\widehat{G(O)})
\underset{\QCoh(\fg((t))dt/\widehat{G(O)})} {\otimes}
\QCoh(\sZ_r/\widehat{G(O)}) = \\
\underset{r^{\prime} \geq r} {\colim} \, 
\Gamma(\sZ_{r^{\prime}}/\widehat{G(O)},\mathsf{C}) 
\underset{\QCoh(\sZ_{r^{\prime}}/\widehat{G(O)})} {\otimes}
\QCoh(\sZ_r/\widehat{G(O)}).
\end{gathered}
\]

\noindent But for each $r^{\prime} \geq r$, we have:

\[
\Gamma(\sZ_{r^{\prime}}/\widehat{G(O)},\mathsf{C}) 
\underset{\QCoh(\sZ_{r^{\prime}}/\widehat{G(O)})} {\otimes}
\QCoh(\sZ_r/\widehat{G(O)}) = 
\Gamma(\sZ_r/\widehat{G(O)},\mathsf{C})
\]

\noindent by (the assumed) 1-affineness of $\sZ_r/\widehat{G(O)}$
and $\sZ_{r^{\prime}}/\widehat{G(O)}$. By filteredness
of the colimit, we obtain the claim.

\end{proof}

\subsection{Formulation of the main technical result}\label{ss:zr-main-statement}

The following result will be proved in
\S \ref{s:z-tens}.

\begin{prop}\label{p:zr-main}

For every sheaf of categories $\mathsf{C}$ on
$\sZ_r/\widehat{G(O)}$, the functor:

\[
\Gamma(\sZ_r/\widehat{G(O)},\mathsf{C}) 
\underset{\QCoh(\sZ_r/\widehat{G(O)})}{\otimes}
\QCoh(\sY_r/G(O)) \to
\Gamma(\sY_r/G(O),\pi^{*,\ShvCat}(\mathsf{C})) 
\]

\noindent is an equivalence.

\end{prop}

\begin{rem}

Note that Proposition \ref{p:zr-main} is a formal
consequence of Theorem \ref{t:z} (and Theorem \ref{t:tame}). 

\end{rem}

\subsection{Some generalities on sheaves of categories}\label{ss:shvcat-dig-start}

Next, we will show that Proposition \ref{p:zr-main} implies
Theorem \ref{t:z} (recall that the Proposition \ref{p:zr-main} is
the main object of Section \ref{s:z-tens}).

Note that $\pi$ is an ind-proper covering morphism. 
We will address the general question: in this
setting, what does the 1-affineness of $\sY_r/G(O)$ tell
us about the 1-affineness of $\sZ_r/\widehat{G(O)}$? 
What does it take for $\sZ_r/\widehat{G(O)}$ to be 1-affine?

We will obtain some partial results in this direction in 
\S \ref{ss:shvcat-dig-start}-\ref{ss:shvcat-dig-finish}. Note that
the deduction of Theorem \ref{t:z} from the proposition will be
given in \S \ref{ss:tensprod-implies-thm}.

\subsection{Some DAG, and $?$-pushforwards}\label{ss:?-pushforward}

Suppose $f:S \to T$ is an eventually coconnective\footnote{
Synonymously:
$\Tor$-finite.} proper morphism
of DG schemes. Recall that $f^*: \QCoh(T) \to \QCoh(S)$
admits a left adjoint, which we will denote $f_?$.

More generally, it is formal to deduce that this left adjoint $f_?$ exists if
$f$ is merely assumed to be ind-(proper and eventually coconnective).
Formation of this left adjoint commutes with base-change, so this
holds for any such morphism of prestacks.

Finally, it is easy to see that $\pi$ is a morphism of this kind.
Indeed, $\sY_r/G(O) \into \sZ_r/G(O)$ is an ind-regular embedding,
and $\sZ_r/G(O) \to \sZ_r/\widehat{G(O)}$ is a 
$\widehat{G(O)}/G(O)$-fibration.

\subsection{Descent and some applications}

We will be giving some consequences of the 
following well-known result.

\begin{lem}\label{l:proper-desc-schs}

$\QCoh$ satisfies descent for proper, 
eventually coconnective morphisms between eventually coconnective
(e.g., classical) DG schemes.
I.e., if $f: S \to T$ is a proper, eventually coconnective
covering with $S$ and $T$ themselves eventually
coconnective, then the
canonical map:

\[
\QCoh(T) \to \underset{\bDelta} {\lim} \, 
\QCoh(S \underset{T}{\times} \ldots \underset{T}{\times} S)
\]

\noindent is an equivalence.

\end{lem}

\begin{counterexample}

Let us see that all these hypotheses are actually necessary.
We will show that even if we assume $f$ is eventually coconnective,
descent can fail if $T$ is not itself eventually coconnective.

Let $A = \Sym(k[2])$ and denote the generator of $H^2(A)$ by $\beta$.
Consider $f:S = \Spec(k) \into T = \Spec(A)$.
This morphism is a closed embedding and obviously a covering.
Moreover, this morphism is eventually coconnective:
it suffices to see that $k$ is perfect as an $A$-module,
which is clear since $k = \Coker(A[2] \xar{\beta} A)$. 
But $f^*$ is \emph{not} conservative, since 
$f^*(k[\beta,\beta^{-1}]) = 0$. 

\end{counterexample}

\begin{proof}[Proof of Lemma \ref{l:proper-desc-schs}]

First, we show that $f^*$ is conservative. 
Because $T$ is eventually coconnective,
it is easy to see that $*$-restriction to all geometric points of $T$
is conservative. Then since every geometric point of $T$
lifts to $S$, we obtain the claim.

From here, the proof of descent given for finite flat maps
in \cite{shvcat} Appendix A holds verbatim:
by Beck-Chevalley, one sees that the totalization maps monadically to
$\QCoh(S)$, but then $f^*$ is obviously monadic (since it commutes
with colimits and has the left adjoint $f_?$), and it is easy to 
see that the two monads coincide.

\end{proof}

\begin{cor}

For $f$ as above, $f^{*,\ShvCat}: \ShvCat_{/T} \to \ShvCat_{/S}$ 
is conservative. 

\end{cor}

\begin{proof}

One can either show by the same argument as above
that descent holds for general sheaves of categories on $T$,
or we can use 1-affineness and the fact that the descent
limit above commutes with tensor products over $\QCoh(T)$ 
(since every morphism admits a $\QCoh(T)$-linear left adjoint).

\end{proof}

\begin{cor}\label{c:ec-desc-shvcat}

Suppose that $f: S \to T$ is an ind-(eventually coconnective
and proper covering) morphism, with $T$ an
eventually coconnective prestack. 

Then: 

\begin{enumerate}

\item $f^{*,\ShvCat}$ is conservative.

\item\label{i:ec-desc-shvcat-2} 

For every $\mathsf{C} \in \ShvCat_{/T}$, the canonical
morphism:

\[
\mathsf{C} \to \underset{\bDelta}{\lim} \, 
(f_*^{\ShvCat}f^{*,\ShvCat})^{\dot + 1}(\mathsf{C})
\]

\noindent is an equivalence. Each morphism in the
corresponding semisimplicial diagram admits a left adjoint
in the 2-category $\ShvCat_{/T}$. 

\end{enumerate}

\end{cor}

\begin{proof}

We tautologically can reduce to the case where $T$ is an 
affine DG scheme. Suppose $S = \colim S_i$ as above.
Then the pullback for sheaves of categories from 
$T$ to each $S_i$ is conservative, and therefore the pullback
to $S$ is as well. 

Then observe that $f^{*,\ShvCat}$ is comonadic: indeed,
it commutes with arbitrary limits since pullback for sheaves of categories
always does, and it is conservative by the above.
The expression for $\mathsf{C}$ as a limit then is a formal
consequence the (co)monadic formalism.

Finally, the fact that each morphism in the semisimplicial
diagram admits a left adjoint is a formal consequence
of $f$ be ind-(eventually coconnective and proper).

\end{proof}

\subsection{}\label{ss:tensprod-implies-thm}

We now deduce the following result, which immediately
shows that Proposition \ref{p:zr-main} implies 
Theorem \ref{t:z}.

\begin{prop}\label{p:proper-ec-loc-1}

Let $f: S \to T$ be an ind-(proper and eventually coconnective covering)
morphism, with $T$ an eventually coconnective prestack.
Suppose that $S$ is 1-affine.

Then $T$ is 1-affine if and only if for every 
$\mathsf{C} \in \ShvCat_{/T}$, the morphism:

\begin{equation}\label{eq:st-tens}
\Gamma(T,\mathsf{C}) \underset{\QCoh(T)}{\otimes} \QCoh(S) \to
\Gamma(S,f^{*,\ShvCat}(\mathsf{C}))
\end{equation}

\noindent is an equivalence.

\end{prop}

\begin{proof}

\step 

We now claim that $f_*^{\ShvCat}: \ShvCat_{/S} \to \ShvCat_{/T}$
commutes with colimits and is a morphism
of $\DGCat_{cont}$-module categories.
Indeed, since $f^{*,\ShvCat}$ is conservative and 
is obviously $\DGCat_{cont}$-linear, it suffices
to show that $f^{*,\ShvCat}f_*^{\ShvCat}$ has the desired properties.
Similarly, by 1-affineness of $S$, it suffices to show that
$\Gamma(S,f^{*,\ShvCat} f_*^{\ShvCat}(-))$ has the desired properties.
But by assumption, we can compute the latter as:

\[
\Gamma(T,f_*^{\ShvCat}(-)) \underset{\QCoh(T)}{\otimes} \QCoh(S) =
\Gamma(S,-) \underset{\QCoh(T)}{\otimes} \QCoh(S)
\]

\noindent which obviously has the desired properties (by
1-affineness of $S$).

\step

Next, we claim that $\Gamma(T,-): \ShvCat_{/T} \to \DGCat_{cont}$
commutes with colimits and is $\DGCat_{cont}$-linear.

Note that Corollary \ref{c:ec-desc-shvcat} \eqref{i:ec-desc-shvcat-2}
gives an expression:

\begin{equation}\label{eq:shvcat-colim}
\mathsf{C} = 
\underset{\bDelta}{\lim} \, 
(f_*^{\ShvCat}f^{*,\ShvCat})^{\dot + 1}(\mathsf{C}) =
\underset{\bDelta^{op}}{\colim} \, 
(f_*^{\ShvCat}f^{*,\ShvCat})^{\dot + 1}(\mathsf{C})
\end{equation}

\noindent for $\mathsf{C} \in \ShvCat_{/T}$,
where in the last expression we are using the left adjoints to the
canonical maps (this follows from the usual 
``limit under right adjoints = colimit under left adjoints"
format, c.f. \cite{dgcat}).

Then for $\mathsf{C} \in \ShvCat_{/T}$, we can then compute:

\[
\Gamma(T,\mathsf{C}) = 
\underset{\bDelta}{\lim} \, 
\Gamma\big(T,(f_*^{\ShvCat}f^{*,\ShvCat})^{\dot + 1}(\mathsf{C})\big)
\]

\noindent by commuting $\Gamma$ with limits. 
But since $\Gamma$ is a morphism of 2-categories,
each of the structure functors admits a left adjoint, so we can
compute this as:

\[
\Gamma(T,\mathsf{C}) = 
\underset{\bDelta^{op}}{\colim} \, 
\Gamma\big(T,(f_*^{\ShvCat}f^{*,\ShvCat})^{\dot + 1}(\mathsf{C})\big) =
\Gamma\big(S,(f^{*,\ShvCat}f_*^{\ShvCat})^{\dot}f^{*,\ShvCat}(\mathsf{C})\big)
\]

\noindent which commutes with colimits and is a morphism
of $\DGCat_{cont}$-module categories because $S$ is 1-affine
and because $f_*^{\ShvCat}$ has the desired properties.

\step

It now follows formally that $\Gamma(T,-): \QCoh(T) \to \DGCat_{cont}$
is a morphism of $\QCoh(T)\mod$-module categories. 
Then the tautological identity 
$\Gamma(T,\Loc(\QCoh(T))) = \QCoh(T)$ means that if
$\Gamma \circ \Loc = \id$,  so $\Loc$ is fully-faithful.
Therefore, it suffices to show that $\Gamma$ is conservative.

But this is clear from the identity \eqref{eq:st-tens} and
the fact that
$f^{*,\ShvCat}$ is conservative (i.e., 
Corollary \ref{c:ec-desc-shvcat}).

\end{proof}

\subsection{Conclusion}\label{ss:shvcat-dig-finish}

We finish this section with the following result, which addresses what more
can be said in Proposition \ref{p:proper-ec-loc-1} if we also
assume that the morphism $f$ is 1-affine. 

Note that this assumption is not at all obvious for the morphism
$\pi$, and therefore the next result will necessarily play an 
auxiliary role in the course of proving Proposition \ref{p:zr-main}
(and also will be key in the deduction of the 1-affineness of $\LocSys$ from
Theorem \ref{t:mod-g(o)-hat}). Therefore, the reader may safely
skip this result for the time being, and refer back to it as necessary.
(That said, it is obviously logically related to the above, and the proof
is close to that of Proposition \ref{p:proper-ec-loc-1}.)

\begin{prop}\label{p:proper-ec-loc-2}

Let $f: S \to T$ be an ind-(proper and eventually coconnective covering)
morphism, with $T$ an eventually coconnective prestack.
Suppose that $f$ is 1-affine and that $S$ is
1-affine.

Then: 

\begin{enumerate}

\item\label{i:proper-ec-1} 

$\Loc: \QCoh(T)\mod \to \ShvCat_{/T}$
is fully-faithful.

\item\label{i:proper-ec-2} 

$\Gamma: \ShvCat_{/T} \to \DGCat_{cont}$ commutes
with colimits and is a morphism of $\DGCat_{cont}$-module categories.

\item\label{i:proper-ec-3}

The category $\ShvCat_{/T}$ is generated under
colimits by objects of the form 
$f_*^{\ShvCat}(\QCoh_{/S}) \otimes \sD$
for $\sD \in \DGCat_{cont}$.

\item\label{i:proper-ec-4}

$T$ is 1-affine if and only if the morphism:

\begin{equation}\label{eq:s/t-qcoh-tens}
\QCoh(S) \underset{\QCoh(T)}{\otimes} \QCoh(S) \to
\QCoh(S \underset{T}{\times} S)
\end{equation}

\noindent is an equivalence.

\end{enumerate}

\end{prop}

\begin{proof}

From \eqref{eq:shvcat-colim}, we see that
$f_*^{\ShvCat}$ generates 
$\ShvCat_{/T}$ under colimits. Since $\ShvCat_{/S}$
is generated under colimits by objects
$\QCoh_{/S} \otimes \sD$ by 1-affineness of $S$,
we obtain \eqref{i:proper-ec-3}.

For \eqref{i:proper-ec-2}, we claim that for 
$\mathsf{C} \in \ShvCat_{/T}$, \eqref{eq:shvcat-colim}
induces an isomorphism:

\[
\underset{\bDelta^{op}}{\colim} \, 
\Gamma(T,(f_*^{\ShvCat}f^{*,\ShvCat})^{\dot + 1}(\mathsf{C})) \isom
\Gamma(T, \underset{\bDelta^{op}}{\colim} \, 
(f_*^{\ShvCat}f^{*,\ShvCat})^{\dot + 1}(\mathsf{C})) =
\Gamma(T,\mathsf{C}).
\]

\noindent Indeed, this is as in the proof of Proposition \ref{p:proper-ec-loc-1}:
we can express all the colimits in sight
as limits under the corresponding right adjoints, and $\Gamma$
commutes with all limits.

To deduce that $\Gamma(T,-)$ commutes with colimits,
we need to see that for each $[n] \in \bDelta$, the corresponding term:

\[
\Gamma(T,(f_*^{\ShvCat}f^{*,\ShvCat})^{n + 1}(\mathsf{C})) =
\Gamma(S,f^{*,\ShvCat} (f_*^{\ShvCat}f^{*,\ShvCat})^{n}(\mathsf{C}))
\]

\noindent commutes with colimits. But since $f_*^{\ShvCat}$ commutes
with colimits (since $f$ is 1-affine), and since $\Gamma(S,-)$ commutes
with colimits (since $S$ is 1-affine), we obtain the claim.

For \eqref{i:proper-ec-1}, observe that the same argument shows that
$\Gamma(T,-): \ShvCat_{/T} \to \DGCat_{cont}$ is a morphism
of $\DGCat_{cont}$-module categories. As in the proof
of Proposition \ref{p:proper-ec-loc-1}, this formally implies 
that $\Loc$ is fully-faithful.

Finally, for \eqref{i:proper-ec-4}, first note that the one implication is
clear: if $T$ is 1-affine, then the compatibility with
tensor products is standard.

Conversely, we claim that if that identity holds, then
for every $\mathsf{C} \in \ShvCat_{/T}$, the map:

\begin{equation}\label{eq:s/t-tensprod-identity}
\Gamma(T,\mathsf{C}) \underset{\QCoh(T)}{\otimes} \QCoh(S) \to
\Gamma(S,f^{*,\ShvCat}(\mathsf{C}))
\end{equation}

\noindent is an isomorphism. Note that this would imply
the claim by Proposition \ref{p:proper-ec-loc-1}.

To see \eqref{eq:s/t-tensprod-identity}, note that
by our earlier results, it suffices to check this
identity for $\mathsf{C} = f_*^{\ShvCat}(\QCoh_{/S} \otimes \sD)$
for $\sD \in \DGCat_{cont}$: indeed, everything in sight
commutes with colimits, and $\ShvCat_{/T}$ is generated
by objects of this form. 

Moreover, since each functor
is a morphism of $\DGCat_{cont}$-module categories,
we reduce to $\sD = \Vect$. 
Then the desired identity is exactly \eqref{eq:s/t-qcoh-tens}
by base-change for sheaves of categories.

\end{proof}


\section{Infinitesimal analysis: tensor product calculations via nilpotence}\label{s:z-tens}


\subsection{}

In this section, we continue the project begun in \S \ref{s:z-red},
retaining the notation of \emph{loc. cit}.

So we have reduced Theorems \ref{t:mod-g(o)-hat}
and \ref{t:z} to the calculation of some tensor products of DG categories,
namely, to Proposition \ref{p:zr-main}. Our goal
in this section is to prove this proposition.

\subsection{Outline of the method}

Recall the statement of Proposition \ref{p:zr-main}:
for $\mathsf{C} \in \ShvCat_{/\sZ_r/\widehat{G(O)}}$, 
we want to show that the horizontal arrow:

\[
\xymatrix{
\Gamma(\sZ_r/\widehat{G(O)},\mathsf{C}) 
\underset{\QCoh(\sZ_r/\widehat{G(O)})}{\otimes}
\QCoh(\sY_r/G(O)) \ar[rr] \ar@<-.4ex>[dr]_(.65){\pi_{\mathsf{C},?}} & &
\Gamma(\sY_r/G(O),\pi^{*,\ShvCat}(\mathsf{C})) \ar@<-.4ex>[dl]_{\pi_{\mathsf{C},?}} \\
& \Gamma(\sZ_r/G(O),\mathsf{C}) 
\ar@<-.4ex>[ul]_(.35){\pi_{\mathsf{C}}^*} \ar@<-.4ex>[ur]_{\pi_{\mathsf{C}}^*}
}
\]

\noindent is an equivalence. 

Some remarks on the arrows in this diagram
are in order. The functor on the left labeled $\pi_{\mathsf{C}}^*$ is induced 
by usual pullback of quasi-coherent sheaves, while the functor
on the right labeled in the same way is pullback of sections for
sheaves of categories. The functors labeled $\pi_{\mathsf{C},?}$ 
are left adjoints,
and exist for the same reason as in \S \ref{ss:?-pushforward}
(namely, that $\pi$ is ind-(proper and eventually coconnective)).
Finally, it is straightforward to see that the 
horizontal arrow is compatible with either pair of functors,
and the induced morphism of comonads, each abusively notated
as $\pi_{\mathsf{C},?}\pi_{\mathsf{C}}^*$, is an equivalence 
(justifying, somewhat, the abuse).

So our objective will be to show that each functor labeled
$\pi_{\mathsf{C},?}$ is comonadic, which would
then give the claim of Proposition \ref{p:zr-main}.

\begin{rem}

The reader should keep in mind that this result cannot be something
that e.g. could be deduced by formal means from Theorem \ref{t:tame}. 
E.g., we have
to distinguish somehow between $\sZ_r/G(O)$ (for which
the corresponding statement is false) and $\sZ_r/\widehat{G(O)}$;
and any meaningful geometric distinction between the two inevitably relies
on the full force of the results from \S \ref{s:locsys}.

\end{rem}

\subsection{What is the method and what are the difficulties?}

We should try to imitate the technique of \S \ref{s:tame}.
So we shrink our congruence subgroup to
$\cK_{r+s}$, work fiberwise over\footnote{To be precise, we
actually need to take fibers over $(t^{-r}\fg[[t]]dt/t^s\fg[[t]]dt)_{dR}$: only
the de Rham version of the leading terms space
receives a map from $\sZ_r/\widehat{\cK}_{r+s}$.} 
the appropriate leading
terms space, and then use Cousin filtrations to pass from the fibers
to the total space of the fibration.

The first two steps work fine. Shrinking the congruence subgroup
is no problem (see the proof of Proposition \ref{p:zr-main} below).
We will analyze the fiberwise geometry in \S \ref{s:z-geom},
but the moral is that fiberwise everything is as nice as the
geometry in the case of $\bG_m$ (so everything is finite type, etc.) 
and therefore can be treated by standard methods. In particular,
we can see that there is a fiberwise version of the comonadicity statements.

The trouble occurs in the last step. We want to deduce
comonadicity from fiberwise comonadicity using the Cousin resolution.
By Barr-Beck, comonadicity
is about commutation with certain totalizations, while 
Cousin resolutions involve infinite direct sums over the fibers.
The problem is that
it is not a priori clear how to commute the relevant limits and colimits
here. 

\subsection{Structure of this section}

Our principal goal in this section is therefore to justify 
commuting the relevant totalizations and direct sums.
This problem is completely impossible without new
ideas beyond the bare Barr-Beck formalism; fortunately,
such ideas are given in \cite{akhil-thesis}.

Namely, in \S \ref{ss:nilpcomonads-start}-\ref{ss:nilp-comonad-finish},
we will develop a formalism of \emph{effective monads}
following \emph{loc. cit}., which gives a strong
version of the (comonadic) Barr-Beck conditions, and which
is of quantitative nature.
In the remainder of the section, we will show how these ideas
can be applied to Proposition \ref{p:zr-main}. 

We have structured this section so that all the geometric
fiberwise analysis is postponed to \S \ref{s:z-geom}, in order
to highlight the role of effectiveness. Therefore,
we formulate two results, whose proofs are postponed, and
deduce Proposition \ref{p:zr-main} from these.

The first, and more interesting, of these two statements 
is Proposition \ref{p:ourmonad-eff}, which says
that some monads are effective. 
At the end of this section,
in \S \ref{ss:ourmonad-redn}, we will show how to reduce 
this proposition to Lemma \ref{l:pi-n-bdded}, which says that
these monads are cohomologically bounded.
This latter result bears a 
strong similarity to Proposition \ref{p:tame-amplitude};
and indeed, in \S \ref{s:z-geom} we will see that it
is amenable to a similar fiberwise analysis.

The second of these statements, Proposition \ref{p:zr-eta-1aff}, is 
more mild: it states that the geometric 
fibers of $\sZ_r/\widehat{\cK}_{r+s}$ over the leading
terms space are 1-affine.

In \S \ref{ss:eff-application}, we will show that these
two propositions imply Proposition \ref{p:zr-main}.

\subsection{Nilpotence and (co)monads}\label{ss:nilpcomonads-start}

As was describe above, we will now give some conditions
on an adjunction that are stronger than the comonadic Barr-Beck
conditions: the idea is to quantitatively carefully analyze
the decay of the totalization tower (as a pro-object). 

This material was heavily 
influenced by Akhil Mathew's senior thesis \cite{akhil-thesis} and 
by \cite{akhil-equivariantdescent}, 
though the ideas date back at least to Bousfield \cite{bousfield}.

\begin{rem}

We suggest that the reader approach the material that follows somewhat 
non-linearly. 

It is essential for reading the remainder of this section to
learn the basic definitions and consequences
from what follows: this is why we are opening this section with 
it.\footnote{Fortunately, the ideas from \cite{akhil-thesis}
are quite natural, and hopefully this should not be too difficult.}
However, we suggest
to skip any arguments that seem technical
or uninteresting and return back to them later (we have tried
to indicate where these points likely are). 

Note that Lemma \ref{l:effmonad-tstr} plays a key role (it is the source
of our examples), but is somewhat abstract (and probably
abstruse). A somewhat casual sense for the
thing is appropriate at first pass (c.f. Remark \ref{r:effmonad-heur}),
e.g. to the point that the reader is not surprised that we are studying
$t$-structures in what follows.

The finer points can then be filled in as the reader likes while
strolling through the remainder of the section.

\end{rem}

\subsection{}

Fix a cocomplete\footnote{This is mostly irrelevant in what follows, but it is convenient to record it once so
we do not need to keep track of it.} 
DG category $\sC$ for the foreseeable future (i.e., through \S \ref{ss:nilp-comonad-finish}). 

Let $\ldots \to \sF^2 \to \sF^1 \to \sF^0$ be a $\bZ^{\geq 0}$-indexed inverse system in $\sC$. 

\begin{defin}

The system $(\sF^i)_{i \geq 0}$ is  
\emph{$n$-nilpotent} for some $n \in \bZ^{\geq 0}$ 
if for every $i$, the morphism $\sF^{i+n} \to \sF^i$ is nullhomotopic.\footnote{
The nullhomotopy is \emph{not} part of the data; i.e., 
$n$-nilpotence is a property, not a structure.}
It is \emph{nilpotent} if it is $n$-nilpotent for some $n$.

\end{defin}

\begin{lem}

If $(\sF^i)$ is nilpotent, then the resulting
object of $\Pro(\sC)$ is zero. 

\end{lem}

\begin{proof}

Say $(\sF^i)$ is $n$-nilpotent.
An object of $\Pro(\sC)$
is zero if and only if every map to an object
$\sG \in \sC \subset \Pro(\sC)$ is nullhomotopic.
A map from the limit (in $\Pro(\sC)$) of our diagram to $\sG$ factors through some $\sF^i$,
but then the map from the limit is nullhomotopic, since it factors 
through $\sF^{i+n}$.

\end{proof}

\begin{cor}

For any DG functor $F:\sC \to \sD$, we have $\lim_{i \geq 0} F(\sF^i) = 0$.

\end{cor}

\subsection{}

We also record the following basic stability for nilpotency,
c.f. \cite{akhil-thicksubcats} Proposition 3.5.

\begin{lem}\label{l:nilp-cones}

Nilpotent diagrams as above are closed under finite colimits.
That is, if $i \mapsto \sF^i$ and $i \mapsto \sG^i$ are 
$\bZ^{\geq 0}$-indexed and nilpotent
diagrams equipped with compatible maps $\sF^i\to \sG^i$, then
the diagram $i \mapsto \Coker(\sF^i \to \sG^i)$ is also nilpotent.

\end{lem}

Note, however, that 
$n$-nilpotent objects (with $n$ fixed) 
are \emph{not} closed under cones.

\subsection{}

The following lemma will play a key role in what follows, and it is our reason
for keeping track of the rate of nilpotence in the above. 
However, the reader may safely skip it for now
and refer back to it later.

\begin{lem}\label{l:nilp-sums}

Suppose that we are given $\bZ^{\geq 0}$-inverse systems $\sF_j^i$ indexed
by some $j$ in some set $J$, and suppose that each of these systems 
is $n$-nilpotent.
Then $(\oplus_j \sF_j^i)$ is $n$-nilpotent as well.

\end{lem}

\begin{proof}

Fix $i$. Choose nullhomotopies of each of the maps 
$\sF_j^{i+n} \to \sF_j$.
This induces a nullhomotopy of the resulting map 
$\oplus_j \sF_j^{i+n} \to \oplus_j \sF_j^i$
as well, giving the claim.

\end{proof}

\begin{rem}

This lemma does not work for mere 
nilpotence (i.e., if we do not choose the degree of nilpotence uniformly
in $j$).

\end{rem}

\subsection{} 

Let $T:\sC \to \sC$ be a monad. Recall that for each $\sF \in \sC$, 
the bar construction gives an augmented cosimplicial object:

\[
\sF \to T(\sF) \rightrightarrows T^2(\sF) \ldots
\]

We denote this cosimplicial object by $T^{\dot+1}(\sF)$ where convenient.
For each integer $i \geq 0$, let $\Tot^i T^{\dot+1}(\sF)$ denote the
limit of this diagram over $\bDelta_{\leq i} \subset \bDelta$ (the subcategory
of simplices of size $\leq i$).

\begin{defin}

$T$ is \emph{$n$-effective} if for every $\sF \in \sC$,
the pro-system:

\[
i \mapsto \Coker(\sF \to \Tot^i T^{\dot +1}(\sF))
\] 

\noindent is $n$-nilpotent functorially in $\sF$.

More precisely, $T$ is $n$-effective if the system:

\[
i\mapsto \id_{\sC} \to \Coker(\id_{\sC} \to \Tot^i T^{\dot+1}(-)) \in \End(\sC)
\]

\noindent is $n$-nilpotent.

We say that $T$ is \emph{effective} if it is $n$-effective for some $n$.

\end{defin}

\begin{rem}

If $T$ is defined by an adjunction 
$F:\sC \rightleftarrows \sD : G$ (so $T = GF$)
and is effective, then the functor $F$ is comonadic.
Indeed, for any $\sF \in \sC$, the map:

\[
\sF \to \Tot (GF)^{\dot+1}(\sF) = \underset{i}{\lim} \Tot^i (GF)^{\dot +1}(\sF)
\]

\noindent is then an isomorphism, 
and formation of this limit is preserved by any DG functor
(since each $\Tot^i$ is a finite limit, and then the isomorphism in $\Pro(\sC)$ implies preservation
of the limit in $i$), so in particular by $F$. 

\end{rem}

\begin{example}

\cite{akhil-thesis} provides many examples of $n$-effective monads, to great effect.
We will provide some additional examples in what follows. 

\end{example}

\subsection{}

Effectivity of a monads has the following nice consequence for computing 
tensor products.

Let $\sA \in \Alg(\DGCat_{cont})$ act on $\sC$. 
Let $\sM \in \DGCat_{cont}$ be a right module category for $\sA$.
Suppose that we are given $\sA$-linear continuous functors 
$F:\sC \rightleftarrows \sD : G$.

\begin{prop}\label{p:tens-eff}

If $T \coloneqq GF$ is effective, then 
$\sC \otimes_{\sA} \sM \to \sD \otimes_{\sA} \sM$ is comonadic.
More precisely, if $T$ is $n$-effective, then the induced monad on 
the tensor product is also $n$-effective.

\end{prop}

\begin{proof}

Let $T^{\prime}$ be the induced monad on $\sC \otimes_{\sA} \sM$, corresponding to the adjunction
$(F\otimes_{\sA} \id_{\sM}, G\otimes_{\sA} \id_{\sM})$.

We have a factorization:

\[
\sC \times \sM \rar{T\times \id_{\sM}} \sC \times \sM \to \sC \otimes_{\sA} \sM
\]

\noindent with an obvious $\sA$-bilinear structure on the composition, inducing the functor 
$T^{\prime}: \sC \otimes_{\sA} \sM  \to \sC \otimes_{\sA} \sM$.
Moreover, 
the same holds for $\Coker(\id_{\sC} \to T^i)\times \id_{\sM}$ versus 
$\Coker(\id_{\sC \otimes_{\sA} \sM} \to (T^{\prime})^i)$. Passing to partial totalizations
and using the $n$-effective structure on $T$ 
(i.e., choosing the relevant nullhomotopies) then induces the desired
$n$-effective structure on $T^{\prime}$.

\end{proof}

\subsection{A pointwise criterion for comonadicity}\label{ss:nilp-comonad-finish}

Let $S$ be a given finite type $k$-scheme, and recall that
$D(S) \coloneqq \QCoh(S_{dR}) \in \DGCat_{cont}$ 
denotes its category of $D$-modules.\footnote{This next result also admits
a quasi-coherent variant, 
but we will apply it in the present form.}

We record the following consequence of Lemma \ref{l:nilp-sums},
which the reader may safely skip for now and refer back to when
it is applied below.

\begin{prop}\label{p:ptwise-eff/equiv}

Suppose we are given a diagram:

\[
\xymatrix{
\sC_1 \ar[rr]^{\Psi} \ar[dr]^{F_1} && \sC_2 \ar[dl]_{F_2} \\
& \sD
}
\]

\noindent in $D(S)\mod \coloneqq D(S)\mod(\DGCat_{cont})$
such that:

\begin{itemize}

\item For every $\eta \in S_{dR}$ a geometric point,
the induced functor on fibers:

\[
\Psi_{\eta}: \sC_{1,\eta} \to \sC_{2,\eta}
\]

\noindent is an equivalence.

\item  $F_1$ and $F_2$ admit $D(S)$-linear (continuous) right adjoints
$G_1$ and $G_2$.

\item $\Psi$ induces an equivalence of 

\[
F_1 G_1 \to F_2 G_2
\]

\noindent of $D(S)$-linear comonads on $\sD$.

\item The monad $G_1F_1$ is effective.

\end{itemize}

Then $\Psi$ is an equivalence.

\end{prop} 

We will use the following lemma.

\begin{lem}\label{l:ptwise-eff}

Suppose $T:\sC \to \sC$ is a $D(S)$-linear monad.

Suppose that for every $\eta \in S_{dR}$ a geometric point,
the monad $T_{\eta}: \sC_{\eta} \to \sC_{\eta}$
an $n$-effective monad, where $n$ can be chosen
independent of $\eta$.

Then $T$ is effective.

\end{lem}

Note that the converse of Lemma \ref{l:ptwise-eff}
holds by Proposition \ref{p:tens-eff}.

\begin{proof}[Proof of Lemma \ref{l:ptwise-eff}]

We will deduce this from Lemma \ref{l:nilp-sums} using 
the Cousin resolution.

Recall that for any $\sF \in \sC$, the Cousin resolution gives a functorial
finite filtration of $\sF$ 
where the subquotients are of the form:

\[
\underset{\dim(\ol{\eta}) = \ell} {\oplus} \, i_{\eta,*,dR} i_{\eta}^!(\sF)
\]

\noindent where $\eta \in S$ is a schematic point, and with
$i_{\eta}$ denoting the embedding of this point.

Since formation of this Cousin resolution commutes with
$D(S)$-linear functors, it induces a similar filtration
on the $\Tot$ tower of $\sF$ (associated with the
monad $T$): this time the subquotients are:

\[
\underset{\dim(\ol{\eta}) = \ell}{\oplus} \, i_{\eta,*,dR} 
\Tot^i (T_{\eta})^{\dot + 1}(i_{\eta}^!\sF)
\]

\noindent since these partial totalizations are finite limits,
so commute with everything in sight.

Applying the fiberwise hypothesis on 
$T$ and Lemma \ref{l:nilp-sums}, 
we see that the tower:

\[
\begin{gathered}
i\mapsto \Coker\Big(
\underset{\dim(\ol{\eta}) = \ell}{\oplus} \, i_{\eta,*,dR} i_{\eta}^!(\sF)
\to
\underset{\dim(\ol{\eta}) = \ell}{\oplus} \, i_{\eta,*,dR} 
\Tot^i T_{\eta}^{\dot + 1}(i_{\eta}^!\sF)
\Big) = \\
\underset{\dim(\ol{\eta}) = \ell}{\oplus} \, i_{\eta,*,dR}
\Coker\big(
i_{\eta}^!(\sF)
\to
\Tot^i T_{\eta}^{\dot + 1}(i_{\eta}^!\sF)
\big)
\end{gathered}
\]

\noindent is $n$-nilpotent (functorially in $\sF$).
But since the Cousin filtration is finite, Lemma \ref{l:nilp-cones}
gives the desired conclusion.

\end{proof}

\begin{proof}[Proof of Proposition \ref{p:ptwise-eff/equiv}]

We claim that the monad $G_2F_2$ is also effective.

Indeed, since $G_1F_1$ is effective, by Proposition \ref{p:tens-eff},
the resulting fiberwise functors are uniformly (in $\eta$) $n$-effective
for some $n$. Since $\Psi$ equates the two pointwise situations,
we see that $G_2F_2$ is also uniformly pointwise $n$-effective.
Then by Lemma \ref{l:ptwise-eff}, we deduce that $G_2F_2$
is effective.

Therefore, the functor $F_2$ is comonadic. Since $F_1$ is
also comonadic (also by effectiveness), we obtain 
the result, since the two comonads on $\sD$ are identified.

\end{proof}

\subsection{Back to local systems}\label{ss:backtolocsys}

We now return to proving Proposition \ref{p:zr-main}.

Before proceeding, we need to introduce some notation relevant to
our need to shrink the congruence subgroup $G(O)$.

Let $\widetilde{\sY}_r$ denote:

\[
\widehat{G(O)} \overset{G(O)}{\times} \sY_r
\] 

\noindent where this notation indicates that we take
the quotient with respect to the diagonal action
of $G(O)$ via its gauge action on $\sY_r$ and
its right action on $\widehat{G(O)}$.

We have a canonical $\widehat{G(O)}$-equivariant map
$\widetilde{\sY}_r \to \sZ_r$ with:

\[
\widetilde{\sY}_r/\widehat{G(O)} =
\sY_r/G(O) \rar{\pi} \sZ_r/\widehat{G(O)}.
\]

For $N \geq 0$, let $\widehat{\cK}_N$ denote the formal
completion of the congruence subgroup $\cK_N$ inside
of $G(K)$.

Finally, let $\pi_N$ (resp. $\widetilde{\pi}_N$) denote the induced maps:

\[
\begin{gathered} 
\pi_N:\sY_r/\cK_N \to
\sZ_r/\widehat{\cK}_N 
\\
\widetilde{\pi}_N:\widetilde{\sY}_r/\widehat{\cK}_N \to
\sZ_r/\widehat{\cK}_N 
\end{gathered}
\]

\noindent Note that this morphism is equivariant for 
the obvious action of 
$(G(O)/K_N)_{dR}$, and reducing modulo this group prestack
recovers the morphism $\pi$.

\subsection{Effectiveness for $\sZ_r$}

We delay the proof of the following result. As was said earlier,
we will make some progress towards its proof at the end
of this section, but a complete proof will require results
from \S \ref{s:z-geom}.

\begin{prop}\label{p:ourmonad-eff}

For every\footnote{Actually, the methods we use can readily
be extended to show this result for all $N \geq 0$. But imposing
the hypothesis makes the exposition a little less clunky.}
$N \geq r+s$, the monads
$\pi_N^*\pi_{N,?}$ on $\QCoh(\sY_r/ \cK_N)$ 
and $\widetilde{\pi}_N^*\widetilde{\pi}_{N,?}$ on
$\QCoh(\widetilde{\sY}_r/\widehat{\cK}_N)$
are effective. 

\end{prop}

\subsection{}

Observe that $\sZ_r$ has a canonical map:

\[
\sZ_r \to (t^{-r}\fg[[t]]dt/t^s\fg[[t]]dt)_{dR}
\]

\noindent that is equivariant for the action of $\widehat{\cK}_{r+s}$
(where the target is equipped with the trivial action).

For $\eta \in (t^{-r}\fg[[t]]dt/t^s\fg[[t]]dt)_{dR}$ a field-valued point,
we let $\sZ_{r,\eta}$ denote the fiber of $\sZ_r$ at $\eta$,
which is equipped with an induced action of $\widehat{\cK}_{r+s}$.

The next result will be proved in \S \ref{s:z-geom}.

\begin{prop}\label{p:zr-eta-1aff}

$\sZ_{r,\eta}/\widehat{\cK}_{r+s}$ is 1-affine.

\end{prop}

\subsection{}\label{ss:eff-application}

We now show that the above results imply 
Proposition \ref{p:zr-main}.

\begin{proof}[Proof that Propositions \ref{p:ourmonad-eff} and 
\ref{p:zr-eta-1aff} imply Proposition \ref{p:zr-main}]

\step 

First, a reduction: we claim that it is enough to
show that for some $N $
and for every sheaf of categories $\widetilde{\mathsf{C}}$
on $\sZ_r/\widehat{\cK}_N$, the functor:

\begin{equation}\label{eq:zr-main-redn}
\Gamma(\sZ_r/\widehat{\cK}_N,\widetilde{\mathsf{C}}) 
\underset{\QCoh(\sZ_r/\widehat{\cK}_N)} {\otimes}
\QCoh(\widetilde{\sY}_r/\widehat{\cK}_N) \to
\Gamma(\widetilde{\sY}_r/\widehat{\cK}_N,
\widetilde{\pi}_N^{*,\ShvCat}(\widetilde{\mathsf{C}})) 
\end{equation}

\noindent is an equivalence.

Indeed, suppose we know this result for some $N$.
Let $\mathsf{C} \in \ShvCat_{/\sZ_r/\widehat{G(O)}}$ be given,
and let $\widetilde{\mathsf{C}} \in \ShvCat_{/\sZ_r/\widehat{\cK}_N}$ be its
pullback.
Then both sides of \eqref{eq:zr-main-redn} are acted
on by $D(G(O)/\cK_N) = \QCoh(\widehat{G(O)}/\widehat{\cK}_N)$,
and passing to invariants = coinvariants (c.f. \cite{dario-*/!}),
we obtain the result for $N = 0$.

\step

So we need to check \eqref{eq:zr-main-redn} for $N = r+s$,
for $s$ as in Theorem \ref{t:jumps}.
We consider the geometry over the de Rham leading
term space $(t^{-r}\fg[[t]]dt/t^s\fg[[t]]dt)_{dR}$,
and will apply Proposition \ref{p:ptwise-eff/equiv}. 

Recall that we have a diagram:

\begin{equation}\label{eq:zr-main-comonads}
\vcenter{
\xymatrix{
\Gamma(\sZ_r/\widehat{\cK}_N,\widetilde{\mathsf{C}}) 
\underset{\QCoh(\sZ_r/\widehat{\cK}_N)} {\otimes}
\QCoh(\widetilde{\sY}_r/\widehat{\cK}_N) \ar[rr] \ar@<-.4ex>[dr] & & 
\Gamma(\widetilde{\sY}_r/\widehat{\cK}_N,
\widetilde{\pi}_N^{*,\ShvCat}(\widetilde{\mathsf{C}}))  \ar@<-.4ex>[dl] \\
& 
\Gamma(\sZ_r/\widehat{\cK}_N,\widetilde{\mathsf{C}}) 
\ar@<-.4ex>[ul] \ar@<-.4ex>[ur]
}
}
\end{equation}

\noindent in which each triangle is commutative, and inducing
the same comonads on the bottom term.

First, note that the monad on the top left corner
of \eqref{eq:zr-main-comonads} is effective. Indeed,
by Proposition \ref{p:tens-eff}, it suffices to verify this
for the monad 
$\widetilde{\pi}_N^*\widetilde{\pi}_{N,?}$ on 
$\QCoh(\widetilde{\sY}_r/\widehat{\cK}_N)$,
which is given by Proposition \ref{p:ourmonad-eff}.

Next, note that over geometric
points $\eta$ of $(t^{-r}\fg[[t]]dt/t^s\fg[[t]]dt)_{dR}$,
the horizontal comparison functor of \eqref{eq:zr-main-comonads}
is an equivalence. Indeed, this follows Proposition \ref{p:zr-eta-1aff},
i.e., from 1-affineness of $\sZ_{r,\eta}/\widehat{\cK}_N$
(and 1-affineness of 
$\widetilde{\sY}_{r,\eta}/\widehat{\cK}_N$, though we
even know that $\widetilde{\sY}_r/\widehat{\cK}_N$ is 1-affine by
Theorem \ref{t:tame} plus 1-affineness of $(G(O)/\cK_N)_{dR}$).

Therefore, Proposition \ref{p:ptwise-eff/equiv} applies,
and we obtain the desired result.

\end{proof}

\subsection{$t$-structures}\label{ss:ourmonad-redn}

Next, we will 
show how to reduce Proposition \ref{p:ourmonad-eff} 
to the following more plausible lemma.

Let $N \geq 0$ be an integer.
Recall from \S \ref{s:tame} that $\QCoh(\sY_r/\cK_N)$ has a canonical
$t$-structure, which is characterized by the fact that the pullback to
$\QCoh(\sY_r)$ is $t$-exact.

\begin{lem}\label{l:pi-n-bdded}

For all $N \geq r+s$,
the monad $\pi_N^*\pi_{N,?}: \QCoh(\sY_r/\cK_N) \to \QCoh(\sY_r/\cK_N)$ 
is left $t$-exact, conservative, and of bounded cohomological amplitude.

\end{lem}

We postpone the proof of this result to \S \ref{s:z-geom}.
In the remainder of this section, we will show how to deduce
Proposition \ref{p:ourmonad-eff} from it.

\subsection{Nilpotence and $t$-structures}\label{ss:nilp-tstr}

We begin with the following general result, which connects
$t$-structures with effectiveness. 

\begin{lem}\label{l:effmonad-tstr}

Suppose that $(\sA,\ast)$ is a monoidal DG category\footnote{The reader should
think of the convolution category $\QCoh(\sY \times \sY)$ for 
$\sY$ a smooth Artin stack.}
with a $t$-structure of
finite homological dimension.

Let $A \in \sA$ be an associative algebra. Suppose that $A$
lies in cohomological degrees $\geq 0$,
and moreover is \emph{(right) flat}, i.e., $A \ast -: \sA \to \sA$ is
left $t$-exact.

Suppose moreover that\footnote{Let us spell out this condition more
explicitly. Since $A$ lies in degrees $\geq 0$, the fact that
the cone of the unit map is in degrees $\geq 0$ means that
$\e_{\sA}$ lies in cohomological degrees $\geq 0$ and
$H^0(\e_{\sA}) \to H^0(A) \in \sA^{\heart}$ is a monomorphism.
Similarly, asking that this cone be flat (with $A$ already flat) 
is equivalent to asking that for every
$\sF \in \sA^{\geq 0}$, $H^0(\sF) \to H^0(A \ast \sF) \in \sA^{\heart}$
is a monomorphism.}
 $\Coker(\e_{\sA} \to A) \in \sA^{\geq 0}$
and is also (right) flat (in the above sense). Finally, suppose that
$\Coker(\e_{\sA} \to A)$ is bounded cohomologically from 
above.\footnote{In the application, $\e_{\sA}$ and $A$ are each
bounded from above.}

Then the monad $A$ defines on $\sA$ is effective. More generally,
for any $\sC$ an $\sA$-module category, the monad on $\sC$
defined by $A$ is effective.

\end{lem}

\begin{rem}\label{r:effmonad-heur}

Heuristically, the idea is that the discrepancy between
$\e_{\sA}$ and $\Tot^i A^{\dot+1}$ is going cohomologically
off to infinity (on the right), so ``of course" the monad is
effective. 

\end{rem}

\begin{example}

This lemma gives a $t$-structure based argument 
that $\Rep(G) \to \Vect$ is comonadic and remains so
after tensoring with objects of $\DGCat_{cont}$. Recall that
the standard proof uses the Beck-Chevalley method, and is
of a substantially different nature. 

E.g., here is one advantage of the
$t$-structure method, which is relevant to our applications. 
The easiest way to see that 
the $?$-pushforward functor $\QCoh((\bA^1)_0^{\wedge})\to \Vect$
is comonadic is to use Cartier duality to identify this functor with
the forgetful functor $\Rep(\bG_a) \to \Vect$ and then use
the Beck-Chevalley method. Translating the Beck-Chevalley method
directly here is a bit unnatural: it requires us to use the (formal)
group structure on $(\bA^1)_0^{\wedge}$, so cannot be adapted
to treat problems involving formal completions without a group
structure. But of course,
this $t$-structure argument \emph{does} readily adapt.\footnote{But
we also see the disadvantage of the $t$-structure method:
it is not rich enough to see that $\Oblv:\QCoh(\bB G(O)) \to \Vect$
is comonadic, which the Beck-Chevalley method is.}

\end{example}

\begin{proof}[Proof of Lemma \ref{l:effmonad-tstr}]

Suppose that $\sA$ has cohomological dimension $\delta_1$
and that $\Coker(\e_{\sA} \to A)$ lies in cohomological
degrees $\leq \delta_2$.
We will show that the $\bZ^{\geq 0}$-indexed pro-system:

\[
i \mapsto \Coker (\e_{\sA} \to \Tot^i A^{\ast \dot + 1}) 
\]

\noindent is $(\delta_1+\delta_2+1)$-nilpotent (which
obviously implies the conclusion of the lemma).

Define $J \coloneqq \Ker(\e_{\sA} \to A) = \Coker(\e_{\sA} \to A)[-1]$.
Recall (see e.g. \cite{akhil-equivariantdescent} Proposition 2.8)
that $\Tot^i A^{\ast \dot + 1}$ is canonically isomorphic
$\Coker(J^{\ast i+1} \to \e_{\sA})$ in a way 
compatible with varying $i$.

Therefore, $\Coker (\e_{\sA} \to \Tot^i A^{\ast \dot + 1}) \simeq 
J^{\ast i+1}[1]$, so we need to show that for every $i$, the canonical
morphism:

\[
J^{\ast i+ \delta_1+\delta_2+2} \to J^{\ast i+1}
\]

\noindent is nullhomotopic.

We claim that $J^{\ast i}$ lies in cohomological degrees $\geq i$
for each $i$. Indeed, this is clear from the identity
$J = \Coker(\e_{\sA} \to A)[-1]$ and the assumptions on the latter.

Now recall that $J$ 
lies in cohomological degrees $[1,\delta_2+1]$ by assumption.
Since $J^{\ast \delta_1+\delta_2+2}$ lies in degrees $\geq \delta_1+\delta_2+2$,
and since $J$ lies in degrees $\leq \delta_2+1$, the canonical
morphism $J^{\ast \delta_1+\delta_2+2} \to J$ must be nullhomotopic,
since $\sA$ has cohomological dimension $\delta_1$.
Tensoring with $J^{\ast i}$, we obtain the claim in the general case.

\end{proof}

\subsection{Proof that Lemma \ref{l:pi-n-bdded} implies Proposition \ref{p:ourmonad-eff}}\label{ss:effmonad-start}

We will show this implication in the remainder of this section.

Here is a sketch of the argument. For
$\pi_{N,?}\pi_N^*$, we will see that the result is
essentially formal from Lemmas \ref{l:pi-n-bdded} and 
\ref{l:effmonad-tstr}:
the only missing ingredient is a translation
between kernels and functors. The proof will 
occupy \S \ref{ss:ourmonad-tstr-start}-\ref{ss:ourmonad-tstr-finish}

Then in \S \ref{ss:effmonad-finish}, we will deduce the effectiveness
of $\widetilde{\pi}_{N,?}\widetilde{\pi}_N^*$ from
a simple formal argument.

\subsection{}\label{ss:ourmonad-tstr-start}

Let: 

\[
\sA = \sA_N \coloneqq \QCoh(\sY_r/\cK_N \times \sY_r/\cK_N) = 
\TwoHom_{\DGCat_{cont}}\big(\QCoh(\sY_r/\cK_N),\QCoh(\sY_r/\cK_N)\big)
\]

\noindent 
(the latter identification being justified by Proposition \ref{p:tame-cpt-gen}).

Observe that $\sA$ has a canonical $t$-structure, as in \S \ref{s:tame}.
Moreover, this $t$-structure has bounded homological dimension: this is immediate
from Proposition \ref{p:tame-amplitude} (for the group $G \times G$)
and the explicit description of our compact generators as pullbacks
from $\bB \cK_N$ of bounded complexes of finite-dimensional representations
(noting also that these compact objects are evidently closed under truncations).

Note that the unit object (for the monoidal product) is
contained in the heart of the $t$-structure: it is 
$\Delta_*(\sO_{\sY_r/\cK_N})$,
so the claim follows from affineness of the diagonal of $\sY_r/\cK_N$.

\subsection{}

Let $A = A_N \in \QCoh(\sY_r/\cK_N \times \sY_r/\cK_N)$ correspond to the 
monad $\pi_N^*\pi_{N,?}$. So we have an identification:

\[
\pi_N^*\pi_{N,?}(\sF) = 
p_{2,*}(p_1^*(\sF) \underset{\sO_{\sY_r/\cK_N \times \sY_r/\cK_N}}{\otimes} A)
\]

\noindent where the $p_i$ are the projections. (We recall that
$p_{2,*}$ is continuous by our analysis of $\QCoh$ of the quotients
$\sY_r/\cK_N$.)
 
It remains to check that $A$ lies in degrees
$\geq 0$, is bounded above,\footnote{Since the unit of
our monoidal category lies in the heart of the $t$-structure,
this is equivalent to the condition that $\Coker(\e_{\sA} \to A)$ be
cohomologically bounded above from the statement
of Lemma \ref{l:effmonad-tstr}.}
and then all the same for $\Coker(\Delta_*(\sO_{\sY_r/\cK_N}) \to A)$.

\subsection{}\label{ss:ourmonad-tstr-finish}

We will use Lemma \ref{l:pi-n-bdded} to analyze $A$.
The following result provides the necessary
transition between kernels and functors.

\begin{prop}\label{p:kernel-pshfwd}

\begin{enumerate}

\item

The pushforward 
$p_{2,*}: \QCoh(\sY_r/\cK_N \times \sY_r/\cK_N) \to 
\QCoh(\sY_r/\cK_N)$ is left $t$-exact, conservative,
and cohomologically bounded.

\item 

An object $\sF \in \QCoh(\sY_r/\cK_N \times \sY_r/\cK_N)$
lies in degrees $\geq 0$ if and only if 
$p_{2,*}(\sF) \in \QCoh(\sY_r/\cK_N)^{\geq 0}$.

\end{enumerate}

\end{prop}

\begin{proof}

The left $t$-exactness is immediate, and 
the fact that it is cohomologically bounded
follows from Proposition \ref{p:tame-amplitude}.
Since $N \geq r+s > 0$, $\cK_N$ is pro-unipotent, giving
conservativeness.

For the second part, 
let $ \sF \in \QCoh(\sY_r/\cK_N \times \sY_r/\cK_N)$ with 
$p_{2,*}(\sF) \in \QCoh(\sY/\cK_N)^{\geq 0}$. 
Then $\Gamma(\sF)$ is in degrees $\geq 0$, meaning
that:

\[
\{0\} = \Hom(\sO_{\sY_r/\cK_N \times \sY_r/\cK_N}[1],\sF) = 
\Omega^{\infty}(\Gamma(\sF)[-1]) \in \Gpd.
\]

\noindent Since $\sO_{\sY_r/\cK_N \times \sY_r/\cK_N}[1]$
generates (under colimits) the subcategory of objects in degrees $<0$,
this gives the claim.

\end{proof}

Now observe that we tautologically have
$p_{2,*}(A) = \pi_N^* \pi_{N,?}(\sO_{\sY_r/\cK_N}) 
\in \QCoh(\sY_r/\cK_N)$. By Lemma \ref{l:pi-n-bdded},
this complex lies in degrees $\geq 0$,
so $A$ does as well. 

Moreover, we claim that $A$ is (right) flat. Recall that this means that for
every $\sK \in \QCoh(\sY_r/\cK_N \times \sY_r/\cK_N)^{\geq 0}$,
the convolution $A \ast \sK$ also lies in degrees $\geq 0$.
We use the same method as above: 

\[
p_{2,*}(A \ast \sK) =  \pi_N^* \pi_{N,?} p_{2,*}(\sK)
\]

\noindent and $p_{2,*}(\sK)$ lies in degrees $\geq 0$ by assumption,
so Lemma \ref{l:pi-n-bdded} gives the claim.

Next, we claim that $A$ is bounded cohomologically from above.
We should check that if we pullback $A$ to $\sY_r \times \sY_r$
and take global sections, the resulting complex of
vector spaces is bounded above.
By the usual functoriality for kernels, we see that this complex computes:

\[
\Gamma(\sY_r,\alpha^*\pi_{N,?}\pi_N^*\alpha_*(\sO_{\sY_r})) 
\]

\noindent where $\alpha$ denotes the structure
map $\sY_r \to \sY_r/\cK_N$, and this complex is
bounded above since $\alpha$ is flat and affine and since
$\pi_{N,?}\pi_N^*$ has bounded amplitude
by Lemma \ref{l:pi-n-bdded}.

We now need to see that the above holds for
  
\[
\Coker(\Delta_*(\sO_{\sY_r/\cK_N}) \to A)
\]

\noindent as well. The cohomological boundedness from above
is clear from the case of $A$, so it remains to show that
it is bounded below by $0$ and (right) flat. 

As in the case of $A$,
it suffices to show that the endofunctor of
$\QCoh(\sY_r/\cK_N)$ defined by this kernel is left $t$-exact.
Note that this endofunctor is:

\[
\Coker(\id \to \pi_N^*\pi_{N,?}).
\]

\noindent Therefore, it suffices to see that for every
$\sF \in \QCoh(\sY_r/\cK_N)^{\geq 0}$:

\[
\Coker(\sF \to \pi_N^*\pi_{N,?}(\sF)) \in \QCoh(\sY_r/\cK_N)^{\geq 0}.
\]

This is standard from the fact that $\pi_{N,?}$ is $t$-exact
and conservative. Indeed, it therefore suffices to see that this cone
is in degrees $\geq 0$ after applying $\pi_{N,?}$. The
map:

\[
\pi_{N,?}(\sF) \to \pi_{N,?}\pi_N^*\pi_{N,?}(\sF)
\]

\noindent admits an obvious splitting (from the counit of the adjunction),
so its cone is a summand of 
$ \pi_{N,?}\pi_N^*\pi_{N,?}(\sF)$, and this latter complex is in degrees
$\geq 0$ by Lemma \ref{l:pi-n-bdded}.

This concludes the proof of Proposition \ref{p:ourmonad-eff}
for the monad $\pi_N^*\pi_{N,?}$.

\subsection{}\label{ss:effmonad-finish}

We now will formally deduce the effectiveness of
$\widetilde{\pi}_N^* \widetilde{\pi}_{N,?}$ from the effectiveness
of $\pi_N^*\pi_{N,?}$. 
We will do so using Proposition \ref{p:tens-eff}. 

Note that when we quotient
$\sY_r/\cK_N$ by the formal group $(G(O)/\cK_N)_e^{\wedge}$
of $G(O)/\cK_N$, 
we tautologically obtain $\widetilde{\sY}_r/\widehat{\cK}_N$.
Moreover, the map $\pi_N$ is equivariant for the trivial action of:

\[
(G(O)/\cK_N)_e^{\wedge} = 
\Ker(G(O)/\cK_N \to
\widehat{G(O)}/\widehat{\cK}_N)
\]

\noindent on $\sZ_r/\widehat{\cK}_N$. 

Therefore, if we start with the adjunction:

\[
\pi_{N,?}: \QCoh(\sY_r/\cK_N) \rightleftarrows 
\QCoh(\sZ_r/\widehat{\cK}_N): \pi_N^*
\]

\noindent and pass to invariants (= coinvariants) for this
formal group,
we obtain the effectiveness of the monad coming from the adjunction:

\[
\xymatrix{
\QCoh(\widetilde{\sY}_r/\widehat{\cK}_N) 
\ar@<.4ex>[rr]^(.35){\widetilde{\pi}_{N,?} \otimes \ind} & & 
\QCoh(\sZ_r/\widehat{\cK}_N) \otimes \Rep\big(\Lie(G(O)/\cK_N)\big)
\ar@<.4ex>[ll]^(.65){\widetilde{\pi}_N^* \otimes \Oblv}
} 
\]

\noindent (for $\ind$ the functor of tensoring a vector space with the enveloping
algebra of this Lie algebra). 
But observe that our monad of interest, 
$\widetilde{\pi}_N^*\widetilde{\pi}_{N,?}$, is a direct
summand of the above monad. Indeed, the augmentation
of the enveloping algebra gives the desired construction.

This completes the proof that Lemma \ref{l:pi-n-bdded} implies
Proposition \ref{p:ourmonad-eff}.


\section{Infinitesimal analysis: fiberwise geometry}\label{s:z-geom}


\subsection{}

The purpose of this section is to prove 
Lemma \ref{l:pi-n-bdded} and Proposition \ref{p:zr-eta-1aff}
from \S \ref{s:z-tens}. 

These results are proved by analyzing the geometry of the
map $\pi$ (and its relatives) over the leading terms space,
using the results of \S \ref{s:locsys}.

We maintain the notation of the previous sections throughout.
We refer especially to \S \ref{ss:backtolocsys} for some
notation that might otherwise be hard to find.

\subsection{Fiberwise geometry}\label{ss:fibstart}

Fix $s$ as in Theorem \ref{t:jumps}.

Throughout, for a field-valued point 
$\eta \in (t^{-r}\fg[[t]]dt/t^s\fg[[t]]dt)_{dR}$, we use the subscript
$\eta$ to indicate that we are taking the fiber at $\eta$.

\begin{warning}\label{w:eta-fiber}

Note the difference between e.g. $\sY_{r,\eta}$ and 
the fiber of $t^{-r}\fg[[t]]dt$ over $\eta \in t^{-r}\fg[[t]]dt/t^s\fg[[t]]dt$:
the later is a scheme (which appeared in \S \ref{s:tame}), 
and the former is a formal thickening of it
(the embedding of the latter in the former is obtained by
base-change from the embedding of $\eta$ into its formal
completion in $t^{-r}\fg[[t]]dt/t^s\fg[[t]]dt$. 

\end{warning}

\subsection{}

Our fiberwise geometry is given by the following:\footnote{A useful, 
general heuristic
to keep in mind when looking at these results: the moral of
\S \ref{s:locsys} is that everything is as nice as for $G = \bG_m$
if we restrict to a field-valued point of the leading terms space.}  

\begin{prop}\label{p:z-eta-coarse}

The morphism:

\[
\widetilde{\sY}_{r,\eta} \to \sZ_{r,\eta}
\]

\noindent factors as:

\[
\widetilde{\sY}_{r,\eta} \into \widetilde{\sY}_{r,\eta}^{\prime}
\to \sZ_{r,\eta}
\]

\noindent where:

\begin{enumerate}

\item Each of the above morphisms is a
nilisomorphism (i.e., an isomorphism at the reduced level). 

\item 

$\widetilde{\sY}_{r,\eta}^{\prime}$ is an indscheme
of ind-finite type relative to $\sZ_{r,\eta}$.

\item $\widetilde{\sY}_{r,\eta}^{\prime} \to \sZ_{r,\eta}$
is formally smooth.

\item 

$\widetilde{\sY}_{r,\eta} \into \widetilde{\sY}_{r,\eta}^{\prime}$ is a regular
embedding.

\end{enumerate}

\end{prop}

\begin{proof}

Let $\nabla$ denote the universal connection 
on the trivial bundle with fiber $\fg((t))dt$ on $\sY_{r,\eta}$.

Because $\nabla$ is Fredholm, for $r^{\prime}$ large enough,
the map:

\[
t^{-r^{\prime}}\fg[[t]]dt \otimes \sO_{\sY_{r,\eta}} \to
H_{dR}^1(\nabla) 
\]

\noindent gives an epimorphism when restricted
to $\sY_r^{red} = t^{-r}\fg[[t]]dt \times_{t^{-r}\fg[[t]]dt/t^s\fg[[t]]dt} \eta \subset \sY_{r,\eta}$
(c.f. Warning \ref{w:eta-fiber}).

Let $\sY_r^{r^{\prime}}$ denote the formal completion 
of $\sY_r$ in $\sY_r^{r^{\prime}}$. Note that $G(O)$ acts on
$\sY_r^{r^{\prime}}$ by gauge transformations, and we can
form: 

\[
\widetilde{\sY}_r^{r^{\prime}} \coloneqq
\sY_r^{r^{\prime}} \overset{G(O)}{\times} \widehat{G(O)}.
\]

Note that $\widetilde{\sY}_r^{r^{\prime}}$ maps to $\sZ_r$.
We claim that the induced map $\widetilde{\sY}_{r,\eta}^{r^{\prime}} \to \sZ_{r,\eta}$
is formally smooth.

Indeed, it suffices to see that the cotangent complex restricted
to $\sY_{r,\eta}^{red}$ is a vector bundle in degree $0$.
It is easy to see that at a point $\Gamma dt \in \sY_{r,\eta}^{red}$,
the tangent complex the map is computed by the two-step complex:

\[
\fg((t))/\fg[[t]] \rar{-\nabla} \fg((t))dt/t^{-r^{\prime}}\fg[[t]]dt, \hspace{.5cm}
\nabla \coloneqq d - [\Gamma,-] dt.
\]

\noindent and this is an epimorphism by choice of $r^{\prime}$.

Clearly $\widetilde{\sY}_{r,\eta} \into \widetilde{\sY}_{r,\eta}^{r^{\prime}}$
is a regular embedding, so we obtain the claim.

\end{proof}

\subsection{Proof of Lemma \ref{l:pi-n-bdded}}

We now prove Lemma \ref{l:pi-n-bdded}. The proof occupies
\S \ref{ss:pi-n-pf-start}-\ref{ss:pi-n-pf-finish}.

Suppose that $N \geq r+s$ throughout.

\begin{rem}

Proposition \ref{p:z-eta-coarse} plays only a minor role in this
argument: we only use it for the conservativeness, not for
the cohomological boundedness.

\end{rem}

\subsection{}\label{ss:pi-n-pf-start}

First, we give a fiberwise version of the cohomological boundedness.

Recall that $\pi_N$ denotes the map 
$\sY_r/\cK_N \to \sZ_r/\widehat{\cK}_{N}$. 
By our lower bound on $N$, this is a morphism of prestacks over
$(t^{-r}\fg[[t]]dt/t^s\fg[[t]]dt)_{dR}$, so we can take fibers
at a field-valued point $\eta$. 

Note that $\QCoh(\sY_{r,\eta}^{red}/\cK_N)$ has a canonical $t$-structure,
and that $\sY_{r,\eta}^{red}/\cK_N$ is even an Artin stack smooth over
$\eta$.\footnote{
The reader wondering why we have ``reduced" in the notation should
refer to Warning \ref{w:eta-fiber}.}

Let $\pi_{N,\eta}^{red}$ denote the restriction of
$\pi_{N,\eta}: \sY_{r,\eta}/\cK_N \to \sZ_{r,\eta}/\widehat{\cK}_N$
to $\sY_{r,\eta}^{red}/\cK_N$. 

\begin{lem}\label{l:pi-n-amp/fibers}

The endofunctor $\pi_{N,\eta}^{red,*}\pi_{N,\eta,?}^{red}$ has 
cohomological amplitude $[0,\dim(G)]$. 

\end{lem}

\begin{proof}

We proceed in steps.

\step 

First, we claim that the tangent complex of the map
$\pi_{N,\eta}^{red}$ is a perfect complex on the algebraic
$\sY_{r,\eta}^{red}/\cK_N$. More precisely, we will
show that its pullback $\sY_{r,\eta}^{red}$, and hence
to any affine scheme, 
can be represented by a two-step complex
$\sE^0 \to \sE^1$ of vector bundles with $\sE^1$ having
rank $\dim(G) $.

Indeed, the tangent complex of this map at a point 
$\Gamma dt \in \sY_{r,\eta}^{red}$ is:\footnote{Of course,
the appropriate version of this 
formula holds with $\Gamma dt$ a test point with values in an affine 
scheme: we are just simplifying the notation here.}

\[
\begin{gathered}
\Ker\Big(
\Coker\big(
t^N\fg[[t]] \xar{-\nabla}
t^s\fg[[t]]dt\big) 
\to
\Coker\big(
\fg((t)) \xar{-\nabla}
\fg((t))dt\big)
\Big) = \\
\Ker\big(
\fg((t))/t^N\fg[[t]] \xar{-\nabla} 
\fg((t))dt/t^s\fg[[t]]dt
\big)
\end{gathered}
\]

\noindent for $\nabla \coloneqq d+\Gamma dt$. 

According to Proposition \ref{p:coda-2}, the universal
connection on $\sY_{r,\eta}^{red}$ has $H_{dR}^1$ which 
is a quotient of $\sE^1 \coloneqq 
\sO_{\sY_{r,\eta}^{red}}^{\oplus \dim(G)}$.
By freeness, we can obviously lift the map:

\[
\vcenter{\xymatrix{
& \sE^1 \ar[d] \ar@{..>}[dl] \\
\fg((t))dt \otimes \sO_{\sY_{r,\eta}} \ar@{->>}[r] & 
H_{dR}^1(\o{\cD}_{\sY_{r,\eta}^{red}},\nabla)
}} \hspace{10pt} \in \Pro(\QCoh(\sY_{r,\eta})^{\heart})
\]

\noindent and therefore we obtain a map 
from $\sE^1[-1]$ to the complex (not the total cohomology!)
of de Rham cohomology 
$H_{dR}^*(\o{\cD}_{\sY_{r,\eta}^{red}},\nabla)$, and so that
this map induces an epimorphism on applying $H^1$.

Define:

\[
\sE^0 \coloneqq 
\Coker\big(\sE^1[-1] \to H_{dR}^*(\o{\cD}_{\sY_{r,\eta}^{red}},\nabla)\big)
\in
\QCoh(\sY_{r,\eta}^{red}).
\]

We claim that $\sE^0$ is a finite 
vector bundle concentrated in cohomological
degree zero, which would clearly suffice since
the (homotopy) kernel of the boundary map $\sE^0 \to \sE^1$
is obviously (quasi)isomorphic to de Rham cohomology.

Indeed, $\sE^1$ is concentrated in cohomological degree
zero from the long exact sequence:

\[
\ldots \to 0 \to 
H_{dR}^0(\o{\cD}_{\sY_{r,\eta}^{red}},\nabla) \to 
H^0(\sE^0) \to 
\sE^1 \onto  H_{dR}^1(\o{\cD}_{\sY_{r,\eta}^{red}},\nabla) \to
H^1(\sE^0) \to 0 \to \ldotsplus
\]

\noindent Clearly $\sE^1$ is perfect when regarded as a complex
(by Fredholmness of the de Rham differential here).
Finally, we see that it has $\Tor$-amplitude $0$, since 
for $\sF \in \QCoh(\sY_{r,\eta}^{red})^{\geq 0}$, we have:

\[
\sE^0 \otimes \sF = 
\Coker(\sE^1[-1] \otimes \sF \to 
H_{dR}^*(\o{\cD}_{\sY_{r,\eta}^{red}},\nabla) \otimes \sF)
\]

\noindent and this is the cone of a map from a complex in 
degrees $\geq 1$ to a complex in degrees $\geq 0$, so itself
lies in degrees $\geq 0$ as desired.

\step 

We will compute $\pi_{N,\eta}^{red,*}\pi_{N,\eta,?}^{red}$
using the theory of Lie algebroids from \cite{grbook}.\footnote{
At the current moment, \cite{grbook} is not
available yet. However, the relevant 
notes on Lie algebroids \emph{are} available
on Gaitsgory's website.}

Note that $\sY_{r,\eta}^{red}/\cK_N$ and
$\sZ_{r,\eta}/\widehat{\cK}_N$ are of ind-finite type
(by Theorems \ref{t:artin} and \ref{t:jumps}), so
$\IndCoh$ makes sense on them. Moreover, 
by formal smoothness, $\IndCoh = \QCoh$ on these
prestacks (the equivalence is given by
tensoring with the dualizing sheaf).
Under this equivalence, the above functor corresponds
to $\pi_{N,\eta}^{red,!}\pi_{N,\eta,*}^{red,\IndCoh}$,
in the usual $\IndCoh$-notation.

Then the upshot of the \cite{grbook} theory is that
the endofunctor $\pi_{N,\eta}^{red,*}\pi_{N,\eta,?}^{red}$
admits a $\bZ^{\geq}$-filtration with
$n$th associated graded term the endofunctor
sending $\sF \in \QCoh(\sY_{r,\eta}^{red}/\cK_N)$ 
to $\sF$ tensored with the $n$th symmetric power of the
tangent complex of $\pi_{N,\eta}^{red}$.

By the above calculations, the $n$th symmetric power of
this tangent complex is a perfect complex
of $\Tor$-amplitude $[0,\dim(G)]$ for all $n$, giving
the lemma.

\end{proof}

\subsection{}\label{ss:pi-n-pf-finish}

We now prove the lemma.

\begin{proof}[Proof of Lemma \ref{l:pi-n-bdded}] \setcounter{steps}{0}

We proceed by steps.

\step First, we show that $\pi_N^*\pi_{N,?}$ is left $t$-exact.
Note that argument is very coarse (and standard), and does not
really use the specifics of the geometry.

Let $\iota$ denote the embedding $\sY_r \into \sZ_r$.
Define a $t$-structure on $\QCoh(\sZ_r)$,
where we let $\QCoh(\sZ_r)^{\leq 0}$ be generated
under colimits by $\iota_?(\QCoh(\sY_r)^{\leq 0})$.

It is well-known that $\iota_?$ is $t$-exact in this case.
Indeed, it is obviously right $t$-exact.
Moreover, the composite $\iota^*\iota_?$ is left 
$t$-exact.\footnote{Since $?$-pushforward functors
can be confusing, a toy model: suppose $i: X \into Y$
is a closed embedding of smooth schemes. 
Then it is easy to see $i_?(\sF) = 
i_*^{\IndCoh}(\sF \otimes \omega_X) \otimes \omega_Y^{-1}$.
So if $\sF \in \QCoh(X)^{\geq 0}$, $i_?(\sF)$ lies
in degrees $\geq \dim(Y) - \dim(X)$. But then $i^*i_?(\sF)$
lies in degrees $\geq 0$ as desired.} 
But we can tautologically test whether an object
of $\QCoh(\sZ_r)$ is coconnective by checking after applying 
$\iota^*$, so we obtain the claim. 

Note that $\iota^*$ is conservative.
Since $\iota_?$ is $t$-exact and the
$t$-structure on $\QCoh(\sY_r)$ is \emph{compactly generated}
(i.e., $\QCoh(\sY_r)$ is compactly generated  with compacts
closed under truncations),
the $t$-structure on $\QCoh(\sZ_r)$ is compactly generated
as well. 

We have a similar $t$-structure on 
$\QCoh$ of: 

\[
\widetilde{\sY}_{r,\cK_N} \coloneqq 
\widehat{\cK}_N \overset{\cK_N}{\times} \sY_r
\]

\noindent (i.e., the $\cK_N$-counterpart of $\widetilde{\sY}_r$)
with all the same formal properties.

In particular, it follows that the $?$-pushforward
$\widetilde{\sY}_{r,\cK_N} \to \sZ_r$ is $t$-exact.
Indeed, the $?$-pushforward is tautologically right $t$-exact.
We need to see that this $?$-pushforward maps
$\QCoh(\widetilde{\sY}_{r,\cK_N})^{\geq 0}$ to 
$\QCoh(\sZ_r)^{\geq 0}$. But the latter category is 
compactly generated 
by objects $?$-pushed forward from compact objects of
$\QCoh(\sY_r)^{\geq 0}$, for which the claim is clear.\footnote{Recall
that compact generation works the same for non-DG categories
as for DG categories. I.e., we first obtain that
$\QCoh(\widetilde{\sY}_{r,\cK_N})^{\geq 0}$ is freely generated under
filtered colimits by its subcategory of compact objects, and
this subcategory in turn is generated under retracts by objects
$?$-pushed forward from compacts in $\QCoh(\sY_r)^{\geq 0}$.}

We can now deduce that $\pi_N^*\pi_{N,?}$ is left $t$-exact.
Indeed, we form the commutative diagram:

\[
\xymatrix{
\sY_r \ar[r]^i & \widetilde{\sY}_{r,\cK_N} \ar[d]^{\alpha} 
\ar[r]^{\vph} & \sZ_r \ar[d]^{\beta} \\
& \sY_r/\cK_N \ar[r]^{\pi_N} & \sZ_r/\widehat{\cK}_N
}
\]

\noindent with the square being Cartesian. 
Recall that the pullback along $\sY_r \to \sY_r/\cK_N$ is $t$-exact.
Then base-change in the above diagram and our
earlier observations immediately gives the result.

Indeed, we compute:

\begin{equation}\label{eq:pi-n-basechange}
i^*\alpha^*\pi_N^*\pi_{N,?} = 
i^* \vph^* \beta^* \pi_{N,?} =
i^* \vph^* \vph_? \alpha^*.
\end{equation}

\noindent Then observe that $\alpha^*$ is left $t$-exact:
indeed, to see that an object of $\QCoh(\widetilde{\sY}_{r,\cK_N})$
is coconnective it suffices to apply $i^*$, so we need to 
see that $i^*\alpha^*$ is left $t$-exact, but this functor
is exact. Furthermore, we have seen that
$\vph_?$ is $t$-exact. Finally, $i^*\vph^* = \iota^*$
is left $t$-exact.

\step 

We now show that $\pi_N^*\pi_{N,?}$ is cohomologically
bounded (for $N \geq r+s$).

Recall that we have the map recording the leading terms of a connection:

\[
\sY_r/\cK_N \to \sZ_r/\widehat{\cK}_N \to (t^{-r}\fg[[t]]dt/t^s\fg[[t]]dt)_{dR}.
\]

\noindent By continuity of these functors and base-change, 
the Cousin resolution means that we need to check that
$\pi_N^*\pi_{N,?}$ has uniformly (in $\eta$) bounded amplitude on objects
supported at a point $\eta \in (t^{-r}\fg[[t]]dt/t^s\fg[[t]]dt)_{dR}$.
Since such objects are canonically filtered (by increasing
infinitesimal neighborhoods) by 
objects $?$-pushed forward from $\QCoh(\sY_{r,\eta}^{red}/\cK_N)$,
we reduce to treating such objects. 

But this case follows from Lemma \ref{l:pi-n-amp/fibers}.

\step

Next, we show that $\pi_N^*\pi_{N,?}$ is conservative.

Using the calculation \eqref{eq:pi-n-basechange}
(and its notation),
we need to see
that $i^* \vph^* \vph_? \alpha^*$ is conservative. 
We will show that each constituent functor is
conservative.
For the pullbacks this follows because the maps are eventually coconnective
nilisomorphisms between eventually coconnective spaces.

For $\vph_?$, note that the Cousin
resolution allows us to test conservativity on geometric points
of $(t^{-r}\fg[[t]]dt/t^s\fg[[t]]dt)_{dR}$ again. Then the
result again readily follows from Proposition \ref{p:z-eta-coarse}
and standard facts about formal completions,
but we outline the argument for the sake of completeness:

In the notation of \emph{loc. cit}, it obviously suffices
to see that $?$-pushforward along the formally smooth
map $\widetilde{\sY}_{r,\cK_N,\eta}^{\prime} \to \sZ_{r,\eta}$
is conservative. 

We can check this after 
(conservatively) restricting to 
$\sY_{r,\eta} \subset \sZ_{r,\eta}$, in which case,
the notation of \emph{loc. cit}., the map becomes
a formally smooth map over $\sY_{r,\eta}$ with finite rank
tangent complex. As in the proof of Corollary \ref{c:fiber-actmap-1aff}
below,
this fiber product is classical and is isomorphic
to the formal completion of a finite rank vector bundle
on $\sY_{r,\eta}$.

So we have $r:\sN \to \sY_{r,\eta}$ a finite rank vector
bundle (the normal bundle appearing above), and
we are $?$-pushing forward from the
formal completion of its zero section. Up to a twist by
a (homologically graded) line bundle, this is the same as
$?$-pushing forward to $\sN$ and then $*$-pushing
forward to $\sY_{r,\eta}$. The former operation is fully-faithful
and the latter operation is obviously conservative, so we obtain the 
claim.

\end{proof}

\subsection{Proof of Proposition \ref{p:zr-eta-1aff}}

We have the following first application of Proposition \ref{p:z-eta-coarse}
to 1-affineness.

\begin{cor}\label{c:fiber-actmap-1aff}

\begin{enumerate}

\item 

The morphism 
$\widetilde{\sY}_{r,\eta} \to \sZ_{r,\eta}$ is 1-affine.\footnote{Though
the indschemes in question are not, since they are infinite-dimensional
in the ind-directions.}

\item\label{i:fiberprod-eta-tens}

For every prestack $S$ with $S \to \sZ_{r,\eta}$, the canonical
morphism:

\[
\QCoh(S) \underset{\QCoh(\sZ_{r,\eta})}{\otimes} 
\QCoh(\widetilde{\sY}_{r,\eta}) \to
\QCoh(S \underset{\sZ_{r,\eta}}{\times} \widetilde{\sY}_{r,\eta}) 
\]

\noindent is an equivalence.

\end{enumerate}

\end{cor}

\begin{proof}

We proceed in steps.

\step 

First, we will show that for a formally smooth nilisomorphism
$p:\cY \to \cZ$ of formal schemes\footnote{I.e.,
indschemes whose reduced part is a scheme.}
(think: $\widetilde{\sY}_{r,\eta}^{\prime} \to \sZ$) is a 
disc bundle in an appropriate sense. The argument is 
well-known, but we include it for completeness.

First, we show that there is a map $i: \cZ \to \cY$ of which
$p$ is a splitting. Indeed, formal smoothness gives this map:

\[
\xymatrix{
\cY^{red} = \cZ^{red} \ar[r] \ar[d] & \cY \ar[d]^p \\
\cZ \ar@{=}[r] \ar@{..>}[ur]^i & \cZ.
}
\]

Now let $\sN \to \cZ$ denote the total space of the normal
bundle to $i: \cZ \to \cY$. We claim that $\cY$ is isomorphic over $\cZ$
to the formal completion of the zero section of $\sN$.

Indeed, let $\cZ \into \cZ^+ \into \cY$ 
denote the first infinitesimal neighborhood
of $\cZ$ in $\cY$ (embedded via $i$).
There is a tautological map $\cZ^+ \to \sN$, 
identifying $\cZ^+$ with the first infinitesimal neighborhood of
the zero section. Then by formal smoothness, we obtain 
an extension of this map:

\[
\xymatrix{
\cZ^+ \ar[r] \ar[d] & \sN \ar[d] \\
\cY \ar[r]^p \ar@{..>}[ur] & \cZ
}
\]

\noindent Obviously this map factors through the formal
completion of the zero section. Finally, filtering functions
on both sides by order of vanishing along $\cZ$, 
we see that the induced map $\cY \to \sN_{\cZ}^{\wedge}$
is an isomorphism.

\step

We now readily see the first claim: since 
$\sY_{r,\eta}^{\prime}$ is now the formal completion of a 
(finite rank) vector bundle over $\cZ_{r,\eta}$, the morphism
$\sY_{r,\eta}^{\prime} \to \cZ_{r,\eta}$ is 1-affine by
\cite{shvcat} \S 4. Obviously 
$\sY_{r,\eta} \into \sY_{r,\eta}^{\prime}$ is 1-affine,
so we obtain the result, since 1-affine morphisms are closed
under compositions.

\step 

The second part follows similarly:

We factor $\widetilde{\sY}_{r,\eta} \to \sZ_{r,\eta}$ as:

\[
\widetilde{\sY}_{r,\eta} \into \widetilde{\sY}_{r,\eta}^{\prime} =
\sN_{\sZ_{r,\eta}}^{\wedge} \into \sN_{\sZ_{r,\eta}} \to \sZ_{r,\eta}
\]

\noindent where $ \sN_{\sZ_{r,\eta}} \to \sZ_{r,\eta}$ is
a finite rank vector bundle (as above).

The desired compatibility between formation
of tensor products and fiber products holds for
any affine morphism (see e.g. \cite{shvcat} Appendix B).
The only non-affine morphism above is 
$\sN_{\sZ_{r,\eta}}^{\wedge} \into \sN_{\sZ_{r,\eta}}$, so it suffices to
treat this map. But this follows since the pullback admits
a fully-faithful $\QCoh(\sZ_{r,\eta})$-left adjoint 
as in \cite{indschemes} \S 7 (c.f. the argument of \cite{shvcat} \S 4).

\end{proof}

\begin{variant}

For $N \geq 0$, let $\widetilde{\sY}_{r,\cK_N}$ denote 
$\widehat{\cK}_N \overset{\cK_N}{\times} \sY_r$,
i.e., the analogue of $\widetilde{\sY}_r$ with the congruence
subgroup  $\cK_N$ replacing $G(O)$.
The above analysis works just as well
for $\widetilde{\sY}_{r,\cK_N}$ in place of 
$\widetilde{\sY}_r$.


\end{variant}

\subsection{} 

We complete our obligations from \S \ref{s:z-tens} with the following.

\begin{proof}[Proof of Proposition \ref{p:zr-eta-1aff}]

Note that $\sY_{r,\eta}/\cK_{r+s} \to \sZ_{r,\eta}/\widehat{\cK}_{r+s}$
is an ind-(eventually coconnective and proper covering),
and that $\sZ_{r,\eta}/\widehat{\cK}_{r+s}$ is a classical prestack.
Moreover, we have seen that $\sY_{r,\eta}/\cK_{r+s}$ is 1-affine.
Therefore, we are in the above setting.

We claim that the morphism 
$\sY_{r,\eta}/\cK_{r+s} \to \sZ_{r,\eta}/\widehat{\cK}_{r+s}$
is 1-affine. It suffices to see that the morphism is 1-affine
after we pullback to $\sZ_{r,\eta}$. Then we obtain
the morphism $\widetilde{\sY}_{r,\cK_{r+s},\eta} \to \sZ_{r,\eta}$,
which we have seen in 1-affine.

Then the only thing remaining to check is the 
criterion \eqref{i:proper-ec-4} from Proposition \ref{p:proper-ec-loc-2},
i.e., that the canonical
map:

\[
\QCoh(\sY_{r,\eta}/\cK_{r+s}) 
\underset{\QCoh(\sZ_{r,\eta}/\widehat{\cK}_{r+s})}{\otimes} 
\QCoh(\sY_{r,\eta}/\cK_{r+s}) \to
\QCoh(\sY_{r,\eta}/\cK_{r+s} 
\underset{\sZ_{r,\eta}/\widehat{\cK}_{r+s}}{\times}  
\sY_{r,\eta}/\cK_{r+s} )
\]

\noindent is an equivalence.

First, note that Corollary \ref{c:fiber-actmap-1aff} \eqref{i:fiberprod-eta-tens}
gives:

\[
\QCoh(\widetilde{\sY}_{r,\cK_{r+s},\eta}) 
\underset{\QCoh(\sZ_{r,\eta})}{\otimes} 
\QCoh(\widetilde{\sY}_{r,\cK_{r+s},\eta}) \isom 
\QCoh(\widetilde{\sY}_{r,\cK_{r+s},\eta} 
\underset{\sZ_{r,\eta}}{\times}  \widetilde{\sY}_{r,\cK_{r+s},\eta} ).
\]

\noindent If we pass to $\widehat{\cK}_{r+s}$-invariants
on the right hand side, we obtain
${\QCoh(\sY_{r,\eta}/\cK_{r+s} 
\times_{\sZ_{r,\eta}/\widehat{\cK}_{r+s}}
\sY_{r,\eta}/\cK_{r+s})}$. Similarly for each of the individual
terms of the tensor product. Therefore, we need to show that
formation of the tensor product commutes with
taking $\widehat{\cK}_{r+s}$-invariants. 

We will show this using ideas from \S \ref{s:finale}.\footnote{This
argument is overkill: ultimately, it relies on the tameness of the 
action of $G(O)$ on gauge forms. However, we are currently
analyzing a given fiber, for which such a claim is easier. The reader
is invited to try to find a simpler argument, but this is the most
efficient one that I could find.}
It suffices to show that $\widehat{\cK}_{r+s}$-invariants functorially
coincide with
$\widehat{\cK}_{r+s}$-coinvariants for everything
in sight, i.e., for the tensor product itself, and for each
term in the tensor product. This would 
imply the desired claim, since coinvariants are tautologically
compatible with tensor products.

Note that each of the four categories in question
is acted on compatibly by $\QCoh(\sZ_r)$ and
$\QCoh(\widehat{\cK}_{r+s})$.
Therefore, the claim follows from\footnote{The reader who glances at the
proof of the cited proposition will see that it 
is a straightforward consequence of
Theorem \ref{t:tame}, i.e., it is not using anything deep about
the action of $\widehat{G(O)}$ on $\sZ_r$, in spite of being
cited out of order.}
Proposition \ref{p:zr-tame}, which exactly says that invariants
and coinvariants naturally coincide for any
category acted on compatibly by $\QCoh(\sZ_r)$ and
$\QCoh(\widehat{\cK}_{r+s})$ (or any other congruence
subgroup in place of $\cK_{r+s}$).

\end{proof}


\section{Conclusion of the proof of the main theorem}\label{s:finale}


\subsection{}

In this section, we prove that for $G$ \emph{reductive},
$\LocSys_G(\o{\cD})$ is 1-affine.

The argument is fairly straightforward, 
given the work we have done already at this point.

\subsection{1-affineness and calculation of tensor products}\label{ss:finale-tens-prod}

Note that the morphism
$\fg((t))dt/\widehat{G(O)} \to \LocSys_G(\o{\cD})$
is a $\Gr_{G,dR}$-fibration.\footnote{Up to sheafification, anyway: 
it is rather a $(G(K)/G(O))_{dR}$-fibration, but we ignore this
distinction since it is irrelevant.}
Therefore, the morphism is 1-affine by \cite{shvcat} Theorem 2.6.3.
Moreover, the pullback for this morphism admits a left adjoint
satisfying the projection formula, and the same holds after any
base-change (since $\Gr_G$ is ind-proper for $G$ reductive).

Therefore, although the hypotheses of 
Proposition \ref{p:proper-ec-loc-2} do not quite hold (since
that result assumes the map to be ind-schematic), its
logic holds as is and we can still apply the conclusion.
I.e., we obtain that $\Loc$ is fully-faithful for $\LocSys_G(\o{\cD})$,
and 1-affineness for $\LocSys_G(\o{\cD})$ reduces to checking
that the morphism:

\begin{equation}\label{eq:finale-1}
\begin{gathered}
\QCoh(\fg((t))dt/\widehat{G(O)})
\underset{\QCoh(\LocSys_G(\o{\cD}))}{\otimes}
\QCoh(\fg((t))dt/\widehat{G(O)}) \to \\
\QCoh\Big(\fg((t))dt/\widehat{G(O)} \underset{\LocSys_G(\o{\cD})}{\times} 
\fg((t))dt/\widehat{G(O)}\Big)
\end{gathered}
\end{equation}

\noindent is an isomorphism.

We will check this in what follows. First, we will give
some generalities in \S \ref{ss:tame-redux}-\ref{ss:tame-redux-finish}
about tameness for $G(K)$, which will justify commuting some
invariants with some tensor products. After this, the argument
will be quick.

\subsection{Tameness redux}\label{ss:tame-redux}

In what follows, let $\cG_1$ be an affine group scheme, and let 
$\cG_1 \subset \cG_2$ with $\cG_2$ a 
group indscheme with $\cG_2/\cG_1$ a formally smooth ind-proper
$\aleph_0$-indscheme.

\begin{example}

The basic examples are $G(O) \subset \widehat{G(O)}$ for
any affine algebraic group,
and for $G$ reductive, $G(O) \subset G(K)$.

\end{example}

\subsection{}

First, we construct a self-duality for $\QCoh(\cG_2)$.
This construction will be of ``semi-infinite" nature,
so e.g. depends on the choice of embedding $\cG_1 \subset \cG_2$.

Let $q$ denote the quotient map $\cG_2 \to \cG_2/\cG_1$.
This map is affine, so 
$\QCoh(\cG_2) = q_*(\sO_{\cG_2})\mod(\QCoh(\cG_2/\cG_1))$.

Note that 
$\QCoh(\cG_2/\cG_1) \xar{-\otimes \omega}\IndCoh(\cG_2/\cG_1)$
is an isomorphism by \cite{indschemes},
and note that $\IndCoh(\cG_2/\cG_1)$ is self-dual via Serre duality.

It is easy to see that the monad $q_*q^*$ is self-dual 
(under Serre duality). Indeed, the self-duality pairing:

\[
\langle-,-\rangle:
\QCoh(\cG_2/\cG_1) \otimes \QCoh(\cG_2/\cG_1) \to \Vect
\]

\noindent is given by the formula:

\[
\langle \sF,\sG\rangle = 
\Gamma^{\IndCoh}(\cG_2/\cG_1,(\sF \otimes \omega_{\cG_2/\cG_1}) 
\overset{!}{\otimes} (\sG \otimes \omega_{\cG_2/\cG_1})) =
\Gamma^{\IndCoh}(\cG_2/\cG_1,(\sF \otimes \sG) \otimes \omega_{\cG_2/\cG_1})
\]

\noindent and we have:

\[
\langle q_*q^*(\sF) , \sG \rangle =
\langle \sF \otimes q_*(\sO_{\cG_2}), \sG \rangle =
\langle \sF, \sG \otimes q_*(\sO_{\cG_2}) \rangle =
\langle \sF, q_*q^*(\sG) \rangle
\]

\noindent where the second equality follows from the explicit
form of the Serre duality pairing, and the others are just the projection
formula for the affine morphism $q$.

This self-duality of the monad readily gives the desired
self-duality for $\QCoh(\cG_2)$.

\begin{example}\label{e:strange-globalsections}

The functor dual to the $*$-restriction
$\QCoh(\cG_2) \to \QCoh(\cG_1)$ is the \emph{left} adjoint
to this restriction. The functor dual to the pullback
$\Vect \to \QCoh(\cG_2)$ is given by $*$-pushing forward
to $\cG_2/\cG_1$ and then taking $\IndCoh$ global
sections (which is a left adjoint here, by assumption).

\end{example}

\subsection{}

Note that the group structure on $\cG_2$
canonically makes $\QCoh(\cG_2)$ into a coalgebra object
of $\DGCat_{cont}$. If e.g. $\cG_2$ acts a prestack $\cY$,
the $\QCoh(\cY)$ is a comodule category for 
$\QCoh(\cG_2)$.\footnote{We need the formula 
$\QCoh(\cG_2) \otimes \QCoh(\cY) \isom \QCoh(\cG_2 \times \cY)$
for this, but this formula holds because $\QCoh(\cG_2)$ is
dualizable.} The restriction functor $\QCoh(\cG_2) \to \QCoh(\cG_1)$
is a morphism of coalgebras.

By self-duality, $\QCoh(\cG_2)$ inherits
a monoidal structure as well, so that $\QCoh(\cG_2)$-module
categories (in $\DGCat_{cont}$) are the same as comodule categories for
the above structure. Moreover, $\QCoh(\cG_2)$
receives a monoidal functor from $\QCoh(\cG_1)$.

Note that $\Vect$ has a tautological $\QCoh(\cG_2)$-module
structure corresponding to the trivial action of $\cG_2$ on a point.
As usual, this allows us to speak about invariants
and coinvariants for $\QCoh(\cG_2)$-module categories.

\subsection{Semi-infinite norm functor}

Suppose now that $\sC$ is acted on by $\QCoh(\cG_2)$.
We will construct a norm functor:

\[
\Nm^{\sinf}: \sC_{\cG_2} \to \sC^{\cG_2}
\]

\noindent that is functorial in $\sC$ and an equivalence for
$\QCoh(\cG_2)$.

Here is an abstract description, though we will give a more
concrete (perhaps) description in what follows.

Regard $\QCoh(\cG_2)$ as a bi-comodule category
over itself. 
Then we have 
$\Vect \isom \QCoh(\cG_2)^{\cG_2,w}$ as $\QCoh(\cG_2)$-comodule
categories, where we are using the residual action on the right
hand side. Indeed, since we have carefully used the comodule
language everywhere, there is nothing non-standard in this
claim, i.e., we are not using the self-duality of $\QCoh(\cG_2)$ anywhere
(which is non-standard, since e.g. it has $\cG_1$ built into it).

But by equating comodule structures and module structures
by self-duality, the situation is more interesting.
In particular, for $\sC\in \QCoh(\cG_2)$,
we can tensor the above map to obtain:

\[
\sC_{\cG_2,w} \coloneqq 
\sC \underset{\QCoh(\cG_2)}{\otimes} \Vect =
\sC \underset{\QCoh(\cG_2)}{\otimes}  \QCoh(\cG_2)^{\cG_2,w} \to
\big(\sC \underset{\QCoh(\cG_2)}{\otimes}  \QCoh(\cG_2)\big)^{\cG_2,w}
\]

\noindent This is our norm map. It obviously satisfies
the desired functoriality, and is obviously an equivalence
for $\sC = \QCoh(\cG_2)$.

\subsection{}

We now give a slightly more concrete description of the norm
functor above.

Suppose that $\sC$ is a $\QCoh(\cG_2)$-module category.
The restriction functor $\Oblv:\sC^{\cG_2} \to \sC^{\cG_1}$ 
is conservative and admits a left adjoint $\Av_!^w$
by ind-properness of $\cG_2/\cG_1$. This functor is functorial
in $\sC$.

We claim that the composite functor:

\[
\sC \to \sC_{\cG_2,w} \xar{\Nm^{\sinf}} \sC^{\cG_2,w}
\]

\noindent is computed by:

\[
\sC \xar{\Av_*^w} \sC^{\cG_1,w} \xar{\Av_!^w} \sC^{\cG_2,w}.
\]

Indeed, this follows from the commutative diagram:

\[
\xymatrix{
\sC =
\sC \underset{\QCoh(\cG_2)}{\otimes} \QCoh(\cG_2) 
\ar[r]^{\id \otimes \Av_*^w} & 
\sC \underset{\QCoh(\cG_2)}{\otimes} \QCoh(\cG_2)^{\cG_1,w} 
\ar[r] \ar[d]^{\Av_!^w} &
\big(\sC \underset{\QCoh(\cG_2)}{\otimes} \QCoh(\cG_2)\big)^{\cG_1,w} =
\sC^{\cG_1,w} 
\ar[d]^{\Av_!^w} \\
& 
\sC_{\cG_2,w} = 
\sC \underset{\QCoh(\cG_2)}{\otimes} \QCoh(\cG_2)^{\cG_2,w}
\ar[r] &
\big(\sC \underset{\QCoh(\cG_2)}{\otimes} \QCoh(\cG_2)\big)^{\cG_2,w} =
\sC^{\cG_2,w}
}
\]

\noindent and the calculation that 
$\QCoh(\cG_2) \xar{\Av_!^w\Av_*^w} \QCoh(\cG_2)^{\cG_2,w} = \Vect$
is the functor dual to the pullback (as follows from Example 
\ref{e:strange-globalsections}).

\subsection{}

We can now proceed as before with tameness.

\begin{defin}

$\sC$ is \emph{tame} with respect to $\cG_2$ if the above
norm map is an equivalence.

\end{defin}

In fact, tameness for $\cG_2$ is the same as tameness
for $\cG_1$:

\begin{prop}\label{p:tame-indproper}

$\sC$ is tame with respect to $\cG_2$ if it is tame as a 
$\QCoh(\cG_1)$-module category.

\end{prop}

\begin{example}

If $\cG_1$ is \emph{normal} in $\cG_2$, this result is
immediate: then it is well-known (c.f. \cite{shvcat} \S 11) that
invariants and coinvariants coincide for $\cG_2/\cG_1$,
so the problem reduces to a construction chase
(i.e., checking that the induced equivalence of invariants
and coinvariants for $\cG_2$ is given by our construction above).

\end{example}

\begin{proof}[Proof of Proposition \ref{p:tame-indproper}]

Suppose $\sC$ is a general $\QCoh(\cG_2)$-module category,
i.e., forget about tameness for a moment.

We have a tautologically commutative diagram:

\[
\xymatrix{
\sC \ar[r] & \sC_{\cG_1,w} \ar[d]_{\Nm_{\cG_1}} \ar[r] & 
\sC_{\cG_2,w} \ar[d]^{\Nm^{\sinf}_{\cG_2}} \\
& \sC^{\cG_1,w} \ar[r]^{\Av_!^w} & \sC^{\cG_2,w}.
}
\]

Note that the top arrow in this square is given
by:

\[
 \sC_{\cG_1,w} = 
\sC \underset{\QCoh(\sG_2)}{\otimes} \QCoh(\cG_2)_{\cG_1,w} =
\sC \underset{\QCoh(\sG_2)}{\otimes} \QCoh(\cG_2/\cG_1) \to
\sC \underset{\QCoh(\sG_2)}{\otimes} \Vect = \sC_{\cG_2,w}.
\]

\noindent corresponding to the $\IndCoh$-global sections
functor on $\cG_2/\cG_1$. In particular, we see that
the functor $\sC_{\cG_1,w} \to \sC_{\cG_2,w}$ admits
a continuous right adjoint, since $\cG_2/\cG_1$ is
ind-proper. Moreover, this right adjoint
is obviously monadic, since it is conservative and continuous.

This description of the right adjoint
immediately gives the commutation of the diagram:

\[
\xymatrix{
\sC_{\cG_1,w} \ar[d]_{\Nm_{\cG_1}} & 
\sC_{\cG_2,w} \ar[d]^{\Nm^{\sinf}_{\cG_2}} \ar[l] \\
\sC^{\cG_1,w}  & \sC^{\cG_2,w} \ar[l]_{\Oblv}
}
\]

Now suppose that $\sC$ is tame with respect to the
$\cG_1$-action. The we have a morphism of monads
corresponding to the commutative diagram:

\[
\xymatrix{
\sC_{\cG_2,w} \ar[rr]^{\Nm^{\sinf}} \ar[dr] & & \sC^{\cG_2,w} \ar[dl] \\
& \sC_{\cG_1,w} = \sC^{\cG_1,w} &
}
\]

\noindent that is an isomorphism of monads
(by the explicit descriptions of these functors). 
Moreover, each of these diagonal functors is monadic
(being continuous and conservative), so we obtain the claim.

\end{proof}

\subsection{}\label{ss:tame-redux-finish}

We now have the following application to the actions on gauge forms.

\begin{prop}\label{p:zr-tame}

\begin{enumerate}

\item 

For any affine algebraic group $G$:

\begin{enumerate}

\item $\QCoh(\sZ_r)$ is tame with respect to the 
$\widehat{G(O)}$-action.

\item More generally, any $\sC \in \QCoh(\sZ_r)\mod(\DGCat_{cont})$
equipped with a compatible action of $\QCoh(\widehat{G(O)})$
(i.e., with an action of the appropriate semidirect product,
c.f. the proof of Proposition \ref{p:tame-1aff}) is
tame with respect to $\widehat{G(O)}$.

\end{enumerate}

\item 

For $G$ reductive, the same conclusions hold with
$G(K)$ in place of $\widehat{G(O)}$ and 
$\fg((t))dt$ in place of $\sZ_r$.

\end{enumerate}

\end{prop}

\begin{proof}

We begin with the second result.

By Proposition \ref{p:tame-indproper}, it suffices
to prove these results with $G(O)$ replacing $\widehat{G(O)}$
everywhere.

Then note that in forming the limit 
$\QCoh(\fg((t))dt) = \lim \QCoh(t^{-r}\fg[[t]]dt)$, 
each of the structural functors admits a left adjoint.
Therefore, this limit is also a colimit, and formation of this
limit commutes with all tensor products.
Finally, note that each of the structural functors (whether left
or right adjoint) is a morphism of $\QCoh(G(O))$-module categories.

Now for $\sC \in \QCoh(\fg((t))dt)\mod$, we obtain:

\[
\sC = \underset{r}{\lim} \, 
\sC \underset{\QCoh(\fg((t))dt)}{\otimes} 
\QCoh(t^{-r}\fg[[t]]dt) =
\underset{r}{\colim} \, 
\sC \underset{\QCoh(\fg((t))dt)}{\otimes} 
\QCoh(t^{-r}\fg[[t]]dt) \in \QCoh(G(O))\mod.
\]

\noindent Note that formation of invariants commutes with
formation of the limit, and formation of coinvariants
commutes with formation of the colimit. Therefore, by
functoriality and by Theorem \ref{t:tame}, we obtain:

\[
\begin{gathered}
\sC_{G(O),w} = 
\underset{r}{\colim} \, 
\Big(\sC \underset{\QCoh(\sZ_r)}{\otimes} 
\QCoh(t^{-r}\fg[[t]]dt)\Big)_{G(O),w} \isom 
\underset{r}{\colim} \, 
\Big(\sC \underset{\QCoh(\sZ_r)}{\otimes} 
\QCoh(t^{-r}\fg[[t]]dt)\Big)^{G(O),w} = \\
\underset{ r}{\lim} \, 
\Big(\sC \underset{\QCoh(\sZ_r)}{\otimes} 
\QCoh(t^{-r}\fg[[t]]dt)\Big)^{G(O),w} =
\sC^{G(O),w}
\end{gathered}
\]

\noindent as desired.

The corresponding result for $\sZ_r$ is proved
in the same way, but with the appropriate formal completions
taken throughout: we leave the details to the reader.

\end{proof}

\subsection{Calculation of tensor products}

First, we claim the following:

\begin{lem}\label{l:finale-gp}

The morphism:

\begin{equation}\label{eq:finale-2}
\QCoh(\fg((t))dt) 
\underset{\QCoh(\LocSys_G(\o{\cD}))}{\otimes}
\QCoh(\fg((t))dt) \to
\QCoh\big(\fg((t))dt \underset{\LocSys_G(\o{\cD})}{\times}
\fg((t))dt\big)
\end{equation}

\noindent is an equivalence.

\end{lem}

\begin{proof}

Consider the projection $p_2: G(K) \times \fg((t))dt \to \fg((t))dt$
as $G(K)$-equivariant for gauge action on the target,
and for the diagonal action on the source induced by the right
action of $G(K)$ on itself and the gauge action on $\fg((t))dt$.
Note that when
we quotient by $G(K)$, we obtain the projection 
$\fg((t))dt \to \LocSys_G(\o{\cD})$.

Clearly the morphism:

\[
\begin{gathered}
\QCoh(G(K) \times \fg((t))dt) 
\underset{\QCoh(\fg((t))dt)}{\otimes}
\QCoh(G(K) \times \fg((t))dt) \to \\
\QCoh\Big((G(K) \times \fg((t)) dt )
\underset{\fg((t))dt}{\times}
(G(K) \times \fg((t))dt)\Big)
\end{gathered}
\]

\noindent is an isomorphism, since
$\QCoh(G(K))\otimes \QCoh(\fg((t))dt =
\QCoh(G(K) \times \fg((t))dt)$.
Moreover, this isomorphism is equivariant
for the diagonal $G(K)$-action (given as above).
We claim that passing to weak $G(K)$-invariants
gives the desired statement of the lemma.

Indeed, we have:

\[
\begin{gathered}
\QCoh(\fg((t))dt) 
\underset{\QCoh(\LocSys_G(\o{\cD}))}{\otimes}
\QCoh(\fg((t))dt) = \\
\QCoh(G(K) \times \fg((t))dt)^{G(K),w}
\underset{\QCoh(\fg((t))dt)^{G(K),w}}{\otimes}
\QCoh(G(K) \times \fg((t))dt)^{G(K),w} \to \\
\bigg(
\QCoh(G(K) \times \fg((t))dt) 
\underset{\QCoh(\fg((t))dt)}{\otimes}
\QCoh(G(K) \times \fg((t))dt)
\bigg)^{G(K),w} = \\
\QCoh\Big((G(K) \times \fg((t)) dt )
\underset{\fg((t))dt}{\times}
(G(K) \times \fg((t))dt)\Big)^{G(K),w} = \\
\QCoh(\fg((t))dt \underset{\LocSys_G(\o{\cD})}{\times}
\fg((t))dt)
\end{gathered}
\]

\noindent so it suffices to see that taking weak $G(K)$-invariants
commutes with the tensor product here, i.e., that the displayed
arrow is an isomorphism.

We will do this using tameness, i.e., we claim that all of the
invariants above can also be computed as coinvariants.
Indeed, this is deduced as follows:

\begin{itemize}

\item 

For 
$\QCoh(G(K) \times \fg((t))dt) = \QCoh(G(K)) \otimes \QCoh(\fg((t))dt)$,
this is tautological.

\item 

For $\QCoh(\fg((t))dt)$, this follows from
Proposition \ref{p:zr-tame}.

\item 

For $\QCoh$ of the fiber product 
$(G(K) \times \fg((t)) dt )
\underset{\fg((t))dt}{\times}
(G(K) \times \fg((t))dt)$, note that we can rewrite this fiber
product as $G(K) \times G(K) \times \fg((t))dt$ so that the $G(K)$-action
is free (in the obvious sense), and then tameness is again tautological.

\end{itemize}

\end{proof}

\subsection{}

Now observe that both the left and right hand sides of 
\eqref{eq:finale-2} have $G(K) \times G(K)$-actions
(i.e., $\QCoh(G(K)\times G(K))$-actions) coming from the gauge
action of $G(K)$ on $\fg((t))dt$. 
We claim that
passing to $\widehat{G(O)} \times \widehat{G(O)}$-invariants
on both sides, we obtain the morphism \eqref{eq:finale-1}.
This would obviously show that \eqref{eq:finale-1} is an isomorphism,
completing the proof of the 1-affineness of
$\LocSys_G(\o{\cD})$.

Indeed, passing to these invariants on the right hand side, 
we tautologically obtain:

\[
\QCoh\Big(\fg((t))dt/\widehat{G(O)} 
\underset{\LocSys_G(\o{\cD})}{\times}
\fg((t))dt/\widehat{G(O)}\Big)
\]

\noindent i.e., the right hand side of \eqref{eq:finale-1}.

On the left hand side, the result follows from tameness. Indeed,
we have:

\[
\begin{gathered}
\QCoh(\fg((t))dt/\widehat{G(O)})
\underset{\QCoh(\LocSys_G(\o{\cD}))}{\otimes}
\QCoh(\fg((t))dt/\widehat{G(O)}) = \\
\QCoh(\fg((t))dt)^{\widehat{G(O)},w} 
\underset{\QCoh(\LocSys_G(\o{\cD}))}{\otimes}
\QCoh(\fg((t))dt)^{\widehat{G(O)},w} \to \\
\Big(\QCoh(\fg((t))dt)
\underset{\QCoh(\LocSys_G(\o{\cD}))}{\otimes}
\QCoh(\fg((t))dt)\Big)^{\widehat{G(O)}\times \widehat{G(O)},w}
\end{gathered}
\]

\noindent and we are trying to see that this arrow is an isomorphism.
If we can show that invariants everywhere (functorially, i.e.,
via the norm map) coincide
with coinvariants, then the result is clear.

For $\QCoh(\fg((t))dt)$, the result follows from Proposition 
\ref{p:zr-tame}.

By Proposition \ref{p:tame-indproper}, 
it remains to check $G(O)\times G(O)$-tameness 
for the tensor product appearing
on the right there. Note that by Lemma \ref{l:finale-gp},
the tensor product coincides with:

\[
\QCoh(\fg((t))dt \underset{\LocSys_G(\o{\cD})}{\times} \fg((t))dt) =
\QCoh(G(K) \times \fg((t))dt).
\]

\noindent Here one $G(O)$-action 
is on just acting on the left on the first factor, and the second
$G(O)$-action is diagonally induced by the right action on the
first factor and the gauge action on the second factor.
Obviously this category is tame with respect to just the
first $G(O)$-action, so it suffices to pass to the invariants and 
check that it is tame with respect to the residual $G(O)$-action.
I.e., we want to see that $\QCoh(G(O)\backslash G(K) \times \fg((t))dt)$
is tame for the diagonal $G(O)$-action. 

It suffices to prove this tameness
with $t^{-r}\fg[[t]]dt$ replacing $\fg((t))dt$, since passing to the
limit = colimit in $r$ we obtain the desired result. 
Then observe that 
$p_2: G(O)\backslash G(K) \times t^{-r}\fg[[t]]dt \to t^{-r}\fg[[t]]dt$ is
$G(O)$-equivariant for the diagonal $G(O)$-action, so that
$\QCoh(G(O) \backslash G(K) \times t^{-r}\fg[[t]]dt)$ is acted
on compatibly by $\QCoh(G(O))$ and by $\QCoh(t^{-r}\fg[[t]]dt)$.
Therefore, the result follows from Theorem \ref{t:tame}, i.e.,
tameness of $t^{-r}\fg[[t]]dt$.

\bibliography{bibtex}{}
\bibliographystyle{alphanum}

\end{document}